\documentclass{article}
\usepackage{fontenc}
\usepackage{amsmath}
\usepackage{amssymb}
\usepackage{amsfonts}
\usepackage{amsthm}

\newtheorem{thm}{Theorem}[section]
\newtheorem{prop}[thm]{Proposition}
\newtheorem{cor}[thm]{Corollary}

\newtheorem{lem}[thm]{Lemma}
\newtheorem{conj}[thm]{Conjecture}

\begin{document}

\title{Thomae Formulae for General Fully Ramified $Z_{n}$ Curves}

\author{Shaul Zemel\thanks{This work was supported by the Minerva Fellowship
(Max-Planck-Gesellschaft).}}

\maketitle

\section*{Introduction}

The original Thomae formula is an assertion relating the theta constants on a
hyper-elliptic Riemann surface $X$, presented as a double cover of
$\mathbb{P}^{1}(\mathbb{C})$, to certain polynomials in the
$\mathbb{P}^{1}(\mathbb{C})$-values of the fixed points of the hyper-elliptic
involution on $X$. They were initially derived by Thomae in the 19th century
(see \cite{[T1]} and \cite{[T2]}). After laying dormant for more than a 100
years, these formulae returned to active research, mainly due to the interest of
the mathematical physics community. The first generalization of these formulae
appears in \cite{[BR]}, who considered non-singular $Z_{n}$ curves, i.e., those
compact Riemann surfaces which are associated with an equation of the type
$w^{n}=\prod_{i=1}^{nr}(z-\lambda_{i})$ for some $r\geq1$, with
$\lambda_{i}\neq\lambda_{j}$ for $i \neq j$. The proof for this case was
simplified by \cite{[N]}, using the Szeg\H{o} kernel function. A family of
singular $Z_{n}$ curves has been treated in \cite{[EG]}. A very elementary proof
for the original formulae was found in \cite{[EiF]}, where the case of
non-singular $Z_{3}$ curves was also seen to be covered by these elementary
techniques. The idea is that certain quotients of powers of theta functions can
be identified as simple meromorphic functions on the Riemann surface under
consideration. These techniques were extended to arbitrary non-singular $Z_{n}$
curves in \cite{[EbF]}. The book \cite{[FZ]} presents in detail the proof of the
Thomae formulae for several families of $Z_{n}$ curves, namely the non-singular
ones, the families treated in \cite{[EG]}, and two other smaller families. The
proofs in this book follow the elementary methods of \cite{[EiF]} and
\cite{[EbF]}. On the other hand, formulae for the general case have been
obtained in \cite{[K]}, again using the Szeg\H{o} kernel function. Additional
results on the Thomae formulae for $Z_{3}$ curves are presented in \cite{[M]}
and \cite{[MT]}. This paper should be considered as a completion, in some sense,
of \cite{[EbF]} and \cite{[FZ]}. Our methods resemble a bit the ones used in
\cite{[M]} (identifying theta quotients as meromorphic functions on the $Z_{n}$
curve by knowing their divisors), and are more elementary than those used in
\cite{[K]}. 

An important step in the proof of the Thomae formulae in all the cases
considered in \cite{[FZ]} is the construction of non-special divisors of degree
$g$ on $Z_{n}$ curves, which are supported on the branch points on the $Z_{n}$
curve. A characterization of these divisors in the case of prime $n$ is
presented in \cite{[GD]}, using certain sums of residues modulo $n$. The first
result of the present paper (Theorem \ref{nonsp}) is a characterization of
non-special divisors on fully ramified $Z_{n}$ curves for arbitrary $n$ in terms
of cardinalities of certain sets which are based on the divisor in question. In
our method, it is easier to work with the cardinalities than with the residues,
and our result is equivalent to that of \cite{[GD]} in the case of prime $n$.
Next, certain operators defined in \cite{[EbF]} and \cite{[FZ]} are useful in
the derivations. We show how to define these operators for general Riemann
surfaces, and provide the formula to evaluate them in the case of a fully
ramified $Z_{n}$ curve and a divisor which is supported on the branch points
(see Theorem \ref{NQTQR} and Proposition \ref{operZn}). We remark that a result
of \cite{[GDT]} yields $Z_{n}$ curves with no non-special divisors of the sort
required for us. However, it turns out that the Thomae formulae which we obtain,
at least in the form presented here, are independent of the cardinality
conditions required for non-specialty. Therefore they can be trivially extended
to the cases considered in \cite{[GDT]}.

We now indicate how the construction and proof of the Thomae formulae are
established in this paper. The process follows \cite{[FZ]}, as well as
\cite{[EiF]} and \cite{[EbF]}. We begin by identifying quotients of powers of
theta functions with characteristics as meromorphic functions on the $Z_{n}$
curve in order to derive relations between pairs of theta constants whose
characteristics are related by operators of the form $T_{Q,R}$, where $Q$ is the
branch point we choose as the base point. Next we obtain, for every type of
points (i.e., the power to which the point appears in the $Z_{n}$ equation
defining the $Z_{n}$ curve), a quotient which is invariant under all the
operator $T_{Q,R}$ with our base point $Q$ and $R$ of the chosen type. A simple
correction of the resulting denominator yields, for every base point $Q$, a
quotient which is invariant under the operators $T_{Q,R}$ for all admissible
points $R$. This quotient is called, following \cite{[FZ]}, the \emph{Poor Man's
Thomae}, or PMT for short. The next step, which is technically more difficult,
is to obtain a denominator for which the quotient is invariant also under the
negation operator $N_{Q}$ (or $N_{\beta}$). A ``base point change operator'' $M$
is also introduced, and the quotient is invariant also under $M$. If these
operators act transitively on the set of divisors under consideration, a fact
which we prove in several cases and should hold in general, then the quotient
$\frac{\theta^{2en^{2}}[\Xi](0,\Pi)}{h_{\Xi}}$ thus obtained is independent of
the divisor $\Xi$. This is the Thomae formulae for the $Z_{n}$ curve.

The resulting Thomae formulae are based on certain integral-valued functions,
denoted here $f_{\beta,\alpha}$ and later $f^{(n)}_{d}$. We investigate the
propeties of these functions, which have an interesting recursive definition
(see Theorem \ref{fndrec}). We remark again that \cite{[K]} obtained
expressions for the Thomae formulae which include sums of certain fractional
parts. Comparing with our results gives a tool for evaluating these sums, and it
seems that our recursive relation yields a more efficient way to obtain the
actual values of these powers in every particular case.

Relations between theta constants (such as the Thomae formulae) may have
several possible applications. On example is to the Schottky problem: The
moduli spaces $\mathcal{A}_{g}$ and $\mathcal{M}_{g}$ of principally polarized
Abelian varieties of dimension $g$ and Jacobians of curves of genus $g$
respectively carry the structure of algebraic varieties over $\mathbb{C}$, of
respective dimensions $\frac{g(g+1)}{2}$ and $3g-3$ (or 0 for $g=0$ and 1 for
$g=1$), such that the latter is a subvariety of the former. Torelli's Theorem
shows that $\mathcal{M}_{g}$ can equivalently be seen as the moduli space of
curves of genus $g$. The Schottky problem is to find (for $g\geq4$, where the
dimensions differ) the algebraic equation which characterize (up to fine
details) $\mathcal{M}_{g}$ inside $\mathcal{A}_{g}$. Variants of this problem
is to find equations determining the locus of Riemann surfaces with additional
structure (e.g., existence of certain automorphisms, such a a
$Z_{n}$-structure) inside either $\mathcal{M}_{g}$ or $\mathcal{A}_{g}$. The
last part of \cite{[A1]} does this for some order 2 automorphisms (with
quotients of arbitrary genus). The paper \cite{[A2]} is concerned with this
task for $Z_{3}$ curves in $\mathcal{M}_{g}$. Although both these references
use (positive) orders of vanishing of theta constants, and the Thomae formulae
require non-vanishing theta constants, it may be possible that such identities
may also be of use for the Schottky problem. In addition, the modularity of
theta constants under the integral symplectic group $Sp_{2g}(\mathbb{Z})$ (see,
e.g., Fact 5.1 of \cite{[M]} for the explicit formula for this modular behavior)
suggests that the Thomae constant (normalized appropriately) may be presented as
a generalized version of a modular form on the moduli space of $Z_{n}$ curves.
Indeed, if the normalization of the Thomae formulae, which is discussed further
in Section \ref{Open}, exists, then the Thomae constant would become a function
on the moduli space of $Z_{n}$ curves considered as Riemann surfaces with a
cyclic group of automorphisms (and some finer ramification data). This may
produce analogues of the modular functions $j$ (living on the compactification
of the quotient of the upper half-plane $\mathbb{H}$ by the action of the
modular group $SL_{2}(\mathbb{Z})$) or $\lambda$ (which is defined on a normal
cover of degree 6 of the latter space), which are defined on higher-dimensional
(moduli) spaces.

Another possible application is for the uniformization problem for spaces
appearing in the Deligne-Mostow list of spaces which arise as quotients of
complex balls by discrete automorphism groups (and their compactifications).
This can be seen, for example, in \cite{[M]}, where the moduli space of
non-singular $Z_{3}$ curves of genus 4 is given by such a ball quotient, and the
parameters defining the $Z_{3}$ curve are given by quotients of theta functions,
like an inverse Thomae formula. The general theory (see \cite{[DM]}) started
with investigating the monodromy of hyper-geometric functions, and it relates
to certain integrals (called Lauricella functions in \cite{[L]}) of
multi-valued differentials involving some parameters $\mu_{k}\in(0,1)$. When
these $\mu_{k}$ are rational numbers, the differential becomes a holomorphic
differential on a $Z_{n}$ curve for some $n$, and relating the Thomae constant
(in the case of full ramification) to these integrals may yield some kind of
uniformization for the discrete groups operating on the complex balls which
arise from this theory as in \cite{[L]}. 

The theory of principally polarized Abelian varieties, curves and their
Jacobians, and theta functions, can be presented in an algebro-geometric
approach, which allows us to consider these notions over fields other than
$\mathbb{C}$. It it possible that our proof of the Thomae formulae may be
formulated in such a way to extend to this more general setting: In fact, once
Proposition \ref{thetaquot} is established, the rest is just elementary algebra
of expressions involving products of values of branch points. In case the Thomae
constant has such an algebraic meaning, it may relate to constructions which are
carried out over, e.g., $p$-adic fields. It thus may be of use for applications
of such constructions, e.g., Mestre's point counting the arithmetic geometric
mean relates to certain theta constants on Jacobians of curves over unramified
extensions of $\mathbb{Q}_{p}$. In case the $p$-adic Thomae constant (if it
exists) is related to the theta constants from Mestre's algorithm (and it's
generalizations), it may be possible that the Thomae formulae increase the
efficiency of the computations involved in this algorithm.

However, this paper is concerned only in establishing the Thomae formulae
itself (over $\mathbb{C}$). It is organized as follows. In Section \ref{ZnDiv}
we describe the fully ramified $Z_{n}$ curves, and characterize the non-special
divisors of degree $g$ supported on their branch points using the cardinality
conditions. Section \ref{Operators} defines the operators whose action is
necessary for establishing the Thomae formulae. The first relations are obtained
in Section \ref{PMT}, and the manipulations required in order to achieve the PMT
are also performed. The quotient which is invariant under the negation operator
is given in Section \ref{NInv}, where certain properties of the functions
$f_{\beta,\alpha}$ appearing in this quotient are also proved. Section
\ref{Trans} introduces the base point change operator $M$, proves some partial
results about transitivity, and states the final Thomae formulae. The recursive
relation which is required in order to evaluate the functions $f_{\beta,\alpha}$
is proved in Section \ref{ExpForm}, where explicit expressions for these
functions are given in a few cases. Finally, in Section \ref{Open} we discuss
some remaining open questions.

Many thanks are due to the referee, whose valuable suggestions made the
presentation of this paper more transparent, and its applicability to several
mathematical problems more visible.

\section{Non-Special Divisors on $Z_{n}$ Curves \label{ZnDiv}}

Let $X$ be a $Z_{n}$ curve, namely a cyclic cover of order $n$ of
$\mathbb{P}^{1}(\mathbb{C})$, with projection map $z$. We assume $n>1$
throughout, since a $Z_{1}$ curve is just $\mathbb{P}^{1}(\mathbb{C})$ with a
chosen isomorphism $z$. Any such curve can be presented as the Riemann surface
associated with an equation of the sort $w^{n}=f(z)$ (called a
\emph{$Z_{n}$-equation}) for some meromorphic function $f\in\mathbb{C}(z)$ which
is not a $d$th power in $\mathbb{C}(z)$ for any $d$ dividing $n$ (so that the
equation is irreducible). We call such a definining equation for a $Z_{n}$ curve
\emph{normalized} if $f$ is a monic polynomial which has no roots of order $n$
or more. Every $Z_{n}$-equation can be made normalized by multiplying $w$ by an
appropriate function of $z$, and this function is unique as long as we keep $w$
in a fixed component under the action of the cyclic Galois group (see the end of
this paragraph). The field $\mathbb{C}(X)$ of meromorphic functions on $X$
decomposes as $\bigoplus_{r=0}^{n-1}\mathbb{C}(z)w^{r}$, or equivalently
$\bigoplus_{k=0}^{n-1}\mathbb{C}(z)\cdot\frac{1}{w^{k}}$. The space of
meromorphic differentials on $X$ is 1-dimensional over $\mathbb{C}(X)$, and is
spanned by any non-zero differential on $X$. By choosing the differential to be
$dz$ we obtain that this space decomposes as
$\bigoplus_{k=0}^{n-1}\mathbb{C}(z)\frac{dz}{w^{k}}$. The cyclic Galois group
of the map $z:X\to\mathbb{P}^{1}(\mathbb{C})$ acts on $\mathbb{C}(X)$ and on
the space of meromorphic differentials on $X$, and the decompositions given here
are precisely the decompositions according to the action of this Galois group:
Indeed, this group consists of automorphisms of the form
$\tau_{\zeta}:(z,w)\mapsto(z,\zeta w)$ with $\zeta$ is taken from the group
$\text{\boldmath$\mu$}_{n}$ of roots of unity of order dividing $n$, and the
$k$th component consists of those functions or differentials which are
multiplied by $\zeta^{n-k}$ under the action of each such automorphism.

Let $\varphi:X \to Y$ be a non-constant holomorphic map between compact Riemann
surfaces, which has degree $r$. Then the inverse image of any point $y \in Y$
consists of $r$ points, counted with multiplicities. Those points $x \in X$
having multiplicities $b_{x}+1>1$ are called \emph{branch points} (and there
are finitely many of those), and in this case the (positive) integer $b_{x}$ is
called the \emph{branching number} of $x$. The degree and branching numbers
appear in the \emph{Riemann--Hurwitz formula}, relating the genus $g_{X}$ of $X$
with the genus $g_{Y}$ of $Y$ according to the equality
\[2g_{X}-2=r(2g_{Y}-2)+\sum_{x \in X}b_{x}.\]

We will consider only the case where $X$ is a $Z_{n}$ curve,
$Y=\mathbb{P}^{1}(\mathbb{C})$, and $\varphi$ is the meromorphic function $z$
(which has degree $n$ since the equation is irreducible). Since the genus of
$Y$ is 0, we write simply $g$ for $g_{X}$, with no confusion arising. Now,
assume that $X$ is given through a $Z_{n}$-equation $w^{n}=f(z)$. For every
point $\lambda\in\mathbb{P}^{1}(\mathbb{C})$, we let $d$ be the greatest common
divisor of $n$ and $ord_{\lambda}(f)$. For each such $\lambda$ there are $d$
points of $X$ lying over $\lambda\in\mathbb{P}^{1}(\mathbb{C})$, each of which
has ramification index $\frac{n}{d}-1$. In particular, no point outside the
divisor of $f$ is a branch point. Let $u$ be another meromorphic function on
$X$, which generates $\mathbb{C}(X)$ over $\mathbb{C}(z)$ and spans a complex
vector space which is invariant under the Galois group of $z$. Replacing $w$ by
$u$ leaves the set of branch points, as well as their ramification indices,
invariant: Indeed, multiplying $w$ by an element of $\mathbb{C}(z)$ changes the
order of $f$ at $\lambda\in\mathbb{P}^{1}(\mathbb{C})$ by a multiple of $n$, and
replacing $w$ by $w^{k}$ for $k\in\mathbb{Z}$ which is prime to $n$ multiplies
these orders by $k$ and leaves the greatest common divisors invariant. Thus we
may assume that the $Z_{n}$-equation is normalized. Then the branch points on
which $z$ is finite lie only over the roots of $f$ (and over all of them).
Points lying over $\infty$ are branch points if and only if the degree of $f$ is
not divisible by $n$.

\medskip

Following \cite{[FZ]}, we call a $Z_{n}$ curve $X$ \emph{fully ramified} if any
branch point on $X$ has maximal ramification index (namely $n-1$). Equivalently,
the $Z_{n}$ curve is fully ramified if for
$\lambda\in\mathbb{P}^{1}(\mathbb{C})$, $z^{-1}(\lambda)$ consists either of $n$
points or of a unique branch point. Given a normalized $Z_{n}$-equation defining
the $Z_{n}$ curve $X$, this property is equivalent to all the roots of $f$
appearing with orders which are prime to $n$, and the degree of $f$ is either
divisible by $n$ or also prime to $n$.

\smallskip

Let $\Delta$ be a divisor on $X$ (not necessarily integral). We recall from
\cite{[FK]} (or \cite{[FZ]}) that the space of meromorphic functions whose
divisor is at least $\Delta$, denoted $L(\Delta)$, is finite-dimensional, of
dimension denoted $r(\Delta)$. Similarly, the space of meromorphic
differentials with this property, denoted $\Omega(\Delta)$, is also
finite-dimensional, with $i(\Delta)$ denoting this dimension. These numbers are
related by the \emph{Riemann--Roch Theorem}, stating that
$r\big(\frac{1}{\Delta}\big)=\deg\Delta+1-g+i(\Delta)$. Using the decomposition
of $\mathbb{C}(X)$ and $\mathbb{C}(X)dz$ from above, we denote $r_{k}(\Delta)$
the dimension of the space of meromorphic functions of the form
$\frac{p(z)}{w^{k}}$ (with $p\in\mathbb{C}(z)$) which lie in $L(\Delta)$, and
$i_{k}(\Delta)$ denotes the dimension of the space of meromorphic differentials
of the form $\varphi(z)\frac{dz}{w^{k}}$ lying in $\Omega(\Delta)$. The
inequalities $\sum_{k=0}^{n-1}r_{k}(\Delta) \leq r(\Delta)$ and
$\sum_{k=0}^{n-1}i_{k}(\Delta) \leq i(\Delta)$ are clear. On the other hand, for
divisors supported on the branch points of a fully ramified $Z_{n}$ curve we
have the following generalization of Lemma 2.7 of \cite{[FZ]}:

\begin{prop}
Let $\Delta$ be a divisor on $X$ (not necessarily integral) which is based on
the branch points of $z$, and let $h$ and $\omega$ be a meromorphic function and
a meromorphic differential on $X$ respectively. Decompose $h$ and $\omega$ as
$\sum_{k=0}^{n-1}h_{k}$ with $h_{k}\in\mathbb{C}(z)\cdot\frac{1}{w^{k}}$ and
$\sum_{k=0}^{n-1}\omega_{k}$ with $\omega_{k}\in\mathbb{C}(z)\frac{dz}{w^{k}}$
respectively. Then $h \in L(\Delta)$ (resp. $\omega\in\Omega(\Delta)$) if and
only if the same assertion holds for $h_{k}$ (resp. $\omega_{k}$) for all $0
\leq k \leq n-1$. In particular $r(\Delta)=\sum_{k=0}^{n-1}r_{k}(\Delta)$ and
$i(\Delta)=\sum_{k=0}^{n-1}i_{k}(\Delta)$. \label{decomri}
\end{prop}

\begin{proof}
One way to prove these assertions (as is done in the special cases appearing in
\cite{[FZ]}) is using the fact that the functions $h_{k}$ as well as the
differentials $\omega_{k}$ must have distinct orders at the branch points, even
modulo $n$. However, we shall be using the action of the cyclic Galois group.

If $\tau$ is an automorphism of a compact Riemann surface, $\Delta$ is a
divisor on this Riemann surface, $h$ is a meromorphic function lying in
$L(\Delta)$ and $\omega$ is a meromorphic differential lying in
$\Omega(\Delta)$, then $\tau(h)$ lies in $L\big(\tau(\Delta)\big)$ and
$\tau(\omega)$ belongs to $\Omega\big(\tau(\Delta)\big)$. Now, a divisor
$\Delta$ which is supported on the branch points of a fully ramified $Z_{n}$
curve is invariant under the Galois group of $z$ (here we need the full
ramification, implying that over any branching value of $z$ lies only one point
of $X$). Hence the action of this cyclic Galois group preserves the spaces
$L(\Delta)$ and $\Omega(\Delta)$. The fact that
$\sum_{\zeta\in\text{\boldmath$\mu$}_{n}}\zeta^{l}$ equals $n$ is $n$ divides
$l$ and 0 otherwise implies that
$h_{k}=\frac{1}{n}\sum_{\zeta\in\text{\boldmath$\mu$}_{n}}\zeta^{k}\tau_{\zeta}
(h)$ and
$\omega_{k}=\frac{1}{n}\sum_{\zeta\in\text{\boldmath$\mu$}_{n}}\zeta^{k}\tau_{
\zeta}(\omega)$. It follows that $h_{k} \in L(\Delta)$ and
$\omega_{k}\in\Omega(\Delta)$. The equalities involving $r(\Delta)$ and
$i(\Delta)$ follow directly from the previous assertions. This proves the
proposition.
\end{proof}

We remark that the proof of Proposition \ref{decomri} shows that the same
assertion holds for any divisor on a $Z_{n}$ curve (not necessarily fully
ramified), provided that the divisor is invariant under the cyclic Galois group
of the projection map $z$. However, we consider only divisors supported on
branch points on fully ramified $Z_{n}$ curves in this paper.

\smallskip

The set of integers between 0 and $n-1$ (inclusive) will play a prominent role
in this paper. Hence we denote it $\mathbb{N}_{n}$. The set $\mathbb{N}_{n}$ is
also a good set of representatives of $\mathbb{Z}/n\mathbb{Z}$ in $\mathbb{Z}$.

\smallskip

We now turn to bases for the holomorphic differentials on a fully ramified
$Z_{n}$ curve. For simplicity and symmetry, we shall assume throughout that
there is no branching over $\infty$ (this can always be obtained by composing
$z$ with an automorphism of $\mathbb{P}^{1}(\mathbb{C})$). We thus write the
normalized $Z_{n}$-equation defining $X$ as
\begin{equation}
w^{n}=\prod_{\alpha}\prod_{i=1}^{r_{\alpha}}(z-\lambda_{\alpha,i})^{\alpha},
\label{Zneq}
\end{equation}
where $\alpha$ runs over the set of numbers in $\mathbb{N}_{n}$ which are prime
to $n$. The assumption that no branch point lies over $\infty$ is equivalent to
the assertion that $n$ divides $\sum_{\alpha}\alpha r_{\alpha}$. The genus $g$
of $X$ equals $(n-1)\big(\sum_{\alpha}r_{\alpha}-2\big)/2$ by the
Riemann--Hurwitz formula. This number is always an integer: This is clear if $n$
is odd, and if $n$ is even then so is $\sum_{\alpha}\alpha r_{\alpha}$, and
since we take only odd $\alpha$, the same assertion holds for
$\sum_{\alpha}r_{\alpha}$. Proposition \ref{decomri} implies that we can
decompose the space $\Omega(1)$ of holomorphic differentials on $X$ as
$\bigoplus_{k=0}^{n-1}\Omega_{k}(1)$. Now, $\Omega_{0}(1)$ is the space of
holomorphic differentials in $\mathbb{C}(z)dz$, i.e., holomorphic differentials
which are pullbacks of holomorphic differentials on
$\mathbb{P}^{1}(\mathbb{C})$. As there are no such differentials, the
decomposition is in fact $\bigoplus_{k=1}^{n-1}\Omega_{k}(1)$. It follows that
for an \emph{integral} divisor $\Delta$, Proposition \ref{decomri} yields
$i(\Delta)=\sum_{k=1}^{n-1}i_{k}(\Delta)$ (since $i_{0}(\Delta)=0$).

We shall denote, here and throughout, the poles of $z$ on $X$ by $\infty_{h}$,
$1 \leq h \leq n$. The (unique) branch point on $X$ lying over
$\lambda_{\alpha,i}$ will be denoted $P_{\alpha,i}$. Then
\[div(z-\lambda_{\alpha,i})=\frac{P_{\alpha,i}^{n}}{\prod_{h=1}^{n}\infty_{h}},
\ \ div(w)=\frac{\prod_{\alpha}\prod_{i=1}^{r_{\alpha}}P_{\alpha,i}^{\alpha}}{
\prod_{h=1}^{n}\infty_{h}^{t_{1}}},\ \
div(dz)=\frac{\prod_{\alpha}\prod_{i=1}^{r_{\alpha}}P_{\alpha,i}^{n-1}}{\prod_{
h=1}^{n}\infty_{h}^{2}}\] where
$div$ denotes the divisor of a meromorphic function or differential and
$nt_{1}=\sum_{\alpha}\alpha r_{\alpha}$. We now introduce a convenient basis for
the space $\mathbb{C}(X)dz$ over $\mathbb{C}(z)$. For any $k\in\mathbb{Z}$ and
any $\alpha$ we define $s_{\alpha,k}=\big\lfloor\frac{\alpha k}{n}\big\rfloor$,
where for a real number $x$ the symbol $\lfloor x \rfloor$ stands for the
integral value of $x$, namely the maximal integer $m$ satisfying $m \leq x$.
Then $s_{\alpha,k}$ satisfies $\frac{\alpha k+1-n}{n} \leq
s_{\alpha,k}\leq\frac{\alpha k}{n}$. Moreover, the number $\alpha
k-ns_{\alpha,k}$, which lies in $\mathbb{N}_{n}$, depends only on the class of
$k$ modulo $n$. Let
$\omega_{k}=\prod_{\alpha}\prod_{i=1}^{r_{\alpha}}(z-\lambda_{\alpha,i})^{s_{
\alpha,k}}\frac{dz}{w^{k}}$. This differential is well-defined for
$k\in\mathbb{Z}/n\mathbb{Z}$, though we usually assume $k\in\mathbb{N}_{n}$. We
remark that for $k$ prime to $n$, the denominator under $dz$ in $\omega_{k}$
corresponds to the normalized $Z_{n}$-equation describing $X$ with $w^{k}$.
Moreover, if the greatest common divisor of $k$ and $n$ is $d$ then this
denominator corresponds to the normalized $Z_{n/d}$-equation describing the
quotient of $X$ by the subgroup of order $d$ of the Galois group, which is a
$Z_{n/d}$ curve. We evaluate
\begin{equation}
div(\omega_{k})=\prod_{\alpha}\prod_{i=1}^{r_{\alpha}}P_{\alpha,i}^{n-1+ns_{
\alpha,k}-\alpha k}\prod_{h=1}^{n}\infty_{h}^{t_{k}-2},\quad
t_{k}=\sum_{\alpha}r_{\alpha} \bigg(\frac{\alpha k}{n}-s_{\alpha,k}\bigg).
\label{divomegak}
\end{equation}
The numbers $ns_{\alpha,k}-\alpha k-1+n$ lie also in $\mathbb{N}_{n}$, and
observe that $t_{k}\in\mathbb{Z}$ (since $n|\sum_{\alpha}r_{\alpha}$).
Moreover, $t_{k}$ vanishes if $n|k$ and is positive for every $k$ not divisible
by $n$ (or $1 \leq k \leq n-1$), since all the summands are positive (recall
that we take only $\alpha$ which is prime to $n$). In particular $\omega_{0}=dz$
with the divisor written above, and $\omega_{1}=\frac{dz}{w}$ since
$s_{\alpha,1}=0$ for all $\alpha$. Hence the two formulae for $t_{1}$ coincide.
For every $d\in\mathbb{Z}$ with $d\geq-1$ we define the space $\mathcal{P}_{\leq
d}(z)$ of polynomials in $z$ of degree not exceeding $d$ (so that
$\mathcal{P}_{\leq0}(z)$ is the space of constant polynomials and
$\mathcal{P}_{\leq-1}(z)=\{0\}$). The dimension of $\mathcal{P}_{\leq d}(z)$ is
$d+1$ (also for $d=-1$). We therefore obtain

\begin{prop}
The space $\Omega(1)$ of holomorphic differentials on $X$ decomposes as
$\bigoplus_{k=1}^{n-1}\mathcal{P}_{\leq t_{k}-2}(z)\omega_{k}$. \label{basdif}
\end{prop}

\begin{proof}
Proposition \ref{decomri} and the paragraph below Equation \eqref{Zneq} show
that $\Omega(1)$ decomposes as $\bigoplus_{k=1}^{n-1}\Omega_{k}(1)$. It
therefore suffices to show that a differential in
$\mathbb{C}(z)\frac{dz}{w^{k}}$, or equivalently
$\mathbb{C}(z)\omega_{k}$, is holomorphic if and only if it lies in
$\mathcal{P}_{\leq t_{k}-2}(z)\omega_{k}$. Let $\varphi\in\mathbb{C}(z)$ and
assume that $\varphi(z)\omega_{k}$ is holomorphic. The divisor of $\omega_{k}$
is supported only on the branch points and poles of $z$, and the former points
appear to non-negative powers which are smaller than $n$ in this divisor. It
follows that $\varphi$ cannot have any pole in $\mathbb{C}$, hence it must be a
polynomial of some degree $d$. But then the order of $\varphi(z)\omega_{k}$ at
any point $\infty_{h}$ is $t_{k}-2-d$, so that $\varphi(z)\omega_{k}$ is
holomorphic precisely when $d \leq t_{k}-2$ (and this implies $\varphi=0$, i.e.,
there exists no such holomorphic differential, if $t_{k}=1$). This proves the
proposition.
\end{proof}

When we evaluate
$\sum_{k=1}^{n-1}t_{k}=\sum_{\alpha,k}r_{\alpha}\big(\frac{\alpha
k}{n}-\big\lfloor\frac{\alpha k}{n}\big\rfloor\big)$, we observe that for any
$\alpha$ which is prime to $n$ the set of numbers $\big\{\frac{\alpha
k}{n}-\big\lfloor\frac{\alpha k}{n}\big\rfloor\big\}_{k=1}^{n-1}$ (or
equivalently $\big\{\frac{\alpha k-ns_{\alpha,k}}{n}\big\}_{k=1}^{n-1}$) is
precisely the set $\big\{\frac{l}{n}\big\}_{l=1}^{n-1}$. Hence the sum
$\sum_{k=1}^{n-1}(t_{k}-1)$ of the dimensions of these spaces is indeed
$\sum_{\alpha,k}r_{\alpha}\frac{n(n-1)}{2n}-(n-1)=g$, as required.

Proposition \ref{basdif}, together with the fact (mentioned in the proof of
Proposition \ref{decomri}) that elements of the different eigenspaces of the
Galois group must have different orders at the branch points, implies that the
set
$\bigcup_{k=1}^{n-1}\{(z-\lambda_{\alpha,i})^{l}\omega_{k}\}_{l=0}^{t_{k}-2}$ is
a basis for $\Omega(1)$ which is adapted to the point $P_{\alpha,i}$ for any
$\alpha$ and $i$. The gap sequence at $P_{\alpha,i}$ can be read from this
basis. However, it is not needed for finding non-special divisors or for the
proof of the Thomae formulae. Moreover, in some of the examples in \cite{[FZ]}
some of the points $P_{i}$ have the usual gap sequence and are not Weierstrass
points. For these reasons we do not pursue this subject further in this work.

Using the notation $|Y|$ for the cardinality of the finite set $Y$, we prove

\begin{cor}
Let $\Delta$ be an integral divisor on $X$ which is supported on the branch
points, and assume that no branch point appears in $\Delta$ to a power $n$ or
higher. For any $\alpha$ and any $1 \leq k \leq n-1$ denote $A_{\alpha,k}$ the
set of indices $1 \leq i \leq r_{\alpha}$ such that $P_{\alpha,i}$ appears in
$\Delta$ to a power larger than $ns_{\alpha,k}-\alpha k-1+n$. Then
$i_{k}(\Delta)=\max\{t_{k}-1-\sum_{\alpha}|A_{\alpha,k}|,0\}$ and $i(\Delta)$ is
the sum of these numbers. \label{ikDelta}
\end{cor}

\begin{proof}
By Proposition \ref{basdif}, $\Omega_{k}(1)$ is a space of differentials of the
form $p(z)\omega_{k}$, where $p$ is a polynomial of degree not exceeding
$t_{k}-2$. We claim that $\Omega_{k}(\Delta)$ consists of those differentials in
which $p$ vanishes at all the values $\lambda_{\alpha,i}$ with $i \in
A_{\alpha,k}$. Indeed, since no branch point appears in $\Delta$ to a power $n$
or higher, if $p(\lambda_{\alpha,i})=0$ then $ord_{P_{\alpha,i}}(p(z)\omega_{k})
\geq n>ord_{P_{\alpha,i}}(\Delta)$. Hence simple zeroes of $p$ suffice. For an
index $i$ not lying in $A_{\alpha,k}$ we have $ord_{P_{\alpha,i}}(\omega_{k})
\geq ord_{P_{\alpha,i}}(\Delta)$, and multiplying by any polynomial in $z$ can
only increase the order of the differential at $P_{\alpha,i}$. On the other
hand, if $i \in A_{\alpha,k}$ then $p$ must vanish at $\lambda_{\alpha,i}$ in
order for $ord_{P_{\alpha,i}}(\varphi(z)\omega_{k})$ to reach
$ord_{P_{\alpha,i}}(\Delta)$. This shows that $\Omega_{k}(\Delta)$ is indeed
the asserted space. Since the conditions $p(\lambda_{\alpha,i})=0$ are linearly
independent (unless we reach the 0 space), the assertion about $i_{k}(\Delta)$
follows. The assertion about $i(\Delta)$ is now a consequence of Proposition
\ref{decomri}. This proves the corollary.
\end{proof}

\smallskip

The definition of $A_{\alpha,k}$ extends to arbitrary $k\in\mathbb{Z}$ by
considering the image of $k$ in $\mathbb{Z}/n\mathbb{Z}$. For $k$ divisible by
$n$ all the sets $A_{\alpha,k}$ are empty, and $i_{0}(\Delta)=0$ since
$t_{0}=0$.

\smallskip

The following argument has been used in several special cases in \cite{[FZ]}. We
include it here since it is simple, short, and general. Recall that an integral
divisor of degree $g$ on a Riemann surface of genus $g$ is called \emph{special}
if $i(\Delta)>0$, and is called \emph{non-special} otherwise.

\begin{lem}
Any integral divisor $\Delta$ of degree $g$ on a fully ramified $Z_{n}$ curve
$X$ containing a branch point to power $n$ or higher is special. \label{P_i^n}
\end{lem}

\begin{proof}
Let $Q$ be a branch point on $X$. Apart from the constant functions, whose
divisor is a multiple of $1/\Delta$ by the integrality of $\Delta$, the space
$L(1/Q^{n})$ contains the meromorphic function $\frac{1}{z-z(P)}$ if
$z(Q)\in\mathbb{C}$ or the function $z$ if $z(Q)=\infty$ (by full
ramification). It follows that $r(1/Q^{n})\geq2$, hence $r(1/\Delta)\geq2$ if
$Q^{n}$ divides $\Delta$. But then the Riemann--Roch Theorem implies
$i(\Delta)\geq1$ (since the degree of $\Delta$ is $g$), hence $\Delta$ is
special.
\end{proof}

We will be interested in non-special divisors supported on the branch points on
a fully ramified $Z_{n}$ curve. In the quest for such divisors, Lemma
\ref{P_i^n} allows us to restrict attention to divisors in which the branch
points appear only to powers at most $n-1$, without losing possibilities. Every
such divisor $\Delta$ can be written as
\begin{equation}
\Delta=\prod_{\alpha}\prod_{l=0}^{n-1}C_{\alpha,l}^{n-1-l}, \label{Delta}
\end{equation}
where for every $\alpha$ the sets $C_{\alpha,l}$, $l\in\mathbb{N}_{n}$, form a
partition of the set of points $\{P_{\alpha,i}\}_{i=1}^{r_{\alpha}}$, and the
power of a (finite) set means the product of the points lying in this set,
raised to the given power.

\smallskip

We can now characterize the non-special divisors of degree $g$ supported on the
branch points on $X$ by appropriate cardinality conditions.

\begin{thm}
Let $\Delta$ be an integral divisor of degree $g$ which is supported on the
branch points on $X$. Then $\Delta$ is non-special if and only if it can be
written as in Equation \eqref{Delta} and the cardinalities of the sets
$C_{\alpha,l}$ satisfy the equality \[\sum_{\alpha}\sum_{l=0}^{\alpha
k-ns_{\alpha,k}-1}|C_{\alpha,l}|=t_{k}-1\] for every $1 \leq k \leq n-1$.
\label{nonsp}
\end{thm}

\begin{proof}
Lemma \ref{P_i^n} allows us to restrict attention to divisors $\Delta$ in which
the branch points appear to powers not exceeding $n-1$. We can thus define, for
every $l\in\mathbb{N}_{n}$, the set $C_{\alpha,l}$ to contain those branch
points $P_{\alpha,i}$ appearing to the power $n-1-l$ in $\Delta$. Every point
$P_{\alpha,i}$ must lie in some set $C_{\alpha,l}$. Thus $\Delta$ takes the form
given in Equation \eqref{Delta}, and the sets $C_{\alpha,l}$ form the required
partitions. As $l\leq\alpha k-ns_{\alpha,k}-1$ is equivalent to
$n-1-l>n-1+ns_{\alpha,k}-\alpha k$, it follows from Corollary \ref{ikDelta} that
$\Delta$ is non-special if and only if $\sum_{\alpha}\sum_{l=0}^{\alpha
k-ns_{\alpha,k}-1}|C_{\alpha,l}| \geq t_{k}-1$ for every $1 \leq k \leq n-1$.
But taking the sum over $k$ yields $g$ on the right hand side, and we claim that
the sum of the left hand sides equals the degree of $\Delta$. Indeed, the set
$C_{\alpha,l}$ appears on the left hand side of the $k$th equality precisely for
those $1 \leq k \leq n-1$ in which the number $\alpha k-ns_{\alpha,k}$ is larger
than $l$. Since the set $\{\alpha k-ns_{\alpha,k}\}_{k=1}^{n-1}$ consists
precisely of the numbers between 1 and $n-1$ (as $\alpha$ is prime to $n$),
precisely $n-1-l$ of those numbers are larger than $l$. Since the degree of
$\Delta$ is $g$, all these inequalities must hold as equalities, which completes
the proof of the theorem.
\end{proof}

One can verify that Theorems 2.6, 2.9, 2.13, 2.15, and A.1 of \cite{[FZ]} are
special cases of Theorem \ref{nonsp}. This verification requires some care: The
sets $C_{j}$ and $D_{j}$ of \cite{[FZ]} correspond to our sets $C_{1,j+1}$ and
$C_{n-1,n-2-j}$ respectively, and the $j$th cardinality condition in these
special cases is obtained by taking the difference of consecutive equalities in
Theorem \ref{nonsp}. It is also possible to verify that Theorems 6.3 and 6.13 of
\cite{[FZ]} follow from Theorem \ref{nonsp}. As a point in $C_{\alpha,l}$
appears to the power $n-1-l$ in $\Delta$, we find that adding $\alpha k$ to it
and then taking the number in $\mathbb{N}_{n}$ which is congruent to the result
yields \[n-1-l+\alpha k-ns_{\alpha,k}-n\chi(l<\alpha k-ns_{\alpha,k}),\] where
$\chi$ of a given condition gives 1 if the condition is satisfied and 0
otherwise. It follows that the sum appearing in Theorem 1 of \cite{[GD]} and
Theorem 2 of \cite{[GDT]} equals $g+nt_{k}-n\sum_{\alpha}\sum_{l=0}^{\alpha
k-ns_{\alpha,k}-1}|C_{\alpha,l}|$, and for prime $n$ Theorem \ref{nonsp} is
equivalent to the results given in these references. Moreover, this argument
shows that the results of \cite{[GD]} and \cite{[GDT]} extend to arbitrary $n$,
provided that the $Z_{n}$ curve is fully ramified (which is always the case when
$n$ is prime). Note that it can happen that no divisors satisfying the
conditions of Theorem \ref{nonsp} exist (see \cite{[GDT]}). We also remark that
our Theorem \ref{nonsp} relates to Theorem 3.4 of \cite{[K]} using these
arguments, but the connection here is more delicate, as the divisors considered
in that reference have degree which is congruent to $g-1$ modulo $n$ (recall
that half the total branching number equals $g+n-1$ in a fully ramified
$Z_{n}$ curve), and ours have degree $g$. The same remark applies for the more
general Theorem 1 of \cite{[A1]}. In addition, one must use the Riemann
Vanishing Theorem (see Section \ref{Operators} below) in order to move between
our language of non-special divisors and the conditions of \cite{[K]} or
\cite{[A1]} regarding vanishing of theta functions. In order to see the relation
more easily, one should assume that the base point in \cite{[K]} or in
\cite{[A1]} is taken to be a branch point (so that some constants there may be
taken to be 0), and multiply our divisors by this base point raised to the
power $n-1$. Indeed, we shall later see that in some sense these divisors of
degree $g+n-1$ (up to come equivalence) are more natural representatives of our
theta characteristics.

\section{Operators on Divisors \label{Operators}}

Let $X$ be a compact Riemann surface of genus $g>0$. By taking a canonical basis
for the homology of $X$, one obtains a symmetric matrix $\Pi \in
M_{g}(\mathbb{C})$, the \emph{period matrix of $X$} with respect to this basis,
whose imaginary part is positive definite. We identify the Jacobian variety
$J(X)$ with the complex torus
$\mathbb{C}^{g}/\mathbb{Z}^{g}\oplus\Pi\mathbb{Z}^{g}$. Let $Div(X)$ denote the
group of divisors on $X$, and let $Div^{0}(X)$ be the subgroup of $Div(X)$
consisting of those divisors whose degree is 0. For a point $Q$ on $X$, we
denote $\varphi_{Q}$ the Abel--Jacobi map from $Div(X)$ to $J(X)$ with base
point $Q$ (see Chapter 3 of \cite{[FK]} or Chapter 1 of \cite{[FZ]} for some
properties of this map). It is related to the algebraic Abel--Jacobi map
$\varphi:Div^{0}(X) \to J(X)$ by
$\varphi_{Q}(\Delta)=\varphi\big(\frac{\Delta}{Q^{\deg\Delta}}\big)$. Hence on
divisors of degree 0 the value of $\varphi_{Q}$ is independent of the choice of
the base point $Q$ (see also Equation (1.1) of \cite{[FZ]}). At this point we
also introduce, following \cite{[FZ]} and others, the useful notation
$\mathbf{e}(t)=e^{2\pi it}$ for $t\in\mathbb{C}$.

Given two vectors $\varepsilon$ and $\varepsilon'$ in $\mathbb{R}^{g}$, one
defines the \emph{theta function with characteristics
$\left[\begin{array}{c}\varepsilon \\ \varepsilon'\end{array}\right]$ and period
matrix $\Pi$} as \[\theta\left[\begin{array}{c}\varepsilon \\
\varepsilon'\end{array}\right](\zeta,\Pi)=\sum_{N\in\mathbb{Z}^{g}}\mathbf{e}
\bigg[\frac{1}{2}\bigg(N+\frac{\varepsilon}{2}\bigg)^{t}\Pi\bigg(N+\frac{
\varepsilon}{2}\bigg)+\bigg(N+\frac{\varepsilon}{2}\bigg)^{t}\bigg(\zeta+\frac{
\varepsilon'}{2}\bigg)\bigg].\] For the properties of this function see Chapter
6 of \cite{[FK]} or Section 1.3 of \cite{[FZ]}. In particular, up to a non-zero
factor, the characteristics correspond to translations of the variable $\zeta$
(see Equation (1.3) of \cite{[FZ]}) in the classical theta function with
$\varepsilon=\varepsilon'=0$. We are interested in \emph{theta constants}, i.e.,
the values of theta functions with rational characteristics at $\zeta=0$. The
original formula of Thomae is a relation between these theta constants on a
hyper-elliptic Riemann surface (or, in our language, a $Z_{2}$ curve). Here we
extend this formula to arbitrary fully ramified $Z_{n}$ curves.

Take a point $e$ in $J(X)$ (or in $\mathbb{C}^{g}$), and consider the
multi-valued function $f(P)=\theta\big(\varphi_{Q}(P)-e,\Pi\big)$ on $X$. The
Riemann Vanishing Theorem (see, e.g., Theorem 1.8 of \cite{[FZ]}) states that
$f$ either vanishes identically on $X$ or has precisely $g$ (well-defined)
zeroes (counted with multiplicity). In the latter case the divisor $\Delta$ of
zeroes of $f$ is non-special and satisfies $e=\varphi_{Q}(\Delta)+K_{Q}$, where
$K_{Q}$ is the vector of Riemann constants associated with $Q$. Moreover, any
element of $e \in J(X)$ can be written as $\varphi_{Q}(\Delta)+K_{Q}$ for some
integral divisor $\Delta$ of degree $g$ on $X$ by the Jacobi Inversion Theorem.
Proposition 1.10 of \cite{[FZ]} shows that $f$ vanishes identically if and only
if $\Delta$ is special. Observe that otherwise the presentation of $e$ as
$\varphi_{Q}(\Delta)+K_{Q}$ is unique: Indeed, applying the Riemann--Roch
Theorem for a non-special integral divisor $\Delta$ of degree $g$ yields
$r\big(\frac{1}{\Delta}\big)=1$. Hence $L\big(\frac{1}{\Delta}\big)=\mathbb{C}$
(the constant functions), and there is no other integral divisor $\Xi$ of degree
$g$ such that $\varphi_{Q}(\Xi)=\varphi_{Q}(\Delta)$ and
$e=\varphi_{Q}(\Xi)+K_{Q}$.

The following proposition about the vector of Riemann constants is very useful
in the theory of Thomae formulae:

\begin{prop}
$\varphi_{Q}$ takes any canonical divisor on $X$ to $-2K_{Q}$. \label{-2KQcan}
\end{prop}

\begin{proof}
See the theorem on page 298 of \cite{[FK]}, or page 21 of \cite{[FZ]}.
\end{proof}

The dependence of $K_{Q}$ on the base point $Q$ is given through the fact that
$\varphi_{Q}(\Delta)+K_{Q}$ is independent of $Q$ if $\Delta$ is a divisor of
degree $g-1$ on $X$ (see Theorem 1.12 of \cite{[FZ]}).

\smallskip

The following property of the vector of Riemann constants, in case the base
point is a branch point on a fully ramified $Z_{n}$ curve, has been obtained in
a few special cases in \cite{[FZ]} (see Lemma 2.4, Lemma 2.12, Lemma 6.2, and
Lemma 6.12 of that reference). However, it turns out to hold in general:

\begin{lem}
Let $Q$ be a branch point on a fully ramified $Z_{n}$ curve of genus $g\geq1$.
Then the vector $K_{Q}$ of Riemann constants associated with the base point $Q$
has order dividing $2n$ in $J(X)$. \label{KQord2n}
\end{lem}

\begin{proof}
Let $\mu=z(Q)\in\mathbb{C}$. Since $g\geq1$, there exists some $1 \leq k \leq
n-1$ such that $t_{k}\geq2$, and then the divisor of
$\omega=(z-\mu)^{t_{k}-2}\omega_{k}$ is supported only on the branch points. But
the fact that $\frac{R^{n}}{Q^{n}}$ is principal for any branch point $R$ (as
the divisor of $\frac{z-z(R)}{z-\mu}$) implies that
$n\varphi_{Q}\big(div(\omega)\big)=0$. In case $z(Q)=\infty$ the divisor of
every differential $\omega_{k}$ is supported on the branch points, and
$\frac{R^{n}}{Q^{n}}$ is the divisor of $z-z(R)$. The conclusion
$n\varphi_{Q}\big(div(\omega)\big)=0$ follows also in this case. As
$\varphi_{Q}\big(div(\omega)\big)=-2K_{Q}$ by Proposition \ref{-2KQcan}, the
assertion follows.
\end{proof}

An assertion related to Lemma \ref{KQord2n} appears just above Theorem 3.4 of
\cite{[K]}, regarding the expression denoted $E_{1}$ there.

\medskip

In all the cases considered in \cite{[FZ]}, the Thomae formulae have been proved
using two types of operators, denoted $N$ and $T_{R}$ (with base point $P_{0}$),
acting on the set of non-special divisors of degree $g$ which are supported on
the branch points distinct from $P_{0}$. We now show that these operators
exist in general (not only on $Z_{n}$ curves!). Let $X$ be an arbitrary compact
Riemann surface of genus $g\geq1$. We denote $v_{Q}(\Delta)$ the power to which
the point $Q$ on $X$ appears in the divisor $\Delta$ on $X$.

\begin{thm}
$(i)$ Let $\Delta$ be a non-special integral divisor of degree $g\geq1$ on $X$,
and let $Q$ be a point on $X$ such that $v_{Q}(\Delta)=0$. There exists a unique
integral divisor $N_{Q}(\Delta)$ of degree $g$ on $X$ satisfying
\begin{equation}
\varphi_{Q}\big(N_{Q}(\Delta)\big)+K_{Q}=-\big(\varphi_{Q}(\Delta)+K_{Q}\big).
\label{NQeq}
\end{equation}
The divisor $N_{Q}(\Delta)$ is non-special, and satisfies
$v_{Q}\big(N_{Q}(\Delta)\big)=0$. The operator $N_{Q}$ is an involution on the
set of non-special integral divisors of degree $g$ not containing $Q$ in their
support. $(ii)$ Given any point $R$ such that $v_{R}\big(N_{Q}(\Delta)\big)=0$,
there exists a unique integral divisor $T_{Q,R}(\Delta)$ of degree $g$ on $X$
such that the equality
\begin{equation}
\varphi_{Q}\big(T_{Q,R}(\Delta)\big)+K_{Q}=-\big(\varphi_{Q}(\Delta)+\varphi_{Q}
(R)+K_{Q}\big) \label{TQReq}
\end{equation}
holds. The divisor $T_{Q,R}(\Delta)$ is also non-special, and we have the
equalities $v_{Q}\big(T_{Q,R}(\Delta)\big)=0$ and
$v_{R}\big(N_{Q}\big(T_{Q,R}(\Delta)\big)\big)=0$. The operator $T_{Q,R}$,
which is defined on the set of non-special divisors on $X$ not containing $Q$
in their support and such that $R$ does not appear in $N_{Q}(\Delta)$, is an
involution on this set of divisors. \label{NQTQR}
\end{thm}

\begin{proof}
Denote by $e \in J(X)$ the expression on the right hand side of Equation
\eqref{NQeq}, and consider the (multi-valued) function
$f(P)=\theta\big(\varphi_{Q}(P)+e,\Pi\big)$. Since $-e$ equals
$\varphi_{Q}(\Delta)+K_{Q}$ and $i(\Delta)=0$, we find that $f$ does not vanish
identically, but rather vanishes only on points in the support of $\Delta$. The
condition $v_{Q}(\Delta)=0$ thus implies $\theta(e,\Pi)\neq0$, and since
$\theta$ is an even function, we deduce $\theta(-e,\Pi)\neq0$. But this implies
that $\psi(P)=\theta\big(\varphi_{Q}(P)-e,\Pi\big)$ does not vanish at $P=Q$,
hence does not vanish identically. Thus $e=\varphi_{Q}(\Xi)+K_{Q}$ for some
non-special divisor $\Xi$ representing the zeroes of $\psi$, so that in
particular $v_{Q}(\Xi)=0$. Since $\Xi$ and $N_{Q}(\Delta)$ are both integral of
degree $g$ and have the same $\varphi_{Q}$-images, the fact that $\Xi$ is
non-special implies $\Xi=N_{Q}(\Delta)$. The fact that $N_{Q}(\Delta)$ is
non-special and $v_{Q}\big(N_{Q}(\Delta)\big)=0$ yields the existence of a
unique divisor $N_{Q}\big(N_{Q}(\Delta)\big)$ satisfying Equation \eqref{NQeq}
with $\Delta$ replaced by $N_{Q}(\Delta)$. As $\Delta$ satisfies this equation,
the equality $N_{Q}\big(N_{Q}(\Delta)\big)=\Delta$ follows, and $N_{Q}$ is an
involution. This proves $(i)$. In order to establish $(ii)$ we denote the value
on the right hand side of Equation \eqref{TQReq} by $d$, and consider the
multi-valued function $\varrho(P)=\theta\big(\varphi_{Q}(P)+d,\Pi\big)$. As
$\varrho(R)=f(Q)$ and the latter expression is non-vanishing, we find that $-d$
can be written as $\varphi_{Q}(\Upsilon)+K_{Q}$ where $\Upsilon$ is a
non-special integral divisor of degree $g$ representing the zeroes of $\varrho$
(hence $v_{R}(\Upsilon)=0$). Moreover, $\varrho(Q)$ equals $f(R)$ and is also
non-vanishing by our assumption on $R$. This shows that $v_{Q}(\Upsilon)=0$ as
well, and we define $T_{Q,R}(\Delta)=N_{Q}(\Upsilon)$. The equality
$\Upsilon=N_{Q}\big(T_{Q,R}(\Delta)\big)$ (as $N_{Q}$ is an involution) and part
$(i)$ imply that $T_{Q,R}(\Delta)$ has the asserted properties. In particular,
$T_{Q,R}\big(T_{Q,R}(\Delta)\big)$ is defined, and since it is characterized by
satisfying Equation \eqref{TQReq} with $\Delta$ replaced by $T_{Q,R}(\Delta)$,
we deduce that $T_{Q,R}$ is an involution as in part $(i)$. This completes the
proof of the theorem.
\end{proof}

Note that part $(ii)$ of Theorem \ref{NQTQR} does not require that $R \neq Q$.
However, if $R=Q$ then the right hand side of Equation \eqref{TQReq} reduces to
that of Equation \eqref{NQeq}, implying that $T_{Q,Q}(\Delta)$ is simply
$N_{Q}(\Delta)$. We shall therefore always assume $R \neq Q$ in $T_{Q,R}$.

\medskip

We are interested in the form of the operators $N_{Q}$ and $T_{Q,R}$ in the case
where $X$ is a fully ramified $Z_{n}$ curve and $Q$ and $R$ are branch points on
$X$. Assume that $X$ is associated with Equation \eqref{Zneq}, $\Delta$ is given
by Equation \eqref{Delta}, and $Q=P_{\beta,i}$ for some $\beta\in\mathbb{N}_{n}$
which is prime to $n$ and some index $i$. Hence $\mu=z(Q)$ equals
$\lambda_{\beta,i}$, but we keep the notation $\mu$. Let $k_{\beta}$ be an
integer such that $n|\beta k_{\beta}-1$ (hence $\beta
k_{\beta}-ns_{\beta,k_{\beta}}=1$). This characterizes the class of $k_{\beta}$
in $\mathbb{Z}/n\mathbb{Z}$ (we rather not impose the assumption
$k_{\beta}\in\mathbb{N}_{n}$). The point $Q$ does not lie in the support of
$\Delta$ if and only if $Q \in C_{\beta,n-1}$ in the notation of Equation
\eqref{Delta}. For any $\alpha$ and $l$ we denote by $a_{\beta,\alpha}(l)$ and
$b_{\beta,\alpha}(l)$ the elements of $\mathbb{N}_{n}$ which are congruent
modulo $n$ to $\alpha k_{\beta}-1-l$ and $2\alpha k_{\beta}-1-l$ respectively.
These numbers are of course independent of the choice of
$k_{\beta}\in\mathbb{Z}$. We consider $a_{\beta,\alpha}$ and $b_{\beta,\alpha}$
as functions on $\mathbb{N}_{n}$, and these functions are involutions. Two
useful equalities concerning these involutions are given in the following

\begin{lem}
The equality
$a_{\beta,\alpha}\big(b_{\beta,\alpha}(l)\big)=n-1-a_{\beta,\alpha}(l)$ holds
for every $\alpha$, $\beta$, and $l\in\mathbb{N}_{n}$. It is equivalent to the
equality
$a_{\beta,\alpha}\big[b_{\beta,\alpha}\big(a_{\beta,\alpha}(l)\big)\big]=n-1-l$
holding for all such $\alpha$, $\beta$, and $l$. \label{abn-1-a}
\end{lem}

\begin{proof}
The first equality follows from the fact that both expressions are elements of
$\mathbb{N}_{n}$ which are congruent to $l-\alpha k_{\beta}$ modulo $n$. The
second equality is obtained from the first by replacing $l$ by
$a_{\beta,\alpha}(l)$ and using the fact that $a_{\beta,\alpha}$ is an
involution. This proves the lemma.
\end{proof}

It will turn our convenient to let the index $l$ of $C_{\alpha,l}$ to be any
integer, while identifying $C_{\alpha,l}$ with $C_{\alpha,l+n}$ for every
$l\in\mathbb{Z}$. In this way we can consider the set $C_{\alpha,\alpha
k_{\beta}}$, for example, without having to write $C_{\alpha,\alpha
k_{\beta}-ns_{\alpha,k_{\beta}}}$.

We now assume that the $Z_{n}$ curve has genus $g$ at least 1, for the theory of
theta functions to be non-trivial. This means $\sum_{\alpha}r_{\alpha}\geq3$ by
the expression for $g$. The following proposition generalizes Definitions 2.16,
2.18, 6.4, and 6.14 of \cite{[FZ]}, as well as Propositions 2.17, 2.19, 6.5, and
6.15 there:

\begin{prop}
If $\Delta$ is given by Equation \eqref{Delta} and $v_{Q}(\Delta)=0$ then the
divisor $N_{Q}(\Delta)$ is defined by the formula
\[N_{Q}(\Delta)=\prod_{\alpha}\prod_{l=0}^{n-1}C_{\alpha,l}^{n-1-a_{\beta,\alpha
}(l)}/Q^{n-2}.\] Moreover, assume that the branch point $R \neq Q$ does not
appear in the support of $N_{Q}(\Delta)$ (this means $R \in C_{\gamma,\gamma
k_{\beta}}$ if $R=P_{\gamma,j}$ for some index $j$). Then $T_{Q,R}(\Delta)$ is
given by
\[T_{Q,R}(\Delta)=\prod_{\alpha}\prod_{l=0}^{n-1}C_{\alpha,l}^{n-1-b_{\beta,
\alpha}(l)}/RQ^{n-3},\] and $v_{R}\big(T_{Q,R}(\Delta)\big)=v_{R}(\Delta)$.
\label{operZn}
\end{prop}

\begin{proof}
Denote the asserted values of $N_{Q}(\Delta)$ and $T_{Q,R}(\Delta)$ by $\Xi$
and $\Psi$ respectively. By Theorem \ref{NQTQR} it suffices to prove that
$\Xi$ and $\Psi$ are of degree $g$ and satisfy Equation \eqref{NQeq} and
\eqref{TQReq} respectively. The latter equations are equivalent to
\[\varphi_{Q}\big(\Delta\cdot\Xi\big)=\varphi_{Q}\big(\Delta \cdot
R\cdot\Psi\big)=-2K_{Q},\] so that by Proposition \ref{-2KQcan} it suffices to
find differentials on $X$ such that their divisors have the same
$\varphi_{Q}$-images as $\Delta\Xi$ or $\Delta R\Psi$. Consider the
differentials $(z-\mu)^{t_{k_{\beta}}-2}\omega_{k_{\beta}}$ and
$(z-\mu)^{t_{2k_{\beta}}-2}\omega_{2k_{\beta}}$. Equation \eqref{divomegak}
shows that their divisors are
\[\prod_{\alpha}\prod_{i=1}^{r_{\alpha}}P_{\alpha,i}^{n-1+ns_{\alpha,k_{\beta}}
-\alpha k_{\beta}}Q^{n(t_{k_{\beta}}-2)}\mathrm{\ \  and\ \
}\prod_{\alpha}\prod_{i=1}^{r_{\alpha}}P_{\alpha,i}^{n-1+ns_{\alpha,2k_{\beta}}
-2\alpha k_{\beta}}Q^{n(t_{k_{2\beta}}-2)}\] respectively (in fact, our choice
of $k_{\beta}$ shows that the total power of $Q$ in these divisors are
$n(t_{k_{\beta}}-1)-2$ and $n(t_{k_{2\beta}}-1)-3$ respectively). Now, for any
$\alpha$ and $l$ the equalities
\[(n-1-l)+(n-1-a)=n-1+ns_{\alpha,k_{\beta}}-\alpha k_{\beta}+n\chi(l<\alpha
k_{\beta}-ns_{\alpha,k_{\beta}})\] and
\[(n-1-l)+(n-1-b)=n-1+ns_{\alpha,2k_{\beta}}-2\alpha k_{\beta}+n\chi(l<2\alpha
k_{\beta}-ns_{\alpha,2k_{\beta}})\] hold, where $a$ and $b$ stand for
$a_{\beta,\alpha}(l)$ and $b_{\beta,\alpha}(l)$ respectively. Indeed, both
sides are congruent to $-1-\eta\alpha k_{\beta}$ modulo $n$ (with $\eta$ being
1 for the first equation and 2 for the second one), and the two numbers on the
left hand side and the number on the right hand side not involving the
conditional expression are all elements of $\mathbb{N}_{n}$. As an index $i$
satisfies $i \in A_{\alpha,\eta k_{\beta}}$ if and only if $P_{i,\alpha}$ lies
in a set $C_{\alpha,l}$ with $l<\eta\alpha k_{\beta}-ns_{\alpha,\eta
k_{\beta}}$, it follows that
\[\Delta\cdot\Xi=\prod_{\alpha}\Bigg[\prod_{i=1}^{r_{\alpha}}P_{\alpha,i}^{
n-1+ns_{\alpha,k_{\beta}}-\alpha k_{\beta}}\prod_{i \in
A_{\alpha,k_{\beta}}}P_{\alpha,i}^{n}\Bigg]/Q^{n-2}\] and \[\Delta \cdot
R\cdot\Psi=\prod_{\alpha}\Bigg[\prod_{i=1}^{r_{\alpha}}P_{\alpha,i}^{n-1+ns_{
\alpha,2k_{\beta}}-2\alpha k_{\beta}}\prod_{i \in
A_{\alpha,2k_{\beta}}}P_{\alpha,i}^{n}\Bigg]/Q^{n-3}.\] The number of points
$P_{\alpha,i}$ appearing to the power $n$ is $\sum_{\alpha}\sum_{l=0}^{\alpha
k_{\beta}-ns_{\alpha,k_{\beta}}-1}|C_{\alpha,l}|$ or
$\sum_{\alpha}\sum_{l=0}^{2\alpha
k_{\beta}-ns_{\alpha,2k_{\beta}}-1}|C_{\alpha,l}|$. These numbers equal
$t_{k_{\beta}}-1$ and
$t_{2k_{\beta}}-1$ respectively by Theorem \ref{nonsp} as $\Delta$ is
non-special. Hence $\Delta\Xi$ is linearly equivalent to
$\prod_{\alpha}\prod_{i=1}^{r_{\alpha}}P_{\alpha,i}^{n-1+ns_{\alpha,k_{\beta}}
-\alpha k_{\beta}}Q^{n(t_{k_{\beta}}-1)-(n-2)}$, while $\Delta R\Psi$ is
linearly equivalent to
$\prod_{\alpha}\prod_{i=1}^{r_{\alpha}}P_{\alpha,i}^{n-1+ns_{\alpha,2k_{\beta}}
-2\alpha k_{\beta}}Q^{n(t_{2k_{\beta}}-1)-(n-3)}$. These divisors are $Q^{2}$
times the divisor of $(z-\mu)^{t_{k_{\beta}}-2}\omega_{k_{\beta}}$ and $Q^{3}$
times the divisor of $(z-\mu)^{t_{2k_{\beta}}-2}\omega_{2k_{\beta}}$ given
above.
Since the degree of a canonical divisor is $2g-2$ and $\varphi_{Q}(Q)=0$, this
proves that $\Xi$ and $\Psi$ have the required properties. Hence
$\Xi=N_{Q}(\Delta)$ and $\Psi=N_{Q}(\Delta)$ as desired. Observe that for
$l=\gamma k_{\beta}-ns_{\gamma,k_{\beta}}$ we have $a_{\beta,\gamma}(l)=n-1$
(as desired for $v_{R}\big(N_{Q}(\Delta)\big)=0$) and $b_{\beta,\gamma}(l)$ is
congruent to $\gamma k_{\beta}-1$ modulo $n$. It follows that
$v_{R}\big(T_{Q,R}(\Delta)\big)$ coincides with $v_{R}(\Delta)$ since the
division by $R$ covers for this difference of 1 between $l$ and
$b_{\beta,\gamma}(l)$. This proves the proposition.
\end{proof}

As an example for the formulae given in Proposition \ref{operZn} we present the
case $n=3$ explicitly (compare the respective formulae from Definitions 2.16
and 2.18 of \cite{[FZ]}). The divisor $\Delta$ takes the form
$C_{1,0}^{2}C_{2,0}^{2}C_{1,1}C_{2,1}$, and $Q$ may be taken from either
$C_{1,2}$ or $C_{2,2}$ (the sets not appearing in $\Delta$). Recall that $\alpha
k_{\beta}$ is 1 for $\alpha=\beta$ and 2 if $\alpha\neq\beta$, where $\alpha$
and $\beta$ are taken from the set $\{1,2\}$. Hence if $Q \in C_{1,2}$ then
$N_{Q}(\Delta)=C_{1,0}^{2}C_{2,1}^{2}C_{1,2}C_{2,0}/Q$, while for $Q \in
C_{2,2}$ we have $N_{Q}(\Delta)=C_{1,1}^{2}C_{2,0}^{2}C_{1,0}C_{2,2}/Q$. In the
former case we may take $R$ either from $C_{1,1}$ or from $C_{2,2}$, and then
$T_{Q,R}(\Delta)=C_{1,1}^{2}C_{2,0}^{2}C_{1,0}C_{2,2}/R$ (observe that the
formula for $T_{Q,R}$ depends only on the type of $Q$, so that we get the same
formula for the two choices of $R$). With the latter choice of $Q$ the
admissible points $R$ lie in the sets $C_{2,1}$ and $C_{1,2}$, both giving
$T_{Q,R}(\Delta)=C_{1,0}^{2}C_{2,1}^{2}C_{1,2}C_{2,0}/R$.

The part of Proposition \ref{operZn} concerning $N_{Q}$ is related to
Proposition 6.2 of \cite{[K]}, as both assertions give formula for ``negation
of a divisor'' in $J(X)$ with respect to $K_{Q}$. However, our restriction of
the degree to stay $g$ makes our argument slightly more delicate.

\section{The Poor Man's Thomae Formulae \label{PMT}}

Let $Q$ be a point on a Riemann surface $X$ with period matrix $\Pi$ with
respect to a canonical basis, and let $\Delta$ be a divisor of degree $g$ on
$X$. If $\varphi_{Q}(\Delta)+K_{Q}$ is the $J(X)$-image of
$\Pi\frac{\varepsilon}{2}+I\frac{\varepsilon'}{2}\in\mathbb{C}^{g}$ then we
denote, following Section 2.6 of \cite{[FZ]}, the theta function with
characteristics $\left[\begin{array}{c} \varepsilon \\
\varepsilon'\end{array}\right]$ by $\theta[Q,\Delta](z,\Pi)$ . This function
depends on the choice of the lift (i.e., on $\varepsilon$ and $\varepsilon'$ not
up to $2\mathbb{Z}^{g}$). However, if $X$ is a fully ramified $Z_{n}$ curve, $Q$
is a branch point, and $\Delta$ is supported on the branch points, then the
vectors $\varepsilon$ and $\varepsilon'$ lie in $\frac{1}{n}\mathbb{Z}^{g}$ (see
Lemma \ref{KQord2n}). In this case Equation (1.4) of \cite{[FZ]} shows that
changing the lift can only multiply the function by a constant which is a root
of unity of order dividing $2n$. It follows that $\theta^{2n}[Q,\Delta](z,\Pi)$
is independent of the lift. The same assertion thus holds for
$\theta^{en^{2}}[Q,\Delta](z,\Pi)$ where $e$ is 1 for even $n$ and 2 for odd
$n$. The main fact which allows us to start this process is given in the
following
\begin{prop}
Let $\Delta$ and $\Xi$ be two non-special divisors of degree $g$ which are
supported on the branch points distinct from $Q$. The quotient
$\frac{\theta^{en^{2}}[Q,\Delta](\varphi_{Q}(P),\Pi)}{\theta^{en^{2}}[Q,\Xi]
(\varphi_{Q}(P),\Pi)}$ is a well-defined function on $X$, which is a constant
multiple of the function
$\prod_{\alpha,i}(z-\lambda_{\alpha,i})^{en[v_{P_{\alpha,i}}(N_{Q}(\Delta))-v_{
P_{\alpha,i}}(N_{Q}(\Xi))]}$. If $R$ is some branch point such that
$v_{R}\big(N_{Q}(\Delta)\big)=v_{R}\big(N_{Q}(\Xi)\big)=0$ then the value of
this function at $P=R$ equals
$\frac{\theta^{en^{2}}[Q,T_{Q,R}(\Delta)](0,\Pi)}{\theta^{en^{2}}[Q,T_{Q,R}(\Xi)
](0,\Pi)}$. In particular, if the divisor $\Xi$ is $T_{Q,R}(\Delta)$ then the
equality
$\frac{\theta^{en^{2}}[Q,\Delta](\varphi_{Q}(R),\Pi)}{\theta^{en^{2}}[Q,T_{Q,R}
(\Delta)](\varphi_{Q}(R),\Pi)}=\frac{\theta^{en^{2}}[Q,T_{Q,R}(\Delta)](0,\Pi)}{
\theta^{en^{2}}[Q,\Delta](0,\Pi)}$ holds. \label{thetaquot}
\end{prop}

\begin{proof}
See Section 2.6 of \cite{[FZ]}, specifically the paragraphs preceding
Proposition 2.21 there and the one involving Equations (2.1), (2.2), (2.3), and
(2.4) of that reference. Some special cases are considered there, but the
argument works in general. Propositions 6.6 and 6.16 of \cite{[FZ]} are
examples of the more general applicability of this argument.
\end{proof}

\smallskip

We now use Proposition \ref{thetaquot} in order to obtain relations between
theta constants on $X$, following the method used in all the special cases
presented in \cite{[FZ]}. By substituting $P=Q$ in the first quotient appearing
in Proposition \ref{thetaquot} we obtain the value of the constant, so that this
quotient equals
\[\frac{\theta^{en^{2}}[Q,\Delta](0,\Pi)}{\theta^{en^{2}}[Q,\Xi](0,\Pi)}
\cdot\frac{\prod_{\alpha,i}(\mu-\lambda_{\alpha,i})^{env_{P_{\alpha,i}}(N_{Q}
(\Xi))}}{\prod_{\alpha,i}(\mu-\lambda_{\alpha,i})^{env_{P_{\alpha,i}}(N_{Q}
(\Delta))}}\cdot\frac{\prod_{\alpha,i}(z-\lambda_{\alpha,i})^{env_{P_{\alpha,i}}
(N_{Q}(\Delta))}}{\prod_{\alpha,i}(z-\lambda_{\alpha,i})^{env_{P_{\alpha,i}}(N_{
Q}(\Xi))}}.\] By choosing $\Xi=T_{Q,R}(\Delta)$ and substituting $P=R$ we obtain
the equality
\[\frac{\theta^{2en^{2}}[Q,\Delta](0,\Pi)}{\prod_{\alpha,i}(\mu-\lambda_{\alpha,
i})^{env_{P_{\alpha,i}}(N_{Q}(\Delta))}\prod_{\alpha,i}(\sigma-\lambda_{\alpha,i
})^{env_{P_{\alpha,i}}(N_{Q}(T_{Q,R}(\Delta)))}}=\]
\begin{equation}
=\frac{\theta^{2en^{2}}[Q,T_{Q,R}(\Delta)](0,\Pi)}{\prod_{\alpha,i}(\mu-\lambda_
{\alpha,i})^{env_{P_{\alpha,i}}(N_{Q}(T_{Q,R}(\Delta)))}\prod_{\alpha,i}
(\sigma-\lambda_{\alpha,i})^{env_{P_{\alpha,i}}(N_{Q}(\Delta))}},
\label{1strel}
\end{equation}
where $\sigma=z(R)$.

In order to illuminate the presentation of the following stages, we show how to
apply the arguments on an explicit example. We assume that $n$ is odd, and that
only the powers 1, 2, $n-2$, and $n-1$ appear. Hence, after easing the notation
a bit, our $Z_{n}$ curve becomes associated with an equation of the form
\[w^{n}=\prod_{i=1}^{r}(z-\lambda_{i})\prod_{i=1}^{p}(z-\sigma_{i})^{2}\prod_{
i=1}^{q}(z-\tau_{i})^{n-2}\prod_{i=1}^{m}(z-\mu_{i})^{n-1}\] for odd $n$, where
$r+2p-2q-m$ is divisible by $n$ hence equals $nu$ for some $u\in\mathbb{Z}$. In
order to ease the notation further, we switch to a notation similar to that of
\cite{[FZ]}, so that the branch points over $\lambda_{i}$, $\mu_{i}$,
$\sigma_{i}$, and $\tau_{i}$ are denoted $P_{i}$, $Q_{i}$, $R_{i}$, and $S_{i}$
respectively. We shall write the divisor $\Xi=Q^{n-1}\Delta$ (see also Equation
\eqref{Xi} below) as
$\prod_{l=0}^{n-1}C_{l}^{n-1-l}D_{l}^{l}E_{l}^{n-1-l}F_{l}^{l}$, where $C_{l}$
contains points of the sort $P_{i}$, $D_{l}$ contains points of the type
$Q_{i}$, $E_{l}$ contains points of the type $R_{i}$, and $F_{l}$ contains
points of the sort $S_{i}$. Hence $C_{l}$ stands for $C_{1,l}$, $D_{l}$ is the
set $C_{n-1,n-1-l}$, $E_{l}$ represents $C_{2,l}$, and $F_{l}$ is our notation
for $C_{n-1,n-1-l}$, but in which we took $Q$ out of the appropriate set
$C_{\beta,n-1}$ and put it into $C_{\beta,0}$. The number $t_{k}$ is $ku+q+m$ if
$k\leq\frac{n-1}{2}$ and equals $ku-p+2q+m$ if $k\geq\frac{n+1}{2}$. The
cardinality conditions from Theorem \ref{nonsp} are seen (mainly by subtracting
two such subsequent conditions, and observing that $m=\sum_{l}|D_{l}|$,
$q=\sum_{l}|F_{l}|$, and $p=\sum_{l}|E_{l}|$) to be equivalent to the equalities
\[|C_{k}|+|E_{2k}|+|E_{2k+1}|=|D_{k}|+|F_{2k}|+|F_{2k+1}|+u\] for all $0 \leq
k\leq\frac{n-3}{2}$,
\[|C_{k}|+|E_{2k-n}|+|E_{2k+1-n}|=|D_{k}|+|F_{2k-n}|+|F_{2k+1-n}|+u\] wherever
$\frac{n+1}{2} \leq k \leq n-1$, and the remaining index $k=\frac{n-1}{2}$ is
associated with the equation
\[|C_{\frac{n-1}{2}}|+|E_{0}|+|E_{n-1}|=|D_{\frac{n-1}{2}}|+|F_{0}|+|F_{n-1}|+u.
\] Indeed, the left hand sides sum to $r+2p$ and the sum of the right hand sides
is $nu+2q+m$, in agreement with the value of $u$. However, the final formulae
will not depend on the cardinalities of the sets involved, hence the values of
$r$, $p$, $q$, and $m$ (and $u$) are irrelevant for the rest of the discussion
about this example (apart for the assumptions $r+p+q+m\geq3$ in order to have
genus $g\geq1$, and $r\geq1$ since we shall take $\beta=1$ in many explicit
expression below).

Back to the general case, we write $\Delta$ as in Equation \eqref{Delta} in
order to express the latter equality using the sets appearing that Equation. The
divisor $N_{Q}(\Delta)$ is given in Proposition \ref{operZn}, and using the fact
that $b_{\beta,\alpha}$ is an involution on $\mathbb{N}_{n}$ we write the
formula for $T_{Q,R}(\Delta)$ in Proposition \ref{operZn} as
$\prod_{\alpha}\prod_{l=0}^{n-1}C_{\alpha,b_{\beta,\alpha}(l)}^{n-1-l}/RQ^{n-3}
$. Thus the divisor $N_{Q}\big(T_{Q,R}(\Delta)\big)$ equals
$\prod_{\alpha}\prod_{l=0}^{n-1}C_{\alpha,b_{\beta,\alpha}(l)}^{n-1-a_{\beta,
\alpha}(l)}/QR^{n-1}$, or equivalently
$\prod_{\alpha}\prod_{l=0}^{n-1}C_{\alpha,l}^{a_{\beta,\alpha}(l)}/QR^{n-1}$ by
the involutive property of $b_{\beta,\alpha}$ and Lemma \ref{abn-1-a}. The
powers of $R$ and $Q$ are determined by the condition that both points must not
appear in the support of $N_{Q}\big(T_{Q,R}(\Delta)\big)$. Let $S$ be a point in
$X$ and let $Y$ and $Z$ be (finite) disjoint subsets of points on $X$.
Following Definition 4.1 of \cite{[FZ]}, we introduce the notation
\[[S,Y]=\prod_{T \in Y,T \neq S}\big(z(S)-z(T)\big),\qquad[Y,Z]=\prod_{S \in Y,T
\in Z}\big(z(S)-z(T)\big),\] and \[[Y,Y]=\prod_{S<T \in Y}\big(z(S)-z(T)\big)\]
for some ordering on the points of $Y$. When taken to an even power, the
expression $[Y,Y]$ becomes independent of the order and $[Y,Z]$ coincides with
$[Z,Y]$. In order to ease notation in some expressions below we shorthand the
sets $C_{\alpha,l}\setminus\{Q\}$ and $C_{\alpha,l}\setminus\{Q,R\}$ to simply
$C_{\alpha,l}^{Q}$ and $C_{\alpha,l}^{Q,R}$. The denominators under
$\theta^{2en^{2}}[Q,\Delta](0,\Pi)$ and
$\theta^{2en^{2}}[Q,T_{Q,R}(\Delta)](0,\Pi)$ in Equation \eqref{1strel}
become
\[\prod_{\alpha,l}[Q,C_{\alpha,l}^{Q,R}]^{en[n-1-a_{\beta,\alpha}(l)]}\prod_{
\alpha,l}[R,C_{\alpha,l}^{Q,R}]^{ena_{\beta,\alpha}(l)}\] and
\[\prod_{\alpha,l}[Q,C_{\alpha,l}^{Q,R}]^{ena_{\beta,\alpha}(l)}\prod_{\alpha,l}
[R,C_{\alpha,l}^{Q,R}]^{en[n-1-a_{\beta,\alpha}(l)]}\] respectively. We prefer
to use the sets $C_{\alpha,l}^{Q}$ and $C_{\alpha,l}^{Q,R}$ rather than
evaluating the powers of $(\sigma-\mu)$ which have to be canceled since the
symmetrization is easier in this way. Note that in our example these sets
``forget'' that $Q$ was taken from $C_{\beta,n-1}$ to $C_{\beta,0}$.

We would like to write these denominators explicitly in our example. Let us
assume further that $\beta=1$, i.e., the base point $Q$ is $P_{1,i}=P_{i}$ for
some $1 \leq i \leq r$ (hence $\mu=\lambda_{i}$, but this value will appear only
implicitly in what follows). Then the value of $a_{1,\alpha}(l)$ is as follows.
For $\alpha=1$ it is $n-l$ for all $l\geq1$ and 0 for $l=0$. If $\alpha=2$ then
it equals $n+1-l$ if $l\geq2$, and the values $l=0$ and $l=1$ are interchanged.
When $\alpha=n-2$ this expression is $n-3-l$ if $l \leq n-3$, and this function
interchanges the values $l=n-2$ with $l=n-1$. In the case $\alpha=n-1$ we get
$n-2-l$ for $l \leq n-2$, while the value $l=n-1$ remains invariant. We also
have $e=2$ since $n$ is assumed to be odd. Using these values and remembering
the index inversion in the sets $D_{l}$ and $F_{l}$, we find that these
denominators are
\[[Q,F_{0}^{Q,R}]^{2n}\prod_{l=1}^{n-1}\big([Q,C_{l}^{Q,R}][R,D_{l}^{Q,R}]\big)^
{2n(l-1)}\big([Q,D_{l}^{Q,R}][R,C_{l}^{Q,R}]\big)^{2n(n-l)}\times\]
\[\times[R,E_{0}^{Q,R}]^{2n}\prod_{l=2}^{n-1}\big([Q,E_{l}^{Q,R}][R,F_{l}^{Q,R}]
\big)^{2n(l-2)}\big([Q,F_{l}^{Q,R}][R,E_{l}^{Q,R}]\big)^{2n(n+1-l)}\times\]
\[\times\big([Q,C_{0}^{Q,R}][Q,E_{1}^{Q,R}][R,D_{0}^{Q,R}][R,F_{1}^{Q,R}]\big)^{
2n(n-1)}\big([Q,E_{0}^{Q,R}][R,F_{0}^{Q,R}]\big)^{2n(n-2)}\] and
\[[Q,E_{0}^{Q,R}]^{2n}\prod_{l=1}^{n-1}\big([Q,C_{l}^{Q,R}][R,D_{l}^{Q,R}]\big)^
{2n(n-l)}\big([Q,D_{l}^{Q,R}][R,C_{l}^{Q,R}]\big)^{2n(l-1)}\times\]
\[\times[R,F_{0}^{Q,R}]^{2n}\prod_{l=2}^{n-1}\big([Q,E_{l}^{Q,R}][R,F_{l}^{Q,R}]
\big)^{2n(n+1-l)}\big([Q,F_{l}^{Q,R}][R,E_{l}^{Q,R}]\big)^{2n(l-2)}\times\]
\[\times\big([Q,D_{0}^{Q,R}][Q,F_{1}^{Q,R}][R,C_{0}^{Q,R}][R,E_{1}^{Q,R}]\big)^{
2n(n-1)}\big([Q,F_{0}^{Q,R}][R,E_{0}^{Q,R}]\big)^{2n(n-2)}\] respectively.

Equation \eqref{1strel} compares the values of two theta constants, and is
dependent on both $Q$ and $R$. Our next step is to obtain an equality which is
independent of $R=P_{\gamma,j}$ (which lies in $C_{\gamma,\gamma
k_{\beta}}^{Q}$) but only of the index $\gamma$. We shall carry out this step
already in the general setting, since the example will not be very illustrative
at this point, but only later. In order to free the denominator under
$\theta^{2en^{2}}[Q,\Delta](0,\Pi)$ in Equation \eqref{1strel} from its
dependence on $R \in C_{\gamma,\gamma k_{\beta}}^{Q}$ we would like to divide
that Equation by the expression \[\prod_{(\alpha,l)\neq(\gamma,\gamma
k_{\beta}-ns_{\gamma,k_{\beta}})}[C_{\gamma,\gamma
k_{\beta}}^{Q,R},C_{\alpha,l}^{Q,R}]^{ena_{\beta,\alpha}(l)}\cdot[C_{\gamma,
\gamma k_{\beta}}^{Q,R},C_{\gamma,\gamma k_{\beta}}^{Q,R}]^{en(n-1)}.\] The
latter multiplier is obtained by setting $\alpha=\gamma$ and $j=\gamma
k_{\beta}$, but as the behavior of $[Y,Z]$ for $Y \cap Z=\emptyset$ is different
from that of $[Y,Y]$, we prefer to separate this term from the product. Now,
Equation \eqref{1strel} is symmetric under interchanging $\Delta$ and
$T_{Q,R}(\Delta)$, and we wish to preserve this symmetry. By writing the formula
for $T_{Q,R}(\Delta)$ from Proposition \ref{operZn} as
$\prod_{\alpha,l}\widetilde{C}_{\alpha,l}^{n-1-l}$ we obtain, after omitting the
problematic points $Q$ and $R$, the equality
$\widetilde{C}_{\alpha,l}^{Q,R}=C_{\alpha,b_{\beta,\alpha}(l)}^{Q,R}$. As
$b_{\beta,\gamma}$ subtracts 1 from $\gamma k_{\beta}-ns_{\gamma,k_{\beta}}$,
the fact that $b_{\beta,\alpha}$ is an involution for every $\alpha$ and Lemma
\ref{abn-1-a} allow us to write the expression by which we have divided Equation
\eqref{1strel} as \[\prod_{(\alpha,l)\neq(\gamma,\gamma
k_{\beta}-1-ns_{\gamma,k_{\beta}})}[\widetilde{C}_{\gamma,\gamma
k_{\beta}-1}^{Q,R},\widetilde{C}_{\alpha,l}^{Q,R}]^{en[n-1-a_{\beta,\alpha}(l)]}
\cdot[\widetilde{C}_{\gamma,\gamma
k_{\beta}-1}^{Q,R},\widetilde{C}_{\gamma,\gamma k_{\beta}-1}^{Q,R}]^{en(n-1)}.\]
In order to keep the symmetry, we must divide Equation \eqref{1strel} also by
\[\prod_{(\alpha,l)\neq(\gamma,\gamma
k_{\beta}-1-ns_{\gamma,k_{\beta}})}[C_{\gamma,\gamma
k_{\beta}-1}^{Q,R},C_{\alpha,l}^{Q,R}]^{en[n-1-a_{\beta,\alpha}(l)]}\cdot[C_{
\gamma,\gamma k_{\beta}-1}^{Q,R},C_{\gamma,\gamma
k_{\beta}-1}^{Q,R}]^{en(n-1)}.\] The following observations help to simplify the
result. First, as $R \in C_{\gamma,\gamma k_{\beta}}$ and $\gamma
k_{\beta}-1\neq\gamma k_{\beta}$ in $\mathbb{Z}/n\mathbb{Z}$, we can omit the
superscript $R$ from $C_{\gamma,\gamma k_{\beta}-1}^{Q,R}$. The same assertion
holds for any set $C_{\alpha,l}$ with $(\alpha,l)\neq(\gamma,\gamma k_{\beta})$.
Second, the sets $C_{\gamma,\gamma k_{\beta}-1}$ and $C_{\gamma,\gamma
k_{\beta}}$ do not contain $Q$ (since $Q \in C_{\beta,n-1}$ and if
$\gamma=\beta$ then neither $\gamma k_{\beta}\in1+n\mathbb{Z}$ nor $\gamma
k_{\beta}-1 \in n\mathbb{Z}$ are congruent to $n-1$ modulo $n$), so that we can
omit $Q$ from its notation as well. Third, the set $C_{\gamma,\gamma
k_{\beta}}^{Q,R}$ (which is the only set $C_{\alpha,l}^{Q,R}$ which really
differs from $C_{\alpha,l}^{Q}$) appears, in the expression involving $Q$ or
$C_{\gamma,\gamma k_{\beta}-1}$, to the power 0 (as $a_{\beta,\gamma}(\gamma
k_{\beta}-ns_{\gamma,k_{\beta}})=n-1$). Hence the superscript $R$ can be omitted
from the notation there as well. Corollary 4.3 of \cite{[FZ]} now allows us,
when considering the total product, to add $R$ and $Q$ to the appropriate sets.
Let $C_{\gamma,\gamma k_{\beta}-1}^{+Q}$ denote the set $C_{\gamma,\gamma
k_{\beta}-1}\cup\{Q\}$, and then the product of the denominator appearing under
$\theta^{2en^{2}}[Q,\Delta](0,\Pi)$ in Equation \eqref{1strel} and the
correction terms considered above equals
\[q_{\Delta}^{Q,\gamma}=\prod_{(\alpha,l)\neq(\gamma,\gamma
k_{\beta}-ns_{\gamma,k_{\beta}})}[C_{\gamma,\gamma
k_{\beta}},C_{\alpha,l}^{Q}]^{ena_{\beta,\alpha}(l)}\cdot[C_{\gamma,\gamma
k_{\beta}},C_{\gamma,\gamma k_{\beta}}]^{en(n-1)}\times\]
\[\times\!\!\prod_{(\alpha,l)\neq(\gamma,\gamma
k_{\beta}-1-ns_{\gamma,k_{\beta}})}\![C_{\gamma,\gamma
k_{\beta}-1}^{+Q},C_{\alpha,l}^{Q}]^{en[n-1-a_{\beta,\alpha}(l)]}\cdot[C_{\gamma
,\gamma k_{\beta}-1}^{+Q},C_{\gamma,\gamma k_{\beta}-1}^{+Q}]^{en(n-1)}.\] This
is the required form of the denominator under
$\theta^{2en^{2}}[Q,\Delta](0,\Pi)$ which depends on $\gamma$ but no longer on
$R \in C_{\gamma,\gamma k_{\beta}}$. As we preserved the symmetry of yielding
the same equation from $\Delta$ and from $T_{Q,R}(\Delta)$, we have established

\begin{prop}
The quotient $\frac{\theta^{2en^{2}}[Q,\Delta](0,\Pi)}{q_{\Delta}^{Q,\gamma}}$
is invariant under the operators $T_{Q,R}$ for all admissible branch points $R$
of the form $P_{\gamma,j}$. \label{PMTgamma}
\end{prop}

As $a_{\beta,\gamma}$ takes $\gamma k_{\beta}-1-ns_{\gamma,k_{\beta}}$ to 0 and
$\gamma k_{\beta}-ns_{\gamma,k_{\beta}}$ to $n-1$, the power to which the
exceptional sets $C_{\gamma,\gamma k_{\beta}-1}$ and $C_{\gamma,\gamma
k_{\beta}}$ appear in in $q_{\Delta}^{Q,\gamma}$ is determined by the same rule
as the other sets. Observe that the expression $q_{\Delta}^{Q,\gamma}$ looks
like the original denominator from Equation \eqref{1strel} but with $R$ and $Q$
replaced by $C_{\gamma,\gamma k_{\beta}}$ and $C_{\gamma,\gamma
k_{\beta}-1}^{+Q}$ respectively, and with some superscripts omitted. In our
example it will be expedient to separate the terms with $l=1$ of the first
product and those with $l=2$ in the second product (we shall in fact carry out
a similar separation in the general setting). As we know that $Q$ lies in
$C_{0}$, many additional superscripts may be omitted, and the corresponding
denominator $q_{\Delta}^{Q,\gamma}$ equals
\[\prod_{l=2}^{n-1}\big([C_{\gamma,\gamma-1}^{+Q},C_{l}][C_{\gamma,\gamma},D_{l}
]\big)^{2n(l-1)}\big([C_{\gamma,\gamma-1}^{+Q},D_{l}][C_{\gamma,\gamma},C_{l}]
\big)^{2n(n-l)}\times\]
\[\times\prod_{l=3}^{n-1}\big([C_{\gamma,\gamma-1}^{+Q},E_{l}][C_{\gamma,\gamma}
,F_{l}]\big)^{2n(l-2)}\big([C_{\gamma,\gamma-1}^{+Q},F_{l}][C_{\gamma,\gamma},E_
{l}]\big)^{2n(n+1-l)}\times\]
\[\times\big([C_{\gamma,\gamma-1}^{+Q},D_{1}][C_{\gamma,\gamma-1}^{+Q},F_{2}][C_
{\gamma,\gamma},C_{1}][C_{\gamma,\gamma},E_{2}]\big)^{2n(n-1)}\times\]
\[\times\big([C_{\gamma,\gamma-1}^{+Q},C_{0}^{Q}][C_{\gamma,\gamma-1}^{+Q},E_{1}
][C_{\gamma,\gamma},D_{0}][C_{\gamma,\gamma},F_{1}]\big)^{2n(n-1)}\times\]
\[\times\big([C_{\gamma,\gamma-1}^{+Q},F_{0}][C_{\gamma,\gamma},E_{0}]\big)^{2n}
\big([C_{\gamma,\gamma-1}^{+Q},E_{0}][C_{\gamma,\gamma},F_{0}]\big)^{2n(n-2)}.\]
We did not separate the expressions $[C_{\gamma,\gamma},C_{\gamma,\gamma}]$ and
$[C_{\gamma,\gamma-1}^{+Q},C_{\gamma,\gamma-1}^{+Q}]$ so that we could cover all
the different possibilities of $\gamma$ simultaneously: Note that the first
expression involves one of the sets $C_{1}$, $E_{2}$, $F_{1}$, or $D_{0}$ while
the second one is based on $C_{0}$ (which already includes $Q$ in our notation),
$E_{1}^{+Q}$, $F_{2}^{+Q}$, or $D_{1}^{+Q}$ (recall the index is reversed in the
last two sets!), and these expressions indeed appear to the power $2n(n-1)$. It
follows that Proposition \ref{PMTgamma} generalizes Propositions 4.4, 5.1 and
5.2 of \cite{[FZ]}, where the sets with superscript $+Q$ are denoted $C_{-1}$
for $\gamma=\beta=1$ and $H$ for $\gamma=-1$, as is easily seen by omitting all
the expressions involving some set $E_{l}$ or $F_{l}$ (as well as $D_{l}$ for
the former). We will ultimately express our formulae, in the general case, in
terms of the set $C_{\beta,0}^{+Q}$, which corresponds to the divisor
$Q^{n-1}\Delta$ used for changing the base point below (this is reflected in the
notation chosen for the example). We also remark that the fact that only
products of the form $\alpha k_{\beta}$ show up in our operators and
denominators is not coincidental. Indeed, by replacing $w$ by $w^{k}$ (divided
by the appropriate polynomial in $z$) for some $k$ which is prime to $n$, all
the indices $\alpha$, $\beta$, etc. are divided by $k$ modulo $n$, so that only
such products are independent of the choice of the generator $w$ of
$\mathbb{C}(X)$ over $\mathbb{C}(z)$.

\smallskip

We can now prove the Poor Man's Thomae (PMT) for $X$. Recall that the PMT is a
formula which attaches, given a branch point $Q$ as base point, an expression
$g_{\Delta}^{Q}$ to every non-special divisor $\Delta$ supported on the branch
points distinct from $Q$, such that the quotient
$\frac{\theta^{2en^{2}}[Q,\Delta](0,\Pi)}{g_{\Delta}^{Q}}$ remains invariant
under all the operators $T_{Q,R}$ for admissible $R$. Our aim is to multiply
$q_{\Delta}^{Q,\gamma}$ (hence divide Equation \eqref{1strel} further) by an
expression which is invariant under all the operators $T_{Q,R}$ with $R \in
C_{\gamma,\gamma k_{\beta}}$, and obtain an expression which is independent of
$\gamma$ as well. In order to do so, we shall use the following generalization
of Lemma 4.2 of \cite{[FZ]}:

\begin{lem}
Assume the set $Y$ is the union of the finite sets $Z_{j}$, $1 \leq j \leq d$,
and let $W$ be a finite set which is disjoint from $Y$. Then $[Y,W]$ is the
product $\prod_{j=1}^{d}[Z_{j},W]$ up to sign, and $[Y,Y]$ equals
$\prod_{j=1}^{d}[Z_{j},Z_{j}]\cdot\prod_{1 \leq i<j \leq d}[Z_{i},Z_{j}]$ up to
sign. \label{undec}
\end{lem}

\begin{proof}
The first assertion is clear. We prove the second assertion by induction. For
$d=2$ this is just Lemma 4.2 of \cite{[FZ]}. Assume that the assertion holds for
$d-1$. Considering the expressions with $Z_{d}$ and $Z_{d-1}$ we claim that
we can replace the product $\prod_{j=1}^{d}[Z_{j},Z_{j}]\cdot\prod_{1 \leq i<j
\leq d}[Z_{i},Z_{j}]$ by $\prod_{j=1}^{d-1}[U_{j},U_{j}]\cdot\prod_{1 \leq i<j
\leq d-1}[U_{i},U_{j}]$ (up to sign), where $U_{j}=Z_{j}$ for $1 \leq j \leq
d-2$ and $U_{d-1}=Z_{d-1} \cup Z_{d}$. Indeed, apply Lemma 4.2 of \cite{[FZ]} to
$[Z_{d-1},Z_{d-1}][Z_{d-1},Z_{d}][Z_{d},Z_{d}]$, and the first assertion here
establishes the claim. The induction hypothesis now completes the proof of the
lemma.
\end{proof}

By taking even powers of the expressions appearing in Lemma \ref{undec} we
obtain exact equalities there.

This time we consider our example first. The denominator $q_{\Delta}^{Q,\gamma}$
involves two products over $l$ as well as an expression with powers $2n$ and
$2n(n-2)$, which involves brackets of one set from $C_{\gamma,\gamma}$ or
$C_{\gamma,\gamma-1}^{+Q}$ (which were seen to be, up to adding or removing $Q$,
two of the sets $C_{0}$, $C_{1}$, $E_{1}$, $E_{2}$, $F_{1}$, $F_{2}$, $D_{0}$,
or $D_{1}$), and another set which is not one of these 8 sets. Hence it makes
sense to multiply $q_{\Delta}^{Q,\gamma}$ by the expressions of this form which
appear in $q_{\Delta}^{Q,\delta}$ for $\delta\neq\gamma$, but without adding $Q$
to $C_{\delta,\delta-1}$. For example, $q_{\Delta}^{Q,1}$ will be multiplied by
\[\big([D_{1},E_{0}][D_{0},F_{0}]\big)^{2n(n-2)}\prod_{l=2}^{n-1}\big([D_{1},C_{
l}][D_{0},D_{l}]\big)^{2n(l-1)}\big([D_{1},D_{l}][D_{0},C_{l}]\big)^{2n(n-l)}
\times\]
\[\times\big([D_{1},F_{0}][D_{0},E_{0}]\big)^{2n}\prod_{l=3}^{n-1}\big([D_{1},E_
{l}][D_{0},F_{l}]\big)^{2n(l-2)}\big([D_{1},F_{l}][D_{0},E_{l}]\big)^{2n(n+1-l)}
\] from $\delta=n-1$, as well as the same expression with $D_{1}$ replaced by
$E_{1}$ (resp. $F_{2}$) and $D_{0}$ replaced by $E_{2}$ (resp. $F_{1}$) arising
from $\delta=2$ (resp. $\delta=n-2$). It will be expedient to denote the unions
$C_{0} \cup E_{1} \cup F_{2} \cup D_{1}$ by $G$ (recall that $Q \in G$) and
$C_{1} \cup E_{2} \cup F_{1} \cup D_{0}$ by $H$. Lemma \ref{undec} now shows
that the product of the three expressions by which we multiply
$q_{\Delta}^{Q,\gamma}$ becomes the one from the last formula, in which $G
\setminus C_{\gamma,\gamma-1}^{+Q}$ stands in the place of $D_{1}$ while $H
\setminus C_{\gamma,\gamma}$ appears in the place of $D_{0}$. We claim that
these expressions are invariant under $T_{Q,R}$ for any $R=P_{\gamma,j}$.
Indeed, evaluating the functions $b_{1,\alpha}$ and remembering the inversion of
the index in the sets $D_{l}$ and $F_{l}$ shows that these operators interchange
$D_{1}$ with $D_{0}$, $E_{1}$ with $E_{2}$, $F_{1}$ with $F_{2}$, while if
$\gamma\neq1$ then $C_{0}^{Q}$ with $C_{1}$ are interchanged (instead of the
corresponding $C_{\gamma,\gamma}$ or $C_{\gamma,\gamma-1}$). Thus this
operation interchanges $G \setminus C_{\gamma,\gamma-1}^{+Q}$ stands in the
place of $D_{1}$ while $H \setminus C_{\gamma,\gamma}$ ($Q$ is omitted by
definition, and $R$ is taken out by the choice of $\gamma$). As similar
considerations show that these operators interchange $C_{l}$ with $C_{n+1-l}$
and $D_{l}$ with $D_{n+1-l}$ for $l\geq2$, $E_{l}$ with $E_{n+3-l}$ and $F_{l}$
with $F_{n+3-l}$ for $l\geq4$, and also $E_{0}$ with $E_{3}$ and $F_{0}$ with
$F_{3}$, the invariance assertion follows.

Now, the first assertion of Lemma \ref{undec} shows that the remaining two lines
in the denominator $q_{\Delta}^{Q,\gamma}$ give the expressions
$[C_{\gamma,\gamma-1}^{+Q},C_{\gamma,\gamma-1}][C_{\gamma,\gamma-1}^{+Q},G
\setminus C_{\gamma,\gamma-1}^{+Q}]$ (and we may add $Q$ to the second
occurrence of $C_{\gamma,\gamma-1}$ without changing the value of this
expression) and $[C_{\gamma,\gamma},C_{\gamma,\gamma}][C_{\gamma,\gamma},H
\setminus C_{\gamma,\gamma}]$ raised to the power $2n(n-1)$: E.g., for
$\gamma=1$ we get these powers of $[C_{0},C_{0}]$, $[C_{0},E_{1} \cup F_{2} \cup
D_{1}]$, $[C_{1},C_{1}]$, and $[C_{1},E_{2} \cup F_{1} \cup D_{0}]$. We thus
multiply $q_{\Delta}^{Q,\gamma}$ further by the two expressions
$[G \setminus C_{\gamma,\gamma-1}^{+Q},G \setminus
C_{\gamma,\gamma-1}^{+Q}]^{2n(n-1)}$ and $[H \setminus C_{\gamma,\gamma},H
\setminus C_{\gamma,\gamma}]^{2n(n-1)}$, and their product is also invariant
under all the operators $T_{Q,R}$ with $R$ having first index $\gamma$ by the
precious paragraph. The total product is now seen, using Lemma \ref{undec}, to
be
\[\prod_{l=2}^{n-1}\big([G,C_{l}][H,D_{l}]\big)^{2n(l-1)}\big([G,D_{l}][H,C_{l}]
\big)^{2n(n-l)}\times\]
\[\times\prod_{l=3}^{n-1}\big([G,E_{l}][H,F_{l}]\big)^{2n(l-2)}\big([G,F_{l}][H,
E_{l}]\big)^{2n(n+1-l)}\times\]
\[\times\big([G,E_{0}][H,F_{0}]\big)^{2n(n-2)}\big([G,F_{0}][H,E_{0}]\big)^{2n}
\big([G,G][H,H]\big)^{2n(n-1)}.\] As this product does not depend on $\gamma$
anymore, putting it under the corresponding theta constant yields an expression
which is invariant under the operators $T_{Q,R}$ for all admissible $R$ (i.e.,
for all $R$ from $H$, regardless of the first index $\gamma$). This is the PMT
in our example.

\smallskip

We now carry out these considerations in the general case. As $a_{\beta,\alpha}$
takes $\alpha k_{\beta}-ns_{\alpha,k_{\beta}}$ to $n-1$ and $\alpha
k_{\beta}-ns_{\alpha,k_{\beta}}-1$ to 0, we denote the set of two indices
$\alpha k_{\beta}-ns_{\alpha,k_{\beta}}$ and $\alpha
k_{\beta}-1-ns_{\alpha,k_{\beta}}$ from $\mathbb{N}_{n}$ by $U_{\beta,\alpha}$,
and decompose the expression for (the general) $q_{\Delta}^{Q,\gamma}$ as
\[\prod_{\alpha}\prod_{l \not\in U_{\beta,\alpha}}[C_{\gamma,\gamma
k_{\beta}},C_{\alpha,l}^{Q}]^{ena_{\beta,\alpha}(l)}[C_{\gamma,\gamma
k_{\beta}-1}^{+Q},C_{\alpha,l}^{Q}]^{en[n-1-a_{\beta,\alpha}(l)]}\times\]
\[\times\prod_{\alpha\neq\gamma}[C_{\gamma,\gamma k_{\beta}},C_{\alpha,\alpha
k_{\beta}}^{Q}]^{en(n-1)}[C_{\gamma,\gamma k_{\beta}-1}^{+Q},C_{\alpha,\alpha
k_{\beta}-1}^{Q}]^{en(n-1)}\times\] \[\times[C_{\gamma,\gamma
k_{\beta}},C_{\gamma,\gamma k_{\beta}}]^{en(n-1)}[C_{\gamma,\gamma
k_{\beta}-1}^{+Q},C_{\gamma,\gamma k_{\beta}-1}^{+Q}]^{en(n-1)}.\] The fact
that $Q$ lies in $C_{\beta,0}$, $0 \in U_{\beta,\beta}$, and $C_{\beta,\beta
k_{\beta}}=C_{\beta,1} \neq C_{\beta,0}$ allows us to remove the superscript
$Q$ from all the sets apart from $C_{\alpha,\alpha k_{\beta}-1}$ in this
expression. Denote the set $\bigcup_{\delta}C_{\delta,\delta
k_{\beta}-1}\cup\{Q\}$ by $E_{Q}$ and the set $\bigcup_{\delta}C_{\delta,\delta
k_{\beta}}$ by $F^{\beta}$ (the index change will turn out useful later). Lemma
\ref{undec} shows that the product over $\alpha\neq\gamma$ is the product of
$[C_{\gamma,\gamma k_{\beta}},F^{\beta} \setminus C_{\gamma,\gamma k_{\beta}}]$
and $[C_{\gamma,\gamma k_{\beta}-1}^{+Q},E_{Q} \setminus C_{\gamma,\gamma
k_{\beta}-1}^{+Q}]$ raised to the power $en(n-1)$. We wish to
multiply $q_{\Delta}^{Q,\gamma}$ by the expression
\[\prod_{\alpha}\prod_{l \not\in U_{\beta,\alpha}}[F^{\beta} \setminus
C_{\gamma,\gamma k_{\beta}},C_{\alpha,l}]^{ena_{\beta,\alpha}(l)}[E_{Q}
\setminus C_{\gamma,\gamma
k_{\beta}-1}^{+Q},C_{\alpha,l}]^{en[n-1-a_{\beta,\alpha}(l)]}\times\]
\[\times[F^{\beta} \setminus C_{\gamma,\gamma k_{\beta}},F^{\beta} \setminus
C_{\gamma,\gamma k_{\beta}}]^{en(n-1)}[E_{Q} \setminus C_{\gamma,\gamma
k_{\beta}-1}^{+Q},E_{Q} \setminus C_{\gamma,\gamma
k_{\beta}-1}^{+Q}]^{en(n-1)}.\] As $T_{Q,R}$ with $R=P_{\gamma,j}$ interchanges
$C_{\delta,\delta k_{\beta}}$ with $C_{\delta,\delta k_{\beta}-1}^{Q}$ for all
$\delta\neq\gamma$, it also interchanges the sets $E_{Q} \setminus
C_{\gamma,\gamma k_{\beta}-1}^{+Q}$ and $F^{\beta} \setminus C_{\gamma,\gamma
k_{\beta}}$. The considerations regarding the sets
$\widetilde{C}_{\alpha,l}^{Q,R}$ above complete the proof of the invariance of
the latter expression under $T_{Q,R}$ for all $R \in C_{\gamma,k_{\beta}}$, as
we have already seen that the superscripts can be ignored for $l \not\in
U_{\beta,\alpha}$. The total product becomes (Lemma \ref{undec} again)
\[\big([F^{\beta},F^{\beta}][E_{Q},E_{Q}]\big)^{en(n-1)}\prod_{\alpha}\prod_{l
\not\in U_{\beta,\alpha}}[F^{\beta},C_{\alpha,l}]^{ena_{\beta,\alpha}(l)}[E_{Q},
C_{ \alpha,l}]^{en[n-1-a_{\beta,\alpha}(l)]},\] an expression which we denote
$g_{\Delta}^{Q}$ since it depends on $Q$ but no longer on $\gamma$. Proposition
\ref{PMTgamma}, the fact that our multiplier was invariant under $T_{Q,R}$ for
$R \in C_{\gamma,k_{\beta}}$, and the independence of $g_{\Delta}^{Q}$ on
$\gamma$ combine to prove
\begin{prop}
The quotient $\frac{\theta^{2en^{2}}[Q,\Delta](0,\Pi)}{g_{\Delta}^{Q}}$ is
invariant under all the admissible operators $T_{Q,R}$, and it is the PMT of the
$Z_{n}$ curve $X$. \label{PMTQ}
\end{prop}

One can check that the PMT appearing in Propositions 4.4, 5.3, and 6.7 of
\cite{[FZ]} are special cases of Proposition \ref{PMTQ}, except that the
isolated divisor $P_{3}^{n-1}$ of Section 6.1 of \cite{[FZ]} (on which no
$T_{P_{0},R}$ can act) is now given the denominator
$(\lambda_{0}-\lambda_{1})^{2n(n-3)}(\lambda_{0}-\lambda_{2})^{2n}(\lambda_{0}
-\lambda_{3})^{2n}$ rather than 1. As for Propositions 6.17 and 6.19 of that
reference, our formula for $g_{\Delta}^{Q}$ multiplies the expression given
there for the divisor $P_{i}^{2s}P_{j}^{s}$ for $t=1$ (resp.
$P_{i}^{2s+1}P_{j}^{s}$ for $t=2$) by the $ens$th power of the
$T_{P_{0},P_{i}}$-invariant (resp. $T_{P_{0},P_{j}}$-invariant) expressions
$(\lambda-\lambda_{j})(\lambda-\lambda_{k})$ and $(\lambda_{i}-\lambda_{j}
)(\lambda_{i}-\lambda_{k})$ (resp.
$(\lambda-\lambda_{i})(\lambda-\lambda_{k})$ and $(\lambda_{j}-\lambda_{i}
)(\lambda_{j}-\lambda_{k})$). Hence our results are compatible also in these
cases.

As already remarked in Section 2.6 of \cite{[FZ]}, we can allow (full)
ramification at $\infty$ by assuming that $\sum_{\alpha}r_{\alpha}$ is prime to
$n$ in Equation \eqref{Zneq}. Then the integers $t_{k}$ from Equation
\eqref{divomegak} (which are no longer integers) have to be replaced by their
upper integral values. All our further results hold also in this setting, when
we omit any meaningless expression involving $\infty$. This assertion extends to
substitutions of $\infty$ in a rational function, since every such substitution
always yields the value 1. This remark applies for what follows as well.

We also observe that the formula for $g_{\Delta}^{Q}$ (as well as the preceding
expressions) is independent of the cardinality conditions on the set
$C_{\alpha,l}$. Therefore the form of the Thomae formulae is unrelated to the
actual set of divisors needed in order to define the characteristics etc., but
is only based on the general shape of a divisor supported on the branch points
distinct from $Q$ containing no $n$th powers or higher. In particular, the
formulae are not connected to the question whether such divisors exist or not,
and one might say that they hold in a trivial manner in the latter case.

\medskip

We now turn to changing the base point $Q$ (but leave the index $\beta$ fixed).
Although this change is not required at this stage, it helps to simplify the
notation. In the proof of Proposition \ref{PMTgamma} we have encountered the
sets $C_{\gamma,\gamma k_{\beta}-1}^{+Q}$, namely $C_{\gamma,\gamma
k_{\beta}-1}\cup\{Q\}$, for various $\gamma$. Since $Q=P_{\beta,i}$, it is
natural to consider this set for $\gamma=\beta$, namely $C_{\beta,0}\cup\{Q\}$.
Omitting $Q$ from its original set $C_{\beta,n-1}$ (which stands for the fact
that $v_{Q}(\Delta)=0$) and including it in $C_{\beta,0}$ (the set of points
$P_{\beta,m}$ appearing to the power $n-1$ in $\Delta$) corresponds to replacing
$\Delta$ by the divisor $\Xi=Q^{n-1}\Delta$ of degree $g+n-1$ (this is already
done in the notation chosen for the example). This is the divisor appearing in
the symmetric notation of the Thomae formulae in Chapters 3, 4, and 5 of
\cite{[FZ]}, since the second statement in Corollary 1.13 there implies that for
such divisors the element $\varphi_{P}(\Xi)+K_{P}$ of $J(X)$ is independent of
the choice of the branch point $P$. Its value coincides with
$\varphi_{Q}(\Delta)+K_{Q}$, as is easily seen by taking $Q=P$. We denote the
appropriate theta constant $\theta[\Xi](0,\Pi)$ (with no need to add the base
point), and it coincides with $\theta[Q,\Delta](0,\Pi)$. Taking $D_{\alpha,l}$
to be $C_{\beta,0}\cup\{Q\}$ if $\alpha=\beta$ and $l=0$ and $C_{\alpha,l}^{Q}$
otherwise, we obtain from Equation \eqref{Delta} that
\begin{equation}
\Xi=Q^{n-1}\Delta=\prod_{\alpha}\prod_{l=0}^{n-1}D_{\alpha,l}^{n-1-l}.
\label{Xi}
\end{equation}
Moreover, for every $1 \leq k \leq n-1$ the value $l=n-1$ does not participate
in the summation over $0 \leq l\leq\beta k-ns_{\beta,k}-1$, while the value
$l=0$ does participate in this summation. It follows that the point $Q$ does not
contribute to any of the cardinalities appearing in Theorem \ref{nonsp}, but
after replacing every set $C_{\alpha,l}$ by $D_{\alpha,l}$ it contributes to all
of them. Therefore the divisors of degree $g+n-1$ in which we are interested are
those which take the form of Equation \eqref{Xi} with the sets $D_{\alpha,l}$
satisfying \[\sum_{\alpha}\sum_{l=0}^{\alpha
k-ns_{\alpha,k}-1}|D_{\alpha,l}|=t_{k}\] for every $1 \leq k \leq n-1$ (these
are also more closely related to the cardinality conditions stated in our
example).

After multiplying the images of the operators from Proposition \ref{operZn} by
$Q^{n-1}$ as well and using the fact that $a_{\beta,\alpha}$ and
$b_{\beta,\alpha}$ are involutions, we can write these operators in terms of the
divisors $\Xi$ as
\begin{equation}
N_{\beta}(\Xi)=\prod_{\alpha,l}D_{\alpha,l}^{n-1-a_{\beta,\alpha}(l)}=\prod_{
\alpha,l}D_{\alpha,a_{\beta,\alpha}(l)}^{n-1-l} \label{Nbetaeq}
\end{equation}
(the notation $N_{\beta}$, rather than $N_{Q}$, can be used here since the
effect of this operator depends only on $\beta$ and not on the choice of $Q \in
D_{\beta,0}$) and
\begin{equation}
T_{Q,R}(\Xi)=Q\prod_{\alpha,l}D_{\alpha,l}^{n-1-b_{\beta,\alpha}(l)}/R=Q\prod_{
\alpha,l}D_{\alpha,b_{\beta,\alpha}(l)}^{n-1-l}/R \label{TQRXieq}
\end{equation}
for $Q \in D_{\beta,0}$ and $R \in D_{\gamma,\gamma k_{\beta}}$. It follows that
the set $E_{Q}$ appearing in $g_{\Delta}^{Q}$ is simply
$\bigcup_{\delta}D_{\delta,\delta k_{\beta}-1}$ (since for $\delta=\beta$ we
already have $Q \in D_{\beta,0}$). This set depends on $\beta$, but no longer on
$Q \in D_{\beta,0}$, just like $F^{\beta}=\bigcup_{\delta}D_{\delta,\delta
k_{\beta}}$. In total $g_{\Delta}^{Q}$ does not depend on $Q \in D_{\beta,0}$ in
this setting, so we denote it $g_{\Xi}^{\beta}$. Furthermore, as
$a_{\beta,\alpha}(\alpha k_{\beta}-ns_{\alpha,k_{\beta}})=n-1$ and
$a_{\beta,\alpha}(\alpha k_{\beta}-1-ns_{\alpha,k_{\beta}})=0$, the power
$en(n-1)$ to which the expressions $[D_{\delta,\delta
k_{\beta}},D_{\alpha,\alpha k_{\beta}}]$ and $[D_{\delta,\delta
k_{\beta}-1},D_{\alpha,\alpha k_{\beta}-1}]$ (coming from
$[F^{\beta},F^{\beta}]$ or $[E_{Q},E_{Q}]$ respectively) appear in
$g_{\Xi}^{\beta}$ obeys the same rule as with the other expressions
$[D_{\delta,\delta k_{\beta}},D_{\alpha,l}]$ or $[D_{\delta,\delta
k_{\beta}-1},D_{\alpha,l}]$. This also shows that the other expressions arising
from the second value of $l$ in $U_{\beta,\alpha}$ come with the power 0, hence
can be trivially added (for the simplicity of the notation). Expanding the
products using Lemma \ref{undec} we can write
\[g_{\Xi}^{\beta}=\prod_{\{(\delta,\alpha,l)|\delta\leq\alpha\mathrm{\
if\ }l=\alpha k_{\beta}-ns_{\alpha,k_{\beta}}\}}[D_{\delta,\delta
k_{\beta}},D_{\alpha,l}]^{ena_{\beta,\alpha}(l)}\times\]
\[\times\prod_{\{(\delta,\alpha,l)|\delta\leq\alpha\mathrm{\ if\
}l=\alpha k_{\beta}-1-ns_{\alpha,k_{\beta}}\}}[D_{\delta,\delta
k_{\beta}-1},D_{\alpha,l}]^{en[n-1-a_{\beta,\alpha}(l)]}\] (the condition
$\delta\leq\alpha$ for the appropriate value of $l$ is imposed to avoid
undesired repetitions), and Proposition \ref{PMTQ} takes the form

\begin{prop}
The quotient $\frac{\theta^{2en^{2}}[\Xi](0,\Pi)}{g_{\Xi}^{\beta}}$ is invariant
under all the operators $T_{Q,R}$ with $Q \in D_{\beta,0}$ and $R \in
D_{\gamma,\gamma k_{\beta}}$ (with arbitrary $\gamma$), and it is the
base-point-invariant form of the PMT of $X$. \label{PMTbpinv}
\end{prop}

We remark that a divisor $\Xi$ takes the form $Q^{n-1}\Delta$ for some integral
divisor $\Delta$ of degree $g$ and some base point $Q$ only if some set
$D_{\beta,0}$ is not empty. In general, however, this condition might not be
satisfied, and there exist divisors $\Xi$ satisfying the cardinality conditions
such that $D_{\beta,0}=\emptyset$ for all $\beta$. These divisors cannot be
presented as $Q^{n-1}\Delta$ for any branch point $Q$. The operators $N_{\beta}$
act on these divisors, but no $T_{Q,R}$ does so since $Q \in D_{\beta,0}$ (for
the appropriate $\beta$) is required to define the action of these operators.
Hence the assertion of Proposition \ref{PMTbpinv} holds trivially for these
divisors, at least at this point. More details will be given in Section
\ref{Trans}.

\section{Invariance Under $N_{\beta}$ \label{NInv}}

Consider a $Z_{n}$ curve $X$, a branch point $Q$ on $X$, and a non-special
divisor $\Delta$ of degree $g$ on $X$ which is supported on the branch points
distinct from $Q$. The combination of Equation (1.5) of \cite{[FZ]} and Equation
\eqref{NQeq} yields the equality
\[\theta^{N}[Q,\Delta](0,\Pi)=\theta^{N}[Q,N_{Q}(\Delta)](0,\Pi)\] for any $N$
divisible by $2n$. The condition $2n|N$ is necessary to ensure independence of
the lifts. Expressed in terms of the degree $g+n-1$ divisors $\Xi$, the latter
equality with $N=2en^{2}$ becomes
\begin{equation}
\theta^{2en^{2}}[N_{\beta}(\Xi)](0,\Pi)=\theta^{2en^{2}}[\Xi](0,\Pi),
\label{even}
\end{equation}
holding for every $\beta\in\mathbb{N}_{n}$ which is prime to $n$ and for every
divisor $\Xi$ of the form presented above. Hence our goal is to divide the
quotient from Proposition \ref{PMTbpinv} (or equivalently, multiply
$g_{\Xi}^{\beta}$) by an expression which is invariant under all the admissible
operators $T_{Q,R}$ considered in that Proposition, such that the product
$h_{\Xi}$ of $g_{\Xi}^{\beta}$ with this expression will satisfy
$h_{N_{\beta}(\Xi)}=h_{\Xi}$. In case the expression $h_{\Xi}$ is independent
also of $\beta$, the quotient $\frac{\theta^{2en^{2}}[\Xi](0,\Pi)}{h_{\Xi}}$
will be invariant under all the operators $T_{Q,R}$ as well as $N_{\beta}$ for
all $\beta$.

To achieve this goal, we need to compare $g_{\Xi}^{\beta}$ with
$g_{N_{\beta}(\Xi)}^{\beta}$. According to Equation \eqref{Nbetaeq}, moving from
$\Xi$ to $N_{\beta}(\Xi)$ is equivalent to replacing every set $D_{\alpha,l}$ by
$D_{\alpha,a_{\beta,\alpha}(l)}$. Now, $a_{\beta,\delta}(\delta
k_{\beta}-1-ns_{\delta,k_{\beta}})=0$ and $a_{\beta,\delta}(\delta
k_{\beta}-ns_{\delta,k_{\beta}})=n-1$, while $a_{\beta,\delta}$ is an
involution. These considerations imply that $g_{N_{\beta}(\Xi)}^{\beta}$ equals
\[\prod_{\{(\delta,\alpha,l)|\delta\leq\alpha\mathrm{\ if\
}l=n-1\}}[D_{\delta,n-1},D_{\alpha,l}]^{enl}\cdot\prod_{\{(\delta,\alpha,
l)|\delta\leq\alpha\mathrm{\ if\
}l=0\}}[D_{\delta,0},D_{\alpha,l}]^{en(n-1-l)}.\] Observe that this expression
does not depend on $\beta$, which suggests that we might take it as the
denominator under $\theta^{2en^{2}}[\Xi](0,\Pi)$ in the PMT using Equation
\eqref{even}. Nevertheless, we prefer to follow \cite{[FZ]} and maintain the
denominator $g_{\Xi}^{\beta}$.

\smallskip

Let us consider our example again. The operator $N_{Q}$ (or $N_{1}$ in the new
notation) takes $C_{l}$ and $D_{l}$ to the sets with index $n-l$, except for
the sets with $l=0$ which remain invariant. The sets $E_{l}$ and $F_{l}$ are
sent to those having index $n+1-l$, while the indices 0 and 1 are being
interchanged. Hence if the desired denominator $h_{\Xi}$ contains, for example,
the expression $[C_{l},E_{j}]$, then the expression $[C_{n-l},E_{n+1-j}]$ must
appear to the same power. On the other hand, the action of the operator
$T_{Q,R}$ with $Q \in C_{0}$ sends (up to the more delicate considerations
involving the points $Q$ and $R$ themselves) the sets $C_{l}$ and $D_{l}$ to
those having index $n+1-l$, and interchanges the sets with the indices 0 and 1.
As for the sets $E_{l}$ and $F_{l}$, here the images have index $n+3-l$, while
the indices 0 and 3 as well as 1 and 2 are being interchanged. This shows that
if the quotient $\frac{h_{\Xi}}{g_{\Xi}^{\beta}}$ involves, e.g., an expression
of the sort $[D_{l},F_{j}]$, then $[D_{n+1-l},F_{n+3-j}]$ must appear in this
quotient to the same power.

We define $h_{\Xi}$ as a product of three expressions (the ones appearing in
Equations \eqref{hXiexCD}, \eqref{hXiexEF}, and \eqref{hXiexmix} below). We 
construct the first part of our expression $h_{\Xi}$ like in Sections 4.3 and 
5.3 of \cite{[FZ]}, and the remaining parts using the lead of Section 6.1 of
that reference. We start with the parts involving only the sets $C_{l}$ and
$D_{l}$. Consider the expression 
\begin{equation}
\prod_{l=0}^{n-1}\prod_{r=0}^{n-1-l}\big([C_{r},C_{r+l}][D_{r},D_{r+l}]\big)^{
2n\big[\frac{n^{2}-1}{4}-l(n-l)\big]}\big([C_{r},D_{r+l}][D_{r},C_{r+l}]\big)^{
2nl(n-l)}. \label{hXiexCD}
\end{equation}
When $N_{1}$ is applied to it, every pair $(r,l)$ is taken (by what we said
above) to $(n-r-l,l)$ if $r\geq1$, to $(0,n-l)$ if $r=0$ but $l\geq1$, and
remains $(0,0)$ in the case $r=l=0$. Hence the $N_{1}$-invariance of this
expression follows from the fact that the powers depend only on the index $l$
and are invariant under taking $l$ to $n-l$. Furthermore, the parts in
which $r\geq2$ does not involve the points $Q$ or $R$ for admissible operators
$T_{Q,R}$, and the action of these operators takes the pair $(r,l)$ to
$(n+1-r-l,l)$. Hence this part is invariant under these operators too. Expand
the expression for $g_{\Xi}^{1}$ according to the definition of $G$ and $H$
using Lemma \ref{undec}, and consider only the parts involving just sets $C_{l}$
and $D_{l}$. As it involves only expressions in which one index is either 1 or
0, the part with $r\geq2$ considered above remains the same in
$\frac{h_{\Xi}}{g_{\Xi}^{1}}$, and its invariance under the operators $T_{Q,R}$
has already been established. Now, for $l\geq2$ the expressions $[C_{0},C_{l}]$
and $[D_{0},D_{l}]$ appear to the power
$2n\big[\frac{n^{2}-1}{4}-l(n-l)-(l-1)\big]$ in the quotient, and the power
$2n\big[\frac{n^{2}-1}{4}-(l-1)(n+1-l)-(n-l)\big]$ to which $[C_{1},C_{l}]$ and
$[D_{1},D_{l}]$ appear coincides with that value, which equals
$2n\big[\frac{n^{2}+3}{4}-l(n+1-l)\big]$. The power to which $[C_{0},D_{l}]$ and
$[D_{0},C_{l}]$ appear in this quotient is $2n[l(n-l)-(n-l)]$, or just
$2n(l-1)(n-l)$, and this value coincides with the power $2n[(l-1)(n+1-l)-(l-1)]$
to which $[C_{1},D_{l}]$ and $[D_{1},C_{l}]$ appear there. The terms
$[C_{0},C_{0}]$, $[D_{0},D_{0}]$, $[C_{1},C_{1}]$, and $[D_{1},D_{1}]$ all
appear in this quotient when raised to the power
$2n\big[\frac{n^{2}-1}{4}-(n-1)\big]$, which is the same power to which
$[C_{0},C_{1}]$ and $[D_{0},D_{1}]$ appear there. Finally, $[C_{0},D_{0}]$ and
$[C_{1},D_{1}]$ do not appear in $h_{\Xi}$ at all, and $[C_{0},D_{1}]$ and
$[D_{0},C_{1}]$ cancel in the quotient. This shows, together with Lemma
\ref{undec}, that the part of the quotient $\frac{h_{\Xi}}{g_{\Xi}^{1}}$ under
consideration is based on the unions $C_{0} \cup C_{1}$ and $D_{0} \cup D_{1}$
(which are invariant under the operators $T_{Q,R}$), and the powers to which
$C_{l}$ and $D_{l}$ with $l\geq2$ are invariant under replacing $l$ with
$n+1-l$. Hence this part of $\frac{h_{\Xi}}{g_{\Xi}^{1}}$ is invariant under all
these operators.

Next we turn to those expressions which involve only the sets $E_{l}$ and
$F_{l}$. The expression we take here is very similar, and equals
\begin{equation}
\prod_{l=0}^{n-1}\prod_{r=0}^{n-1-l}\big([E_{r},E_{r+l}][F_{r},F_{r+l}]\big)^{
2n\big[\frac{n^{2}-1}{4}-l(n-l)\big]}\big([E_{r},F_{r+l}][F_{r},E_{r+l}]\big)^{
2nl(n-l)}. \label{hXiexEF} 
\end{equation}
In this case a pair $(r,l)$ with $r\geq2$ is taken by $N_{1}$ to $(n+1-r-l,l)$,
the pairs $(1,l)$ with $l\geq1$ and $(0,l)$ with $l\geq2$ are mapped to
$(0,n-l)$ and $(1,n-l)$ respectively, and the pair $(0,1)$ remains invariant
while the pairs $(0,0)$ and $(1,0)$ are being interchanged. Hence the
$N_{1}$-invariance of this expression follows from the same argument as above.
The operators $T_{Q,R}$ take the pairs $(r,l)$ with $r\geq4$ (which do not see
the points $Q$ or $R$) to $(n+3-r-l,l)$ yields the invariance of this part
under these operators as well. Further pairs which are not affected by the
special points $Q$ and $R$ are $(0,0)$ and $(3,0)$ (which are interchanged) as
well as $(0,3)$ (which remains invariant), $(0,l)$ for $l\geq4$, and $(3,l)$
for $l\geq1$. As these pairs are sent to $(3,n-l)$ and $(0,n-l)$ respectively,
the invariance of the product of the expressions corresponding to these pairs
is also established. Moreover, as all these expressions do not involve any of
the subsets $E_{1}$, $E_{2}$, $F_{1}$, or $F_{2}$ of $G$ and $H$, these parts
are not affected by the division by $g_{\Xi}^{1}$. Now, for any $l\geq3$ the
power $2n\big[\frac{n^{2}-1}{4}-(l-1)(n+1-l)-(l-2)\big]$ to which
$[E_{1},E_{l}]$ and $[F_{1},F_{l}]$ appear in the quotient
$\frac{h_{\Xi}}{g_{\Xi}^{1}}$ coincides with the power
$2n\big[\frac{n^{2}-1}{4}-(l-2)(n+2-l)-(n+1-l)\big]$ to which $[E_{2},E_{l}]$
and $[F_{2},F_{l}]$ are raised there, and this common value is
$2n\big[\frac{n^{2}+3}{4}-(l-1)(n+2-l)\big]$. In addition, $[E_{0},E_{1}]$ and
$[F_{0},F_{1}]$ appear to the power
$2n\big[\frac{n^{2}-1}{4}-(n-1)-(n-2)\big]$, $[E_{0},E_{2}]$ and
$[F_{0},F_{2}]$ carry the power $2n\big[\frac{n^{2}-1}{4}-2(n-2)-1\big]$, and
this common value, which we write as $2n\big[\frac{n^{2}+3}{4}-2(n-1)\big]$,
coincides with the value obtained by substituting $l=3$ in the previous
expression. Back to $l\geq3$, the two expressions $[E_{1},F_{l}]$ and
$[F_{1},E_{l}]$ appear in the quotient to the power $2n[(l-1)(n+1-l)-(n+1-l)]$,
$[E_{2},F_{l}]$ and $[F_{2},E_{l}]$ carry the power $2n[(l-2)(n+2-l)-(l-2)]$,
and this common value is $2n[(l-2)(n+1-l)]$. Furthermore, $[E_{0},F_{1}]$ and
$[F_{0},E_{1}]$ carry the power $2n[(n-1)-1]$, $[E_{0},F_{2}]$ and
$[F_{0},E_{2}]$ come with the power $2n[2(n-2)-(n-2)]$, both equal just
$2n(n-2)$, and this is the value which the latter expression attains on $l=3$.
Finally, all the expressions $[E_{1},E_{1}]$, $[F_{1},F_{1}]$, $[E_{2},E_{2}]$,
and $[F_{2},F_{2}]$ come raised to the power
$2n\big[\frac{n^{2}-1}{4}-(n-1)\big]$, $[E_{1},E_{2}]$ and $[F_{1},F_{2}]$
appear with the same power, $[E_{1},F_{1}]$ and $[E_{2},F_{2}]$ do not appear
at all, and the powers of $[E_{1},F_{2}]$ and $[F_{1},E_{2}]$ cancel in the
quotient. Hence this part can be written (Lemma \ref{undec} again) in terms of
the unions $E_{1} \cup E_{2}$ and $F_{1} \cup F_{2}$ (which are also invariant
under all the operators $T_{Q,R}$ with $Q \in C_{0}$), and the expressions
involving the other sets $E_{l}$ or $F_{l}$ come raised to powers which are
invariant under taking $l\geq4$ to $n+3-l$ and interchanging the values $l=0$
and $l=3$. This proves the invariance of this part of the quotient
$\frac{h_{\Xi}}{g_{\Xi}^{1}}$ as well.

It remains to take care of the parts in which one set is $C_{l}$ or $D_{l}$ and
the other one is $E_{l}$ or $F_{l}$. It turns out that the formulae depends on
$l$ when we take expressions of the form $[C_{r},E_{2r+l-\rho n}]$ (here and
throughout, $\rho$ is chosen such that if $2r+l-\rho n\in\mathbb{N}_{n}$).
Moreover, the exponent then depends on the parity of $l$ (see the case
considered in Section 6.1 of \cite{[FZ]}). Just as in the case involving only
$\alpha=1$ and $\alpha=-1$ the parity of $n$ affected the constant appearing in
the powers to which $[C_{r},C_{j}]$ and $[D_{r},D_{j}]$ are raised (because of
some maximality issue---see Section \ref{ExpForm} below), here we have a
difference between the cases $n\equiv1(\mathrm{mod\ }4)$ and
$n\equiv3(\mathrm{mod\ }4)$. In order to write the formula for both these cases
together, we define $\varepsilon$ to be 1 in the former case and 0 in the
latter. The expression we now consider is
\[\prod_{r=0}^{n-1}\bigg[\prod_{k=0}^{(n-1)/2}\big([C_{r},E_{2r+2k-\rho
n}][D_{r},F_{2r+2k-\rho
n}]\big)^{2n\big[\frac{n^{2}+2n+1-4\varepsilon}{8}-k(n+1-2k)\big]}\times\]
\[\times\prod_{k=0}^{(n-3)/2}\big([C_{r},E_{2r+2k+1-\rho
n}][D_{r},F_{2r+2k+1-\rho
n}]\big)^{2n\big[\frac{n^{2}+2n+1-4\varepsilon}{8}-k(n-1-2k)\big]}\times\]
\[\times\prod_{k=0}^{(n-1)/2}\big([C_{r},F_{2r+2k-\rho n}][D_{r},E_{2r+2k-\rho
n}]\big)^{2nk(n+1-2k)}\times\]
\begin{equation}
\times\prod_{k=0}^{(n-3)/2}\big([C_{r},F_{2r+2k+1-\rho n}][D_{r},E_{2r+2k+1-\rho
n}]\big)^{2nk(n-1-2k)}\Bigg]. \label{hXiexmix} 
\end{equation}
The value of $\varepsilon$ is chosen so that the powers will be integral,
non-negative, and minimal. In fact, the values of $a_{\beta,\alpha}$ and
$b_{\beta,\alpha}$ are defined by congruences modulo $n$, so that working with a
correction factor like $\rho$ simplifies the theory, as the actual value of
$\rho$ is not important for the evaluations. It will also be useful to write the
expressions involving $k$ in terms of $l$, which is either $2k$ or $2k+1$: This
is $\frac{l(n+1-l)}{2}$ for even $l$ and $\frac{(l-1)(n-l)}{2}$ for odd $l$ (as
in Section 6.1 of \cite{[FZ]}). We call this function
$f_{2}^{(n)}:\mathbb{N}_{n}\to\mathbb{Z}$. We can combine the first two lines to
one product by writing the power as
$2n\big[\frac{n^{2}+2n+1-4\varepsilon}{8}-f_{2}^{(n)}(l)\big]$, and for the last
two lines we can use the power $2nf_{2}^{(n)}(l)$. Note that $f_{2}^{(n)}$ is
invariant under replacing $l\geq2$ by $n+1-l$ (as the latter number has the same
parity as $l$ since $n$ is odd) and interchanging the values $l=0$ and $l=1$ (on
both of which the function attains 0). Now, applying $N_{1}$ takes the index $r$
of $C_{r}$ and $D_{r}$ to a number which is congruent to $n-r$ modulo $n$ (this
is $n-r$ itself, unless $r=0$ where this number is also 0). The index of the
sets $E_{j}$ and $F_{j}$, which is congruent to $2r+l$ modulo $n$, is taken to a
number which is congruent to $n+1-2r-l$ modulo $n$. As this number is
$2(n-r)+(n+1-l)$ modulo $n$, the invariance of $f_{2}^{(n)}$ under $l \mapsto
n+1-l$ (modulo $n$) shows that each multiplier in the product under
consideration is invariant under the operator $N_{1}$.

It remains to prove the invariance of the corresponding part of the quotient
$\frac{h_{\Xi}}{g_{\Xi}^{1}}$ under the operators $T_{Q,R}$. Ignoring the
points $Q$ and $R$ for the moment, the action of these operators sends the sets
$C_{r}$ and $D_{r}$ to the sets with the index which is congruent to $n+1-r$
modulo $n$, while the sets $E_{j}$ and $F_{j}$ with $j \equiv 2r+l(\mathrm{mod\
}n)$ are taken to those sets whose index is congruent to $n+3-l-2r$, or
equivalently $2(n+1-r)+(n+1-l)$, modulo $n$. The invariance of $f_{2}^{(n)}$
under replacing $l$ by $n+1-l$ (or $1-l$ for $l\in\{0,1\}$) now shows that all
the parts not involving subsets of $G$ or of $H$ (hence concerning only sets
which do not contain $Q$ or $R$ and whose images do not contain these points)
are invariant under all the operators $T_{Q,R}$. Next, if $r\geq2$ then in
order to get the sets $E_{1}$ and $F_{1}$ we must take $l$ to be $n+1-2r$
(which is even) for $r\leq\frac{n+1}{2}$ and $2n+1-2r$ (which is odd) if
$r\geq\frac{n+3}{2}$. On the other hand, the sets $E_{2}$ and $F_{2}$ are
obtained using the odd index $l=n+2-2r$ if $r\leq\frac{n+1}{2}$ and the even
index $l=2n+2-2r$ in case $r\geq\frac{n+3}{2}$. Therefore, the power to which
$[C_{r},E_{1}]$ and $[D_{r},F_{1}]$ appear in the quotient in question is
$2n\big[\frac{n^{2}+2n+1-4\varepsilon}{8}-r(n+1-2r)-(r-1)\big]$ for
$r\leq\frac{n+1}{2}$ and
$2n\big[\frac{n^{2}+2n+1-4\varepsilon}{8}-(n-r)(2r-1-n)-(r-1)\big]$ if
$r\geq\frac{n+3}{2}$, while $[C_{r},E_{2}]$ and $[D_{r},F_{2}]$ appear when
raised to the power
$2n\big[\frac{n^{2}+2n+1-4\varepsilon}{8}-(r-1)(n+1-2r)-(n-r)\big]$ if
$r\leq\frac{n+1}{2}$ and
$2n\big[\frac{n^{2}+2n+1-4\varepsilon}{8}-(n+1-r)(2r-1-n)-(n-r)\big]$ in case
$r\geq\frac{n+3}{2}$. In both cases the powers coincide, yielding
$2n\big[\frac{n^{2}+2n+9-4\varepsilon}{8}-r(n+2-2r)\big]$ for small $r$ and 
$2n\big[\frac{n^{2}+2n+9-4\varepsilon}{8}-(n+1-r)(2r-n)\big]$ for big $r$. The
terms $[C_{r},F_{1}]$ and $[D_{r},E_{1}]$ are raised in the quotient to a power
which equals $2n[r(n+1-2r)-(n-r)]$ in case $r\leq\frac{n+1}{2}$ and
$2n[(n-r)(2r-1-n)-(n-r)]$ if $r\geq\frac{n+3}{2}$, while $[C_{r},F_{2}]$ and
$[D_{r},E_{2}]$ carry the power $2n[(r-1)(n+1-2r)-(r-1)]$ for
$r\leq\frac{n+1}{2}$ and $2n[(n+1-r)(2r-1-n)-(r-1)\big]$ for
$r\geq\frac{n+3}{2}$. These powers coincide to $2n[(r-1)(n-2r)]$ when $r$ is
small and to $2n[(n-r)(2r-2-n)]$ in case $r$ is large. Note that all these final
powers are invariant under replacing $r$ by $n+1-r$ (to which the middle value
$r=\frac{n+1}{2}$ is invariant). We consider next the terms involving $C_{0}$,
$C_{1}$, $D_{0}$, and $D_{1}$ together with $E_{l}$ and $F_{l}$ with $l\geq3$:
Then the argument of $f_{2}^{(n)}$ is $l$ for $C_{0}$ and $D_{0}$ and $l-2$ (of
the same parity) for $C_{1}$ and $D_{1}$. $[C_{0},E_{l}]$ and $[D_{0},F_{l}]$
carry the power
$2n\big[\frac{n^{2}+2n+1-4\varepsilon}{8}-\frac{l(n+1-l)}{2}-(l-2)\big]$ if $l$
is even and
$2n\big[\frac{n^{2}+2n+1-4\varepsilon}{8}-\frac{(l-1)(n-l)}{2}-(l-2)\big]$ if
$l$ is odd, and $[C_{1},E_{l}]$ and $[D_{1},F_{l}]$ come raised to the power
$2n\big[\frac{n^{2}+2n+1-4\varepsilon}{8}-\frac{(l-2)(n+3-l)}{2}-(n+1-l)\big]$
for even $l$ and
$2n\big[\frac{n^{2}+2n+1-4\varepsilon}{8}-\frac{(l-3)(n+2-l)}{2}-(n+1-l)\big]$
for odd $l$. Once again we have coincidence, the value being
$2n\big[\frac{n^{2}+2n+17-4\varepsilon}{8}-\frac{l(n+3-l)}{2}\big]$ in the case
of even $l$ and
$2n\big[\frac{n^{2}+2n+17-4\varepsilon}{8}-\frac{(l-1)(n+2-l)}{2}\big]$ if $l$
is odd. In addition, $[C_{0},F_{l}]$ and $[D_{0},E_{l}]$ appear to the power
$2n\big[\frac{l(n+1-l)}{2}-(n+1-l)\big]$ for even $l$ and
$2n\big[\frac{(l-1)(n-l)}{2}-(n+1-l)\big]$ for odd $l$, where $[C_{1},F_{l}]$
and $[D_{1},E_{l}]$ come with a power which equals
$2n\big[\frac{(l-2)(n+3-l)}{2}-(l-2)\big]$ if $l$ is even and
$2n\big[\frac{(l-3)(n+2-l)}{2}-(l-2)\big]$ if $l$ is odd. The coincidence here
yields $2n\frac{(l-2)(n+3-l)}{2}$ when $l$ is even and
$2n\big[\frac{(l-3)(n-l)}{2}-1\big]$ when $l$ is odd. All the final powers here
are invariant under substituting $n+3-l$ in the place of $l$. We complement the
latter values with $[C_{0},E_{0}]$ and $[D_{0},F_{0}]$ (with $l=0$) carrying the
power $2n\big[\frac{n^{2}+2n+1-4\varepsilon}{8}-(n-2)\big]$, $[C_{1},E_{0}]$ and
$[D_{1},F_{0}]$ (for which we take $l=n-2$) being raised to the power
$2n\big[\frac{n^{2}+2n+1-4\varepsilon}{8}-(n-3)-1\big]$, and this common value
being equal to the one attained on $l=3$ above. In addition, $[C_{0},F_{0}]$ and
$[D_{0},E_{0}]$ (also with $l=0$) come with the power $-2n$, the power to which
$[C_{1},F_{0}]$ and $[D_{1},E_{0}]$ (with $l=n-2$ again) is
$2n[(n-3)-(n-2)]=-2n$, and substituting $l=3$ in the appropriate expression
yields also $-2n$.

The last paragraph shows (after applying Lemma \ref{undec}) that the expression
we have in the quotient $\frac{h_{\Xi}}{g_{\Xi}^{1}}$ are based on the
$T_{Q,R}$-invariant unions $C_{0} \cup C_{1}$ and $D_{0} \cup D_{1}$ together
with terms involving $E_{l}$ and $F_{l}$ in a way which is invariant under
taking $l\geq4$ to $n+3-l$ and interchanging the values $l=0$ and $l=3$, as well
as the unions $E_{1} \cup E_{2}$ and $F_{1} \cup F_{2}$ (which are also
invariant under the operators $T_{Q,R}$) with the sets $C_{l}$ and $D_{l}$ in a
manner which is invariant under replacing $l\geq2$ by $n+1-l$. Hence the
operators in question leave all these parts invariant. When we consider the
remaining terms, we find that $[C_{0},E_{1}]$, $[D_{0},F_{1}]$, $[C_{1},E_{2}]$,
and $[D_{1},F_{2}]$ (in which $f_{2}^{(n)}$ takes values on $l=0$ and $l=1$) all
come raised to the power $2n\big[\frac{n^{2}+2n+1-4\varepsilon}{8}-(n-1)\big]$,
while $[C_{0},E_{2}]$, $[D_{0},F_{2}]$, $[C_{1},E_{1}]$, and $[D_{1},F_{1}]$ (in
which we have to substitute the arguments $l=2$ and $l=n-1$) also carry the same
power. These combine to a power of the product of the two invariant expressions
$[C_{0} \cup C_{1},E_{1} \cup E_{2}]$ and $[D_{0} \cup D_{1},F_{1} \cup F_{2}]$.
As $[C_{0},F_{1}]$, $[D_{0},E_{1}]$, $[C_{1},F_{2}]$, and $[D_{1},E_{2}]$ (again
with the values on $l=0$ and $l=1$) do not appear at all and the powers
$[C_{0},F_{2}]$, $[D_{0},E_{2}]$, $[C_{1},F_{1}]$, and $[D_{1},E_{1}]$ (where
$f_{2}^{(n)}$ is evaluated on $l=2$ and on $l=n-1$ again) cancel, this completes
the proof of the $T_{Q,R}$-invariance of the quotient
$\frac{h_{\Xi}}{g_{\Xi}^{1}}$. In total, putting the product $h_{\Xi}$ of the
three expressions from Equations \eqref{hXiexCD}, \eqref{hXiexEF}, and
\eqref{hXiexmix} under the corresponding theta constant yields a quotient which
is invariant under all the operators $T_{Q,R}$ as well as under $N_{1}$. This
will be the Thomae formula for our example. 

\smallskip

The construction in the example shows that the key point in defining the
denominators $h_{\Xi}$ in the general case is the powers to which the
expressions $[D_{\delta,r},D_{\alpha,l}]$ appear in it, and that these are given
in terms of some functions of the indices $r$ and $l$. In order to determine the
properties which these functions must possess, we start by considering only
those parts of $g_{\Xi}^{\beta}$ and $g_{N_{\beta}(\Xi)}^{\beta}$ which involve
the set $D_{\beta,0}$ (which remains invariant under $N_{\beta}$). An expression
of the form $[D_{\beta,0},D_{\alpha,l}]$ appears to the power
$en[n-1-a_{\beta,\alpha}(l)]$ in $g_{\Xi}^{\beta}$. Assume that there exists an
expression $h_{\Xi}$ with the properties stated in the previous paragraph. We
define the functions $f_{\beta,\alpha}:\mathbb{N}_{n}\to\mathbb{Z}$ such that
the power to which $[D_{\beta,0},D_{\alpha,l}]$ appears in $h_{\Xi}$ as
$en[c(\beta,\alpha)-f_{\beta,\alpha}(l)]$, where $c(\beta,\alpha)$
is some integral constant. By altering $c(\beta,\alpha)$ if necessary, we can
always assume $f_{\beta,\alpha}(0)=0$. Then the $N_{\beta}$-invariance of
$h_{\Xi}$ yields the equality
\begin{equation}
f_{\beta,\alpha}\big(a_{\beta,\alpha}(l)\big)=f_{\beta,\alpha}(l) \label{ainv}
\end{equation}
for every $l\in\mathbb{N}_{n}$. Moreover, the $T_{Q,R}$-invariance of
$\frac{h_{\Xi}}{g_{\Xi}^{\beta}}$ implies that the equality
\[f_{\beta,\alpha}\big(b_{\beta,\alpha}(l)\big)-a_{\beta,\alpha}\big(b_{\beta,
\alpha}(l)\big)=f_{\beta,\alpha}(l)-a_{\beta,\alpha}(l)\] holds for every
$l\in\mathbb{N}_{n}$. The latter property becomes easier to work with when we
replace $l$ by $a_{\beta,\alpha}(l)$. Indeed, Lemma \ref{abn-1-a}, the fact that
$a_{\beta,\alpha}$ is an involution, and Equation \eqref{ainv} combine to show
that the latter equality is equivalent to
\begin{equation}
f_{\beta,\alpha}\big[b_{\beta,\alpha}\big(a_{\beta,\alpha}(l)\big)\big]+l=f_{
\beta,\alpha}(l)+n-1-l \label{binv}
\end{equation}
holding for every $l\in\mathbb{N}_{n}$. Moreover, the common value of the two
sides in Equation \eqref{binv} is left invariant under replacing $l$ by
$n-1-l$, as follows from Equation \eqref{ainv} and the second assertion of Lemma
\ref{abn-1-a}. In particular, the normalization $f_{\beta,\alpha}(0)=0$ implies
$f_{\beta,\alpha}(n-1)=n-1$ for every $\alpha$ and $\beta$.

\begin{thm}
For any $n$ and any $\alpha$ and $\beta$ in $\mathbb{N}_{n}$ which are prime to
$n$ there exists a unique function
$f_{\beta,\alpha}:\mathbb{N}_{n}\to\mathbb{Z}$ which satisfies Equations
\eqref{ainv} and \eqref{binv} for every $l\in\mathbb{N}_{n}$ and attains 0 on
$l\in\mathbb{N}_{n}$. \label{fba}
\end{thm}

\begin{proof}
We first prove that there is a unique function
$f_{\beta,\alpha}:\mathbb{N}_{n}\to\mathbb{Z}$ satisfying
$f_{\beta,\alpha}(0)=0$ and Equation \eqref{binv} for every
$l\in\mathbb{N}_{n}$. Observe that $b_{\beta,\alpha} \circ a_{\beta,\alpha}$
adds $\alpha k_{\beta}$ to $l$ up to multiples of $n$, and that $\alpha
k_{\beta}$ is prime to $n$. Hence multiple applications of $b_{\beta,\alpha}
\circ a_{\beta,\alpha}$ takes any element of $\mathbb{N}_{n}$ to any other.
Since Equation \eqref{binv} presents
$f_{\beta,\alpha}[b_{\beta,\alpha}\big(a_{\beta,\alpha}(l)\big)\big]$ as
$f_{\beta,\alpha}(l)$ plus another term, knowing the value of $f_{\beta,\alpha}$
on one element of $\mathbb{N}_{n}$ determines the values of $f_{\beta,\alpha}$
on all the elements of $\mathbb{N}_{n}$. Hence the normalization
$f_{\beta,\alpha}(0)=0$ determines $f_{\beta,\alpha}$ uniquely. Note that $n$
applications of $b_{\beta,\alpha} \circ a_{\beta,\alpha}$ takes every
$l\in\mathbb{N}_{n}$ to itself. Applying Equation \eqref{binv} $n$ times shows
that while doing so we add to $f_{\beta,\alpha}(l)$ the values $n-1-2j$ for all
possible values of $j\in\mathbb{N}_{n}$. As this sum is
$n(n-1)-2\frac{n(n-1)}{2}=0$, these equalities are consistent with one another,
and the function $f_{\beta,\alpha}$ indeed exists (and is unique).

It remains to show that the function $f_{\beta,\alpha}$ thus obtained satisfies
also Equation \eqref{ainv} for all $l\in\mathbb{N}_{n}$. First, Equation
\eqref{ainv} holds if $l$ is a fixed point of $a_{\beta,\alpha}$, and we claim
that $a_{\beta,\alpha}$ must have at least one fixed point. Indeed, we are
looking for $l\in\mathbb{N}_{n}$ such that $2l\equiv\alpha
k_{\beta}-1(\mathrm{mod\ }n)$. For odd $n$ such $l$ exists and is unique. On the
other hand, if $n$ is even then so is $\alpha k_{\beta}-1$, implying that there
are two such values of $l$. Assume that Equation \eqref{ainv} holds for some
value of $l$. We claim that Equation \eqref{ainv} holds also for
$b_{\beta,\alpha}\big(a_{\beta,\alpha}(l)\big)$. To see this, first substitute
$l=a_{\beta,\alpha}\big(b_{\beta,\alpha}(j)\big)$ in Equation \eqref{binv}.
Since $a_{\beta,\alpha}$ and $b_{\beta,\alpha}$ are involutions, Lemma
\ref{abn-1-a} shows that this substitution yields the equality
\[f_{\beta,\alpha}(j)+n-1-a_{\beta,\alpha}(j)=f_{\beta,\alpha}\big[a_{\beta,
\alpha}\big(b_{\beta,\alpha}(j)\big)\big]+a_{\beta,\alpha}(j).\] Put now
$j=a_{\beta,\alpha}(l)$ and use the involutive property of $a_{\beta,\alpha}$
again to obtain
\[f_{\beta,\alpha}\big(a_{\beta,\alpha}(l)\big)+n-1-l=f_{\beta,\alpha}\big\{a_{
\beta,\alpha}\big[b_{\beta,\alpha}\big(a_{\beta,\alpha}(l)\big)\big]\big\}+l.\]
Equation \eqref{binv} and the assumption that Equation \eqref{ainv} holds for
$l$ now yield Equation \eqref{ainv} for
$b_{\beta,\alpha}\big(a_{\beta,\alpha}(l)\big)$. Since we have shown that
Equation \eqref{ainv} holds for some $l\in\mathbb{N}_{n}$ and that multiple
applications of $b_{\beta,\alpha} \circ a_{\beta,\alpha}$ connect any two
elements of $\mathbb{N}_{n}$, this completes the proof of the theorem.
\end{proof}

Observe that altering the constants $c(\beta,\alpha)$ does not affect the
invariance of the quotient $\frac{\theta^{2en^{2}}[\Xi](0,\Pi)}{h_{\Xi}}$ under
any operator, so one may choose these constants arbitrarily. It is natural to
normalize the constants such that $h_{\Xi}$ is a polynomial (i.e., excluding
negative powers) and reduced (i.e., some $[D_{\beta,j},D_{\alpha,l}]$ appears
with vanishing power). However, determining these constants depends much more
delicately on the relations between $n$, $\alpha$, and $\beta$: For example,
such a normalizing constant $c(\beta,\beta)$ depends on the parity of $n$ while
$f_{\beta,\beta}$ does not (see the differences between the formulae for odd and
even $n$ in Chapters 4 and 5 of \cite{[FZ]}). As another example, if $n$ is odd
and $\alpha k_{\beta}$ is 2 modulo $n$ then the form of these constants depends
on whether $n$ is equivalent to 1 or to 3 modulo 4, while the form of the
function $f_{\beta,\alpha}$ does not depend on this congruence (see the example
in Section 6.1 of \cite{[FZ]}).

\medskip

We now present several lemmas, which are needed to define the denominator
$h_{\Xi}$ and to establish its properties.

\begin{lem}
Given three elements $\alpha$, $\beta$, and $\delta$ of $\mathbb{N}_{n}$ which
are all prime to $n$ and two elements $l$ and $r$ of $\mathbb{N}_{n}$, let
$j\in\mathbb{N}_{n}$ be the element which is congruent to $l+r\alpha k_{\delta}$
modulo $n$. Then the congruences
\[a_{\beta,\alpha}(j) \equiv a_{\delta,\alpha}(l)+a_{\beta,\delta}(r)\alpha
k_{\delta}(\mathrm{mod\ }n),\quad b_{\beta,\alpha}(j) \equiv
a_{\delta,\alpha}(l)+b_{\beta,\delta}(r)\alpha k_{\delta}(\mathrm{mod\ }n)\]
hold. \label{abbetadeltaalpha}
\end{lem}

\begin{proof}
As in the proof of Proposition \ref{operZn}, we take $\eta$ to be 1 when we work
with $a$ and 2 when we work with $b$. By definition, the left hand side of our
expressions is congruent to $\eta\alpha k_{\beta}-1-l-r\alpha k_{\delta}$
modulo $n$, while the right hand side is congruent to $\alpha
k_{\delta}-1-l+\alpha k_{\delta}(\eta\delta k_{\beta}-1-r)$ modulo $n$. The
latter expression contains $-1-l-r\alpha k_{\delta}$, the two terms with $\alpha
k_{\delta}$ cancel, and the terms including $\eta$ also coincide since $\delta
k_{\delta}\equiv1(\mathrm{mod\ }n)$. This proves the lemma.
\end{proof}

\begin{lem}
For every $l\in\mathbb{N}_{n}$ (given $\alpha$ and $\delta$) let
$y_{\alpha,\delta,l}$ denote the number $-l\delta
k_{\alpha}-ns_{\delta,lk_{\alpha}}\in\mathbb{N}_{n}$. Then the equality
\[f_{\alpha,\delta}(y_{\alpha,\delta,l}+n\delta_{l0}-1)+l=f_{\alpha,\delta}(y_{
\alpha,\delta,l})+n-1-l\] holds. \label{fralphakbeta}
\end{lem}

In this Lemma, $\delta_{l0}$ denotes Kronecker's symbol (namely 1 if $l=0$ and 0
otherwise). It is included here to account for the fact that for
$y_{\alpha,\delta,0}=0$ the number $y_{\alpha,\delta,0}-1=-1$ is not in
$\mathbb{N}_{n}$ but adding $n$ to it yields $n-1\in\mathbb{N}_{n}$. In this
case the assertion of Lemma \ref{fralphakbeta} reduces to the equality
$f_{\beta,\alpha}(n-1)=n-1$ which we already obtained above.

\begin{proof}
We prove the asserted equality by decreasing induction on $l$. We begin by
observing that $y_{\alpha,\delta,n-1}=a_{\alpha,\delta}(n-1)$ while
$y_{\alpha,\delta,n-1}-1=b_{\alpha,\delta}\big(a_{\alpha,\delta}(n-1)\big)$ (or
alternatively,
$y_{\alpha,\delta,n-1}=b_{\alpha,\delta}\big(a_{\alpha,\delta}(0)\big)$ and
$y_{\alpha,\delta,n-1}-1=a_{\alpha,\delta}(0)$). Hence the assertion for $l=n-1$
follows directly from Equations \eqref{ainv} and \eqref{binv}. Now assume that
the assertion holds for $0<l \leq n-1$, and we wish to prove it for $l-1$. As
$y_{\alpha,\delta,l-1}$ is
$b_{\alpha,\delta}\big(a_{\alpha,\delta}(y_{\alpha,\delta,l} )\big)$ and
$y_{\alpha,\delta,l-1}+n\delta_{l1}-1$ equals
$b_{\alpha,\delta}\big(a_{\alpha,\delta}(y_{ \alpha,\delta,l}-1)\big)$, Equation
\eqref{binv} shows that the left hand side and right hand side of the equation
corresponding to $l-1$ are
\[f_{\alpha,\delta}(y_{\alpha,\delta,l}-1)+n-2y_{\alpha,\delta,l}+l\quad\mathrm{
and }\quad f_{\alpha,\delta}(y_{\alpha,\delta,l})+2n-1-2y_{\alpha,\delta,l}-l\]
respectively. But these expressions are obtained by adding
$n-2y_{\alpha,\delta,l}$ to both sides of the equality corresponding to $l$.
Hence if the equality holds for $l$ it also holds for $l-1$. This completes the
proof of the lemma.
\end{proof}

\begin{lem}
The equality $f_{\delta,\alpha}(l)=f_{\alpha,\delta}(-l\delta
k_{\alpha}-ns_{\delta,-lk_{\alpha}})$ (namely
$f_{\alpha,\delta}(y_{\alpha,\delta,l})$ in the notation of Lemma
\ref{fralphakbeta}) holds for every $\alpha$ and $\delta$ and every
$l\in\mathbb{N}_{n}$. \label{invind}
\end{lem}

\begin{proof}
As both sides attain 0 on $l=0$, Theorem \ref{fba} reduces the assertion to
verifying that the function of $l$ given on the right hand side satisfies
Equations \eqref{ainv} and \eqref{binv} with the parameters $\delta$ and
$\alpha$. Substituting $a_{\delta,\alpha}(l)$ in place of $l$ yields an argument
of $f_{\alpha,\delta}$ which lies between 0 and $n-1$ and is congruent to
$\delta k_{\alpha}-1+l\delta k_{\alpha}$ modulo $n$ (recall that $\alpha
k_{\alpha}\equiv\delta k_{\delta}\equiv1(\mathrm{mod\ }n)$). Since this number
is (by definition) the $a_{\alpha,\delta}$-image of $-l\delta
k_{\alpha}-ns_{\delta,-lk_{\alpha}}$, Equation \eqref{ainv} for the latter
function confirms that Equation \eqref{ainv} is satisfied also with the required
argument. For Equation \eqref{binv} we consider the right hand side as
$f_{\alpha,\delta}(y_{\alpha,\delta,l})$. Applying $b_{\delta,\alpha} \circ
a_{\delta,\alpha}$ to $l$ is the same as adding $\alpha k_{\delta}$ to it
(modulo $n$), and after multiplying by $-\delta k_{\alpha}$ the argument of
$f_{\alpha,\delta}$ becomes $y_{\alpha,\delta,l}+n\delta_{l0}-1$ (recall that
both $\delta k_{\delta}$ and $\alpha k_{\alpha}$ are 1 modulo $n$). The desired
Equation \eqref{binv} now follows from Lemma \ref{fralphakbeta}. This proves the
lemma.
\end{proof}

The following lemma is not a part of the proof of the Thomae formulae (Theorem
\ref{NTinv} below), but it will turn out to be useful for deriving explicit
expressions for the functions $f_{\beta,\alpha}$ in Section \ref{ExpForm}.

\begin{lem}
The function $f_{n-\beta,\alpha}$ is related to the function $f_{\beta,\alpha}$
through the equality $f_{n-\beta,\alpha}(l)=2l-f_{\beta,\alpha}(l)$ (holding
for all $l\in\mathbb{N}_{n}$). \label{fsign}
\end{lem}

\begin{proof}
First we observe that the equalities
$n-1-a_{n-\beta,\alpha}(l)=a_{\beta,\alpha}(n-1-l)$ and
$n-1-b_{n-\beta,\alpha}(l)=b_{\beta,\alpha}(n-1-l)$ hold for every $\alpha$,
$\beta$, and $l\in\mathbb{N}_{n}$. Indeed, all four numbers are in
$\mathbb{N}_{n}$, the former two are congruent to $\alpha k_{\beta}+l$ modulo
$n$, and the latter two are $2\alpha k_{\beta}+l$ up to multiples of $n$.
Consider now the function
$\psi_{n-\beta,\alpha}(l)=n-1-f_{\beta,\alpha}(n-1-l)$. Equation \eqref{ainv}
for $f_{\beta,\alpha}$ implies
\[f_{\beta,\alpha}\big(n-1-a_{n-\beta,\alpha}(l)\big)=f_{\beta,\alpha}\big(a_{
\beta,\alpha}(n-1-l)\big)=f_{\beta,\alpha}(n-1-l),\] which yields Equation
\eqref{ainv} for $\psi_{n-\beta,\alpha}$ with the parameters $n-\beta$ and
$\alpha$. Using the equalities above and Equation \eqref{binv} for
$f_{\beta,\alpha}$ we also obtain
\[f_{\beta,\alpha}\big[n-1-b_{n-\beta,\alpha}\big(a_{n-\beta,\alpha}
(l)\big)\big]+n-1-l=f_{\beta,\alpha}\big[b_{\beta,\alpha}\big(a_{\beta,\alpha}
(n-1-l)\big)\big]+n-1-l=\]
\[=f_{\beta,\alpha}(n-1-l)+n-1-(n-1-l)=f_{\beta,\alpha}(n-1-l)+l.\] Subtracting
both sides from $2n-2$ establishes \eqref{binv} for $\psi_{n-\beta,\alpha}$ with
the same parameters. As $f_{\beta,\alpha}(n-1)=n-1$ for all $\alpha$ and
$\beta$, we deduce that $\psi_{n-\beta,\alpha}(0)=0$. Hence
$\psi_{n-\beta,\alpha}=f_{n-\beta,\alpha}$ by Theorem \ref{fba}. As replacing
$l$ by $n-1-l$ leaves the expression appearing in Equation \eqref{binv}
invariant, the expression defining $\psi_{n-\beta,\alpha}(l)$ can be written as
$\psi_{n-\beta,\alpha}(l)=2l-f_{\beta,\alpha}(l)$ for every
$l\in\mathbb{N}_{n}$. This proves the lemma.
\end{proof}

Lemmas \ref{invind} and \ref{fsign} are related to the fact that the third part
of the expression for $h_{\Xi}$ in our example was based on a single function
$f_{2}^{(n)}$. The role of Lemma \ref{invind} in this becomes more apparent
when one sees the proof of Theorem \ref{NTinv}, and Lemma \ref{fsign} is
related to the fact that all our expressions there came in pairs, e.g.,
$[C_{r},E_{l}]$ with $[D_{r},F_{l}]$ and $[C_{r},F_{l}]$ with $[D_{r},E_{l}]$
(recall the index inversion in the sets $D_{l}$ and $F_{l}$ there). 

\medskip

Fix an order on the set of pairs $(\alpha,l)$ with $\alpha\in\mathbb{N}_{n}$
prime to $n$ and $l\in\mathbb{Z}/n\mathbb{Z}$. Choose, for every $\delta$ and
$\alpha$, an integral constant $c(\delta,\alpha)$ such that
$c(\delta,\alpha)=c(\alpha,\delta)$ for every $\alpha$ and $\delta$. Define, for
any divisor $\Xi$ as in Equation \eqref{Xi}, the expression
\[h_{\Xi}=\prod_{(\delta,r)\leq(\alpha,l+r\alpha
k_{\delta})}[D_{\delta,r},D_{\alpha,l+r\alpha
k_{\delta}}]^{en[c(\delta,\alpha)-f_{\delta,\alpha}(l)]}.\] The inequality in
the product is with respect to the chosen order. We now prove

\begin{thm}
The expression $h_{\Xi}$ is independent of the order chosen. The quotient
$\frac{\theta^{2en^{2}}[\Xi](0,\Pi)}{h_{\Xi}}$ is invariant under all the
operators $N_{\beta}$ as well as under all the admissible operators $T_{Q,R}$.
\label{NTinv}
\end{thm}

\begin{proof}
Changing the order means that for some pairs, we write
$[D_{\delta,r},D_{\alpha,l+r\alpha k_{\delta}}]$ as
$[D_{\alpha,s},D_{\delta,j+s\delta k_{\alpha}}]$ for appropriate $s$ and $j$. We
need to see that the power to which this expression appears in $h_{\Xi}$ is the
same. But $s \equiv l+r\alpha k_{\delta}(\mathrm{mod\ }n)$ and $j \equiv
r-s\delta k_{\alpha}(\mathrm{mod\ }n)$, so that $j\equiv-l\delta
k_{\alpha}(\mathrm{mod\ }n)$ since $\alpha k_{\alpha}$ and $\delta k_{\alpha}$
are congruent to 1 modulo $n$. Therefore the powers to which the two forms of
this expression appear in $h_{\Xi}$, namely
$c(\delta,\alpha)-f_{\delta,\alpha}(l)$ and
$c(\alpha,\delta)-f_{\alpha,\delta}(j)$, coincide by Lemma \ref{invind} and the
choice of the constants. This proves the independence of $h_{\Xi}$ of the order
chosen on the set of pairs $(\alpha,l)$.

Proposition \ref{PMTbpinv} and Equation \eqref{even} reduce the invariance
assertions to the statements that $h_{\Xi}$ is invariant under any operator
$N_{\beta}$, and for any $\beta$ the quotient $\frac{h_{\Xi}}{g_{\Xi}^{\beta}}$
is invariant under every admissible operator $T_{Q,R}$ with $Q \in D_{\beta,0}$
(for this $\beta$). Decompose $\frac{h_{\Xi}}{g_{\Xi}^{\beta}}$ into the product
of expressions involving some set $D_{\alpha,\alpha k_{\beta}}$ or
$D_{\alpha,\alpha k_{\beta}-1}$ and those which do not. The division by
$g_{\Xi}^{\beta}$ affects only the powers appearing in the first part in this
decomposition. We start with the invariance under $N_{\beta}$, as well as the
$T_{Q,R}$-invariance of the second part of $\frac{h_{\Xi}}{g_{\Xi}^{\beta}}$ (or
simply of $h_{\Xi}$). By Equations \eqref{Nbetaeq} and \eqref{TQRXieq} this
invariance reduces to verifying that together with any expression
$[D_{\delta,r},D_{\alpha,j}]$, the expressions
$[D_{\delta,a_{\beta,\delta}(r)},D_{\alpha,a_{\beta,\alpha}(j)}]$ and
$[D_{\delta,b_{\beta,\delta}(r)},D_{\alpha,b_{\beta,\alpha}(j)}]$ appear to the
same power in $h_{\Xi}$. But if $j \equiv l+r\alpha k_{\delta}$ for some
$l\in\mathbb{N}_{n}$ then the first expression appears to the power
$en[c(\delta,\alpha)-f_{\delta,\alpha}(l)]$, and Lemma \ref{abbetadeltaalpha}
implies that the other two expression must then appear to the power
$en\big[c(\delta,\alpha)-f_{\delta,\alpha}\big(a_{\delta,\alpha}(l)\big)\big]$.
The two invariance assertions now follow from Equation \eqref{ainv} for
$f_{\delta,\alpha}$.

It remains to prove the invariance of the first part of
$\frac{h_{\Xi}}{g_{\Xi}^{\beta}}$ under the operators $T_{Q,R}$ (with $Q \in
D_{\beta,0}$). We may choose the order such that the expressions we consider
include only powers of $[D_{\delta,\delta k_{\beta}},D_{\alpha,j}]$ and
$[D_{\delta,\delta k_{\beta}-1},D_{\alpha,j}]$ (either with $j$ taking the
values $\alpha k_{\beta}$ or $\alpha k_{\beta}-1$ or with $j$ taking other
values). Since the operators $T_{Q,R}$ may mix $D_{\delta,\delta k_{\beta}}$
with $D_{\delta,\delta k_{\beta}-1}$, Equation \eqref{TQRXieq} shows that for
$T_{Q,R}$-invariance of $\frac{h_{\Xi}}{g_{\Xi}^{\beta}}$, this quotient must
contain all the expressions $[D_{\delta,\delta k_{\beta}},D_{\alpha,j}]$,
$[D_{\delta,\delta k_{\beta}-1},D_{\alpha,j}]$, $[D_{\delta,\delta
k_{\beta}},D_{\alpha,b_{\beta,\alpha}(j)}]$, and $[D_{\delta,\delta
k_{\beta}-1},D_{\alpha,b_{\beta,\alpha}(j)}]$ raised to the same power. Observe
that this assertion holds regardless of whether $j$ is congruent to one of
$\alpha k_{\beta}$ and $\alpha k_{\beta}-1$ modulo $n$ or not, since in the
former case, where additional mixing may appear, $b_{\beta,\alpha}$ interchanges
the elements of $\mathbb{N}_{n}$ which are congruent to $\alpha k_{\beta}$ and
$\alpha k_{\beta}-1$ modulo $n$ with one another. We remark that in the former
case with $\alpha=\delta$ the assertion refers to the three expressions
$[D_{\delta,\delta k_{\beta}},D_{\delta,\delta k_{\beta}}]$, $[D_{\delta,\delta
k_{\beta}},D_{\delta,\delta k_{\beta}-1}]$, and $[D_{\delta,\delta
k_{\beta}-1},D_{\delta,\delta k_{\beta}-1}]$.

Now, $h_{\Xi}$ is given in terms of $[D_{\delta,r},D_{\alpha,l+r\alpha
k_{\delta}}]$ while $g_{\Xi}^{\beta}$ is given in terms of
$[D_{\delta,r},D_{\alpha,l}]$ for $r$ being either $\delta k_{\beta}$ or $\delta
k_{\beta}-1$. In the case $r=\delta k_{\beta}$ the index $l+r\alpha k_{\delta}$
coincides modulo $n$ with $l+\alpha k_{\beta}$ (as $\delta$ and $k_{\delta}$
cancel modulo $n$) hence with $b_{\beta,\alpha}\big(a_{\beta,\alpha}(l)\big)$.
We shall thus express $g_{\Xi}^{\beta}$ also in terms of this set. In addition,
$D_{\delta,\delta k_{\beta}-1}$ is associated in $h_{\Xi}$ with
$D_{\alpha,l+\alpha k_{\beta}-\alpha k_{\delta}}$, which we write as
$D_{\alpha,a_{\delta,\alpha}(b_{\delta,\alpha}(l))+\alpha k_{\beta}}$. By
replacing $l$ by $b_{\delta,\alpha}\big(a_{\delta,\alpha}(l)\big)$ in the
expressions involving $D_{\delta,\delta k_{\beta}-1}$ in $h_{\Xi}$ we find that
the part of $h_{\Xi}$ containing $D_{\delta,\delta k_{\beta}}$ or
$D_{\delta,\delta k_{\beta}-1}$ is \[\prod_{\{(\delta,\alpha,l)|l \neq
n-1,\delta\leq\alpha\mathrm{\ if\ }l=0\}}[D_{\delta,\delta
k_{\beta}},D_{\alpha,l+\alpha
k_{\beta}}]^{en[c(\delta,\alpha)-f_{\delta,\alpha}(l)]}\times\]
\[\times\prod_{\{(\delta,\alpha,l)|\delta\leq\alpha\mathrm{\ if\
}l=n-1\}}[D_{\delta,\delta k_{\beta}-1},D_{\alpha,l+\alpha
k_{\beta}}]^{en\{c(\delta,\alpha)-f_{\delta,\alpha}[b_{\delta,
\alpha}(a_{\delta,\alpha}(l))]\}}\] (this form is based on an order in which
$\delta\leq\alpha$ implies $(\delta,\delta k_{\beta})\leq(\alpha,\alpha
k_{\beta})$ and $(\delta,\delta k_{\beta}-1)\leq(\alpha,\alpha k_{\beta}-1)$
and in which $(\delta,\delta k_{\beta}-1)<(\alpha,\alpha k_{\beta})$ for all
$\alpha$ and $\delta$). On the other hand, Lemma \ref{abn-1-a}, the fact that
$a_{\beta,\alpha}$ is an involution, and the congruence $l+\alpha k_{\beta}
\equiv b_{\beta,\alpha}\big(a_{\beta,\alpha}(l)\big)(\mathrm{mod\ }n)$ allow us
to write $g_{\Xi}^{\beta}$ as
\[\prod_{\{(\delta,\alpha,l)|\delta\leq\alpha\mathrm{\ if\
}l=0\}}\!\!\!\!\!\![D_{\delta,\delta k_{\beta}},D_{\alpha,l+\alpha
k_{\beta}}]^{en(n-1-l)}\cdot\!\!\!\!\!\!\prod_{\{(\delta,\alpha,
l)|\delta\leq\alpha\mathrm {\ if\ }l=n-1\}}\!\!\!\!\!\![D_{\delta,\delta
k_{\beta}-1},D_{\alpha,l+\alpha k_{\beta}}]^{enl}\] (we can add the condition $l
\neq n-1$ trivially to the first product, in order to resemble the expression
arising from $h_{\Xi}$). The powers to which $[D_{\delta,\delta
k_{\beta}},D_{\alpha,l+\alpha k_{\beta}}]$ and $[D_{\delta,\delta
k_{\beta}-1},D_{\alpha,l+\alpha k_{\beta}}]$ appear in the quotient
$\frac{h_{\Xi}}{g_{\Xi}^{\beta}}$ are now seen to be $en$ times
$c(\delta,\alpha)-f_{\delta,\alpha}(l)-n+1+l$ and
$c(\delta,\alpha)-f_{\delta,\alpha}\big[b_{\delta,\alpha}\big(a_{\delta,\alpha}
(l)\big)\big]-l$ respectively. These numbers are equal by Equation \eqref{binv}.
Applying $b_{\beta,\alpha}$ to an element of $\mathbb{N}_{n}$ which is congruent
to $l+\alpha k_{\beta}$ modulo $n$ yields the element of $\mathbb{N}_{n}$ which
is congruent to $n-1-l+\alpha k_{\beta}$ modulo $n$. Hence the action of
$b_{\beta,\alpha}$ in this setting takes $l$ to $n-1-l$. The invariance of the
number appearing in Equation \eqref{binv} under this operation now completes the
proof of the theorem.
\end{proof}

We remark that the assertion of Theorem \ref{NTinv} holds also for divisors
$\Xi$ for which all the sets $D_{\beta,0}$ are empty. In this case it refers
only to the action of the operators $N_{\beta}$. Moreover, the fact that the
power to which an expression $[D_{\delta,j},D_{\alpha,l}]$ appears in $h_{\Xi}$
depends only on the index difference between $l$ and $j$ in some sense allows,
in any particular case, for a pictorial description of these powers, in
similarity to Chapters 4 and 5 of \cite{[FZ]}.

\section{Transitivity and the Full Thomae Formulae \label{Trans}}

Another operation on the divisors $\Xi$ from Equation \eqref{Xi} is related to
changing the base point. For any $k\in\mathbb{N}_{n}$ (and even
$k\in\mathbb{Z}$), we let
$w_{k}=\frac{dz}{\omega_{k}}=\frac{w^{k}}{\prod_{\alpha,i}(z-\lambda_{\alpha,i}
)^{s_{\alpha,k}}}$ be the ``normalized $k$th power of $w$''. $w_{k}$ is defined
for $k\in\mathbb{Z}/n\mathbb{Z}$ and its divisor is
$\prod_{\alpha}\prod_{i=1}^{r_{\alpha}}P_{\alpha,i}^{\alpha
k-ns_\alpha,k}/\prod_{h=1}^{n}\infty_{h}^{t_{k}}$ (which is trivial if $n$
divides $k$). If $k$ is prime to $n$ then this function is the function whose
$n$th power gives a normalized $Z_{n}$-equation for $X$ with the appropriate
generator of $\mathbb{C}(X)$ over $\mathbb{C}(z)$. Multiplying $\Xi$ by
$div(w_{k})$ yields a divisor with the same $\varphi_{Q}$-image for any base
point $Q$ (by Abel's Theorem), and the same claim holds for multiplication by
$div\big(\frac{w_{k}}{p(z)}\big)$ for any polynomial in $z$. For every
$k\in\mathbb{Z}$ we define $M^{k}$ to be the operator which takes a divisor
$\Xi$ from Equation \eqref{Xi} and multiplies it by
$div\big(\frac{w_{k}}{p_{k,\Xi}(z)}\big)$, where
$p_{k,\Xi}(z)=\prod_{\alpha}\prod_{i \in A_{\alpha,k}}(z-\lambda_{\alpha,i})$.
Recall from Corollary \ref{ikDelta} that $A_{\alpha,k}$ denotes the set of
indices $1 \leq i \leq r_{\alpha}$ such that $v_{P_{\alpha,i}}(\Xi)>n-1-\alpha
k+ns_{\alpha,k}$, or equivalently $P_{\alpha,i} \in D_{\alpha,l}$ for $0 \leq
l<\alpha k-ns_{\alpha,k}$ (so that for $k$ divisible by $n$ all the sets
$A_{\alpha,k}$ are empty). Thus, dividing by the divisor of $p_{k,\Xi}(z)$
ensures that all the branch points appear in $M^{k}(\Xi)$ raised to powers from
$\mathbb{N}_{n}$.

\begin{prop}
$M^{k}$ defines an operator on the set of divisors $\Xi$ from Equation
\eqref{Xi} satisfying the cardinality conditions. Moreover, this operator
leaves the quotient $\frac{\theta^{2en^{2}}[\Xi](0,\Pi)}{h_{\Xi}}$ invariant.
\label{Minv}
\end{prop}

\begin{proof}
By definition, the divisor $M^{k}(\Xi)$ contains no branch point to the power
$n$ or higher. Moreover, the $k$th cardinality condition on $\Xi$ implies that
the powers of the points $\infty_{h}$ arising from $w_{k}$ and from
$p_{k,\Xi}(z)$ cancel, so that $M^{k}(\Xi)$ also takes the form given in
Equation \eqref{Xi}. Its degree is $g+n-1$, since we multiplied $\Xi$ by the
divisor of a rational function on $X$, which thus has degree 0. The set
$D_{\alpha,l}$ now appears in $M^{k}(\Xi)$ to a power which lies in
$\mathbb{N}_{n}$ and is congruent to $n-1-l+\alpha k$ modulo $n$. Hence we can
write $M^{k}(\Xi)$ in the form of Equation \eqref{Xi} as
$\prod_{\alpha}\prod_{l=0}^{n-1}D_{\alpha,l+\alpha k}^{n-1-l}$. As for the
cardinality conditions, we may assume that $n$ does not divide $k$ (since
otherwise $M^{k}(\Xi)=\Xi$ for which we know that the assertion holds). To see
that $M^{k}(\Xi)$ satisfies the $j$th cardinality condition, we observe that for
$\alpha$ such that $\alpha j-ns_{\alpha,j}+\alpha k-ns_{\alpha,k}<n$ we take the
cardinalities of the sets $D_{\alpha,l}$ with $\alpha k-ns_{\alpha,k} \leq l
\leq \alpha(k+j)-ns_{\alpha,k+j}-1$, while if $\alpha j-ns_{\alpha,j}+\alpha
k-ns_{\alpha,k} \geq n$ then we consider the sets with $\alpha k-ns_{\alpha,k}
\leq l \leq n-1$ together with those with $0 \leq l \leq
\alpha(k+j)-ns_{\alpha,k+j}-1$. Adding and subtracting
$\sum_{\alpha}\sum_{l=0}^{\alpha k-ns_{\alpha,k}}|D_{\alpha,l}|$ (which equals
$t_{k}$) shows that the sum in question (which we need to be $t_{j}$) equals
$t_{k+j}-t_{k}+\sum_{\{\alpha|\alpha j-ns_{\alpha,j}+\alpha k-ns_{\alpha,k} \geq
n\}}r_{\alpha}$. The fact that $s_{\alpha,k}+s_{\alpha,j}$ is $s_{\alpha,k+j}-1$
if $\alpha j-ns_{\alpha,j}+\alpha k-ns_{\alpha,k} \geq n$ and equals
$s_{\alpha,k+j}$ otherwise and the definition of the numbers $t_{k}$, $t_{j}$,
and $t_{k+j}$ completes the proof of the cardinality conditions for
$M^{k}(\Xi)$. Note that the proof works also if $n|k+j$, where $t_{k+j}=0$ and
the sum of $r_{\alpha}$ is taken over all $\alpha$. This proves the first
assertion.

We now turn to the second assertion. Since $M^{k}$ multiplies $\Xi$ by a
principal divisor, $M^{k}(\Xi)$ and $\Xi$ represent the same characteristic for
all $k$ and $\Xi$. Hence we have to show that $h_{M^{k}(\Xi)}=h_{\Xi}$.
According to the formula for the action of $M^{k}$, this assertion is equivalent
to the statement that $[D_{\delta,r},D_{\alpha,j}]$ appears in $h_{\Xi}$ to the
same power as $[D_{\delta,r+\delta k},D_{\alpha,j+\alpha k}]$ for every
$\alpha$, $\delta$, $r$, and $j$. Write $j$ as $l+r\alpha k_{\delta}$ modulo
$n$, so that the first power is $c(\delta,\alpha)-f_{\delta,\alpha}(l)$. The
congruence $l+(r+\delta k)\alpha k_{\delta} \equiv l+r\alpha k_{\delta}+\alpha
k(\mathrm{mod\ }n)$ (since $n$ divides $\delta k_{\delta}-1$) shows that
$[D_{\delta,r+\delta k},D_{\alpha,j+\alpha k}]$ appears to the same power in
$h_{\Xi}$. This completes the proof of the proposition.
\end{proof}

We remark that the proof of the cardinality conditions in Proposition \ref{Minv}
can be adapted to provide a direct proof of the appropriate assertions for
$N_{\beta}(\Xi)$ (or $N_{Q}(\Delta)$) above, as well as for the images of
$T_{Q,R}$. However, using Theorem \ref{NQTQR}, Proposition \ref{operZn}, and
Theorem \ref{nonsp} we could establish these assertions without the need of
direct evaluations.

From the formula for $M^{k}(\Xi)$ in the proof of Proposition \ref{Minv}, it is
clear that $M^{k}$ is the $k$th power of the operator $M=M^{1}$, and that this
operator is of order $n$ (recall that the index $l$ of $D_{\alpha,l}$ is
considered in $\mathbb{Z}/n\mathbb{Z}$). This operator $M$ reduces to the
operator denoted $M$ and given explicitly in Propositions 1.14 and 1.15 (as well
as after Theorem A.2) and implicitly in Propositions 6.10 and 6.12 of
\cite{[FZ]} in the appropriate special cases.

The proof of Proposition \ref{Minv} shows that rather than divisors of the form
of Equation \eqref{Xi} defining characteristics, $M$-orbits of such divisors
(which are sets of $n$ such divisors) provide better definitions for these
characteristics. Moreover, it follows from the formula for $M^{k}(\Xi)$ that for
any branch point $Q$ and any $M$-orbit, any two divisors in that orbit contain
$Q$ raised to different powers in $\mathbb{N}_{n}$. In particular, there is
precisely one divisor $\Xi$ in this orbit such that $v_{Q}(\Xi)=n-1$, and by
Theorem \ref{nonsp} this divisor $\Xi$ is of the form $Q^{n-1}\Delta$ with
$\Delta$ non-special. As indicated in Section \ref{Operators}, different
non-special divisors represent different characteristics (with a given base
point). Hence every $M$-orbit represents another characteristic. In addition,
this establishes the base point change formula: Let a non-special divisor
$\Delta$ supported on the branch points distinct from $Q$ and another base point
$S$ be given. The non-special degree $g$ divisor $\Gamma$ not containing $S$ in
its support and satisfying the equality
$\varphi_{S}(\Gamma)+K_{S}=\varphi_{Q}(\Delta)+K_{Q}$ is
$\frac{\Upsilon}{S^{n-1}}$, where $\Upsilon$ is the unique divisor in the
$M$-orbit of $\Xi=Q^{n-1}\Delta$ with $v_{S}(\Upsilon)=n-1$. This generalizes
Propositions 1.14, 1.15, 6.10, and 6.21 of \cite{[FZ]}.

We remark at this point that given a base point $Q$, we considered only those
divisors from Theorem \ref{nonsp} whose support does not contain $Q$ for
characteristics. This is required for the theta constant
$\theta[\Delta,Q](0,\Pi)$ not to vanish. It follows that if a non-special
divisor of degree $g$ contains all the branch points in its support, then it
cannot represent a non-vanishing theta constant with any branch point as base
point. One such divisor shows up in Theorem 6.3 of \cite{[FZ]}.

\smallskip

We present an assertion about the operators which are defined on all the
divisors $\Xi$.

\begin{lem}
The operator $M$ and the operators $N_{\beta}$ for the various $\beta$ all lie
in the same dihedral group $G$ of order $2n$. \label{NMdih}
\end{lem}

\begin{proof}
Given any $\beta$ and $k$, both compositions $M^{k}\big(N_{\beta}(\Xi)\big)$
and  $N_{\beta}\big(M^{-k}(\Xi)\big)$ yield the divisor
$\prod_{\alpha,l}D_{\alpha,\alpha(k_{\beta}-k)-1-l}^{n-1-l}$. Hence the order 2
operator $N_{\beta}$ and the elements $M^{k}$ of the cyclic group
$M^{\mathbb{Z}/n\mathbb{Z}}$ satisfy the relation $M^{k} \circ
N_{\beta}=N_{\beta} \circ M^{-k}$ defining the dihedral group of order $2n$.
Moreover, replacing $\beta$ by another element $\delta\in\mathbb{N}_{n}$ which
is prime to $n$ and replacing $k$ by $j=k+k_{\delta}-k_{\beta}$ yields the same
operator as above (namely $M^{k} \circ N_{\beta}=N_{\beta} \circ M^{-k}$
coincides with $M^{j} \circ N_{\delta}=N_{\delta} \circ M^{-j}$). Hence the
dihedral group generated by $M$ and $N_{\beta}$ is the same group for every
$\beta$. This proves the lemma.
\end{proof}

Apart from the powers of $M$, the dihedral group $G$ from Lemma \ref{NMdih}
consists of the operators taking $\Xi$ from Equation \eqref{Xi} to
$\prod_{\alpha,l}D_{\alpha,\alpha k-1-l}^{n-1-l}$ for all
$k\in\mathbb{Z}/n\mathbb{Z}$. Let $N$ be the operator with the simplest choice
$k=0$, whose action is
$N(\Xi)=\prod_{\alpha,l}D_{\alpha,n-1-l}^{n-1-l}=\prod_{\alpha,l}D_{\alpha,l}^{l
}$, and consider $G$ as generated by $M$ and $N$. Moreover, given any $\beta$ as
above, the operator mapping $\Xi$ to
$\prod_{\alpha,l}D_{\alpha,b_{\beta,\alpha}(l)}^{n-1-l}$ lies in $G$ (as $N
\circ M^{2k_{\beta}}$ or as $M^{-2k_{\beta}} \circ N$). Let $\widehat{T}_{Q,R}$
be the composition of $T_{Q,R}$ with $N \circ M^{2k_{\beta}}=M^{-2k_{\beta}}
\circ N$. The operators $T_{Q,R}$ and $\widehat{T}_{Q,R}$ are admissible on the
same divisors $\Xi$, but the action of $\widehat{T}_{Q,R}$ is much simpler: It
takes every such $\Xi$ to $R\Xi/Q$. Moreover, since we work with $M$-orbits
rather than divisors, we can phrase the condition for admissibility of $T_{Q,R}$
(or of $\widehat{T}_{Q,R}$) on some divisor $\Xi$ as the requirement that if $Q
\in D_{\beta,j}$ for some $j\in\mathbb{Z}/n\mathbb{Z}$ then the set containing
$R$ is $D_{\gamma,\gamma k_{\beta}(j+1)}$. As the action of $M^{k}$ takes
$D_{\beta,j}$ to $D_{\beta,j-\beta k}$ and $D_{\gamma,\gamma k_{\beta}(j+1)}$ to
$D_{\gamma,\gamma k_{\beta}(j+1)-\gamma k}$, the congruence $\gamma
k_{\beta}(j+1)-\gamma k\equiv\gamma k_{\beta}\big(j-\beta k+1\big)(\mathrm{mod\
}n)$ shows that this requirement is well-defined on $M$-orbits. This condition
actually means that these operators are admissible on the (unique) divisor $\Xi$
in this orbit containing $Q$ to the power $n-1$. The action of
$\widehat{T}_{Q,R}$ on such divisors is obtained using conjugation by the
appropriate power of $M$. Explicitly, it takes $\Xi$ to $R\Xi/Q$, unless
the index $j$ equals $n-1$ and the resulting divisor is $Q^{n-1}\Xi/R^{n-1}$.
The operator $\widehat{T}_{R,Q}$ is applicable precisely on those divisors which
are images of $\widehat{T}_{Q,R}$: Indeed, after this application $Q$ is taken
to $D_{\beta,j+1}$ and $R$ to $D_{\gamma,\gamma k_{\beta}(j+1)-1}$, and if
$l=\gamma k_{\beta}(j+1)-1$ then $\beta k_{\gamma}(l+1)$ is congruent to $j+1$
modulo $n$. The operator $\widehat{T}_{R,Q}$ is now seen to be the inverse of
$\widehat{T}_{Q,R}$.

\smallskip

The final step of the establishment of the Thomae formulae depends on the
following

\begin{conj}
The action of $G$ and the admissible operators $\widehat{T}_{Q,R}$ relate any
two operators $\Xi$ from Equation \eqref{Xi} satisfying the cardinality
conditions. \label{transact}
\end{conj}

We remark that the ordering (4.9) of \cite{[FZ]} was implicitly based on the
fact that given a base point $Q$ and an index $\gamma\in\mathbb{N}_{n}$ which is
prime to $n$, one of each pair of the sets which are mixed by $T_{Q,R}$ is
invariant under $N_{\beta}$. This happens for $\gamma=\beta$ and for
$\gamma=n-\beta$ (the cases appearing in Chapters 4 and 5 and Appendix A of
\cite{[FZ]}), but in no other case. The cases studied in Chapter 6 of that book
considered a small number of divisors with a simple behavior, so that ad-hoc
considerations were sufficient to prove Conjecture \ref{transact} in these
cases. Hence any proof of Conjecture \ref{transact} in general must involve new
considerations, and cannot resemble any of these special cases. In fact, as some
$Z_{n}$ curves do not carry any such divisors $\Xi$, finding an entirely general
argument might be difficult.

We now prove several assertions, which together with the special cases given in
\cite{[FZ]}, support Conjecture \ref{transact}. Fix $\beta\in\mathbb{N}_{n}$
which is prime to $n$, and take $j\in\mathbb{Z}/n\mathbb{Z}$. If $j$ is neither
0 nor $n-1$ modulo $n$, then consider the difference between the cardinality
conditions corresponding to $jk_{\beta}$ and to $(j+1)k_{\beta}$ (neither
elements of $\mathbb{Z}/n\mathbb{Z}$ are 0 by our assumption on $j$, hence they
both yield cardinality conditions). This difference yields a relation of the
form
\begin{equation}
|D_{\beta,j}|=|D_{n-\beta,n-1-j}|+u_{j}, \label{carddif}
\end{equation}
where $u_{j}$ is $t_{(j+1)k_{\beta}}-t_{jk_{\beta}}$ plus the appropriate
difference of the cardinalities of the sets $D_{\alpha,l}$ with $\alpha$
different from $\beta$ and $n-\beta$. For $j=0$ we get Equation \eqref{carddif}
by subtracting $r_{n-\beta}$ from the $k_{\beta}$th cardinality condition (with
$u_{0}$ being $t_{k_{\beta}}$ minus the appropriate cardinalities), and for
$j=n-1$ the $-k_{\beta}$th cardinality condition minus $r_{\beta}$ yields
Equation \eqref{carddif} as well (where $u_{n-1}$ involves $-t_{-k_{\beta}}$
and certain cardinalities). We remark that this is the form in which the
cardinality conditions for the divisors $\Delta$ in Theorems 2.6, 2.9, 2.13,
2.15, and A.1 of \cite{[FZ]} are given. We adopt from \cite{[FZ]} the useful
notation $(D_{\alpha,l})_{\Xi}$ for the sets $D_{\alpha,l}$ appearing in
Equation \eqref{Xi} for the divisor $\Xi$. Observe that in contrast to the
numbers $t_{k}$, the numbers $u_{j}$ may be positive, negative or 0, and they
depend on the divisor $\Xi$ (and on $\beta$). In fact, given a divisor $\Xi$,
the number $u_{j}$ arising from the index $\beta$ is the additive inverse of the
number $u_{n-1-j}$ arising from $n-\beta$. We also remark that an argument
similar to the proof of Proposition \ref{Minv} shows that replacing $\Xi$ by
$M^{ik_{\beta}}(\Xi)$ takes the number $u_{j}$ to $u_{j-i}$.

\begin{prop}
Let $\Xi$ and $\Upsilon$ be two divisors such that
$(D_{\alpha,l})_{\Xi}=(D_{\alpha,l})_{\Upsilon}$ for all $l$ wherever $\alpha$
is neither $\beta$ nor $n-\beta$. Assume that either $(i)$
$D_{\beta,j}\neq\emptyset$ and $D_{n-\beta,j}=\emptyset$ for all $j$, or $(ii)$
there exists some $j$ such that both $D_{\beta,j}$ and $D_{n-\beta,n-1-j}$ are
non-empty. Then the operators $\widehat{T}_{Q,R}$, with $Q$ and $R$ being either
$P_{\beta,i}$ or $P_{n-\beta,i}$, are sufficient in order to reach from $\Xi$ to
$\Upsilon$. \label{difbeta}
\end{prop}

\begin{proof}
Note that as the numbers $u_{j}$ depend on the sets $D_{\alpha,l}$ for $\alpha$
being neither $\beta$ nor $n-\beta$, the first hypothesis implies that they
coincide for $\Xi$ and $\Upsilon$. Since in case $(i)$ we have $r_{n-\beta}=0$
and the sets $D_{n-\beta,j}$ are empty for every divisor, Equation
\eqref{carddif} implies $|(D_{\beta,j})_{\Xi}|=|(D_{\beta,j})_{\Upsilon}|$ for
all $j$. Moreover, Equation \eqref{carddif} shows that
$|D_{\beta,j}|+|D_{n-\beta,n-1-j}|\geq|u_{j}|$, with equality holding if and
only if one of the sets in question is empty. Summing the latter inequality over
$j\in\mathbb{Z}/n\mathbb{Z}$ shows that $\sum_{j=0}^{n-1}|u_{j}| \leq
r_{\beta}+r_{n-\beta}$, and the hypothesis of case $(ii)$ is satisfied precisely
when this inequality is strict (in comparison, summing Equation \eqref{carddif}
over $j\in\mathbb{Z}/n\mathbb{Z}$ yields
$\sum_{j=0}^{n-1}u_{j}=r_{\beta}-r_{n-\beta}$). Hence $\Xi$ and $\Upsilon$
satisfy the assumption of case $(ii)$ simultaneously, and we can indeed use this
assumption without referring to any of the divisors $\Xi$ and $\Upsilon$.

The first observation we make is that for $Q=P_{\beta,i}$ lying in some set
$D_{\beta,j}$ and $R=P_{\beta,m}$, the operator $\widehat{T}_{Q,R}$ can act on
$\Xi$ if and only if $R \in D_{\beta,j+1}$. In this case the operator
interchanges these branch points. We can thus also interchange the point $Q$
with a point $S$ from $D_{\beta,j+2}$, provided that $D_{\beta,j+2}$ is not
empty: Indeed, take $R \in D_{\beta,j+1}$, and the combination
$\widehat{T}_{R,S}\circ\widehat{T}_{Q,S}\circ\widehat{T}_{Q,R}$ is a
composition of operators, each one applicable on the divisor on which it is
supposed to act, which has the desired effect. Easy induction now shows that if
no set $D_{\beta,l}$ is empty then we can interchange any two points
$P_{\beta,i}$ and $P_{\beta,m}$ with one another using these operators,
regardless of the sets in which they lie. As in case $(i)$ the divisor
$\Upsilon$ can be obtained from $\Xi$ by a finite sequence of such
transpositions, this proves the assertion in this case.

On the other hand, if $R=P_{n-\beta,m}$ (and $Q$ is as above) then the
admissibility condition is $R \in D_{n-\beta,n-1-j}$. Applying
$\widehat{T}_{Q,R}$ takes $Q$ to $D_{\beta,j+1}$ and $R$ to $D_{n-\beta,n-j}$.
Hence $\widehat{T}_{Q,R}$ is applicable again, so that we can move $Q$ to
$D_{\beta,j+k}$ and $R$ to $D_{n-\beta,n-1-j+k}$ for any
$k\in\mathbb{Z}/n\mathbb{Z}$. Therefore if $D_{\beta,j}$ and $D_{n-\beta,n-1-j}$
are both not empty then we can take an arbitrary point from each set and move it
to $D_{\beta,l}$ and $D_{n-\beta,n-1-l}$ respectively for any $l$ of our choice.
We now claim that also in this case the operators $\widehat{T}_{Q,R}$ allow us
to interchange any two points $P_{\beta,i}$ and $P_{\beta,m}$ with one another.
Indeed, assume that one point $Q$ lies in $C_{\beta,j}$ and the other point $S$
lies in $C_{\beta,l}$, assume that both $C_{\beta,k}$ and $C_{n-\beta,n-1-k}$
are non-empty, and take $P \in C_{\beta,k}$ and $R \in C_{n-\beta,n-1-k}$.
Using the operations just described, we can take $P$ and $R$ to $D_{\beta,j}$
and $D_{n-\beta,n-1-j}$ respectively, then $Q$ and $R$ to $D_{\beta,l}$ and
$D_{n-\beta,n-1-l}$ respectively, followed by transferring $S$ and $R$ to
$D_{\beta,j}$ and $D_{n-\beta,n-1-j}$ again. Sending $P$ and $R$ back to their
original sets completes a combination of admissible operations which acts as
the asserted interchange. In a similar manner we can replace any point from
$C_{n-\beta,j}$ with any point from $C_{n-\beta,l}$ by admissible operations.
Note that under the assumptions of case $(ii)$, the divisor $\Upsilon$ can be
reached from $\Xi$ by a finite sequence of transfers of points from
$D_{\beta,j}$ and $D_{n-\beta,n-1-j}$ to $D_{\beta,l}$ and $D_{n-\beta,n-1-l}$
followed by permutations of the points $P_{\beta,i}$, $1 \leq i \leq r_{\beta}$
and of the points $P_{n-\beta,i}$, $1 \leq i \leq r_{n-\beta}$. The proof of
the proposition is therefore complete.
\end{proof}

The validity of Proposition \ref{difbeta} extends to divisors $\Xi$ and
$\Upsilon$ for which $D_{\beta,j}$ is empty and $D_{n-\beta,j}$ is not empty for
all $j\in\mathbb{Z}/n\mathbb{Z}$. This is achieved either using symmetry or by a
simple adaptation of the proof.

We note that by taking $\beta=1$ in a $Z_{n}$ curve of the form considered in
Section A.7 of \cite{[FZ]}, Proposition \ref{difbeta} can replace Theorem A.2
of that reference in the proof of the Thomae formulae for these $Z_{n}$ curves.
The same statement holds for the special cases treated in Theorems 4.8 and 5.7
there. Indeed, if either $r_{1}$ or $r_{n-1}$ vanish then every two divisors
satisfy the assumptions of case $(i)$ of Proposition \ref{difbeta}. Otherwise
the assumptions of case $(ii)$ hold for any $\Xi$ and $\Upsilon$. The proof
here is simpler than those given in \cite{[FZ]} because of the freedom to move
the base point. This removes a basic obstacle in the argument, and deals with
the non-singular case in a unified way, regardless of whether $r=1$ or $r\geq2$.

\smallskip

We proceed with the following

\begin{lem}
Let $\Upsilon$ be a divisor satisfying the hypothesis of either case of
Proposition \ref{difbeta}, and let $\Xi$ be a divisor satisfying the usual
cardinality conditions. Let $S$ be a branch point $P_{\gamma,m}$ for $\gamma$
which is neither $\beta$ nor $n-\beta$. Assume that the equality
$v_{P_{\delta,i}}(\Xi)=v_{P_{\delta,i}}(\Upsilon)$ holds for every branch point
$P_{\delta,i} \neq S$ in which $\delta$ equals neither $\beta$ nor $n-\beta$,
and that the difference between $v_{S}(\Xi)$ and $v_{S}(\Upsilon)$ is 1. Then
one can get from $\Xi$ to $\Upsilon$ using the operators $\widehat{T}_{Q,R}$.
\label{muldivR}
\end{lem}

\begin{proof}
Fix the sign of $\pm$ according to $v_{S}(\Upsilon)=v_{S}(\Xi)\pm1$, and choose
the element $j\in\mathbb{Z}/n\mathbb{Z}$ such that the set containing $S$ is
$(D_{\gamma,\gamma k_{\beta}(j+1)})_{\Xi}$ in case $\pm=+$ and is
$(D_{\gamma,\gamma k_{\beta}j-1})_{\Xi}$ if $\pm=-$. By denoting the difference
$|(D_{\beta,i})_{\Upsilon}|-|(D_{n-\beta,n-1-i})_{\Upsilon}|$ by
$\widehat{u}_{i}$ we find that $\widehat{u}_{j}=u_{j}-1$,
$\widehat{u}_{j\pm1}=u_{j\pm1}+1$, and $\widehat{u}_{k}=u_{k}$ for any other
$k\in\mathbb{Z}/n\mathbb{Z}$. Assume first that $\Xi$ satisfies the hypothesis
of case $(ii)$ of Proposition \ref{difbeta}. Then the proof of this proposition
shows that there exist points $Q=P_{\beta,i}$ and $R=P_{n-\beta,p}$ such that an
appropriate power of $\widehat{T}_{Q,R}$ takes $\Xi$ to a divisor $\Sigma$ such
that $(D_{\beta,j})_{\Sigma}$ contains $Q$. Otherwise the cardinalities of the
sets $(D_{\beta,k})_{\Xi}$ (as well as $(D_{n-\beta,k})_{\Xi}$) are determined
by the cardinalities of the other sets corresponding to $\Xi$. Indeed, the
proof of Proposition \ref{difbeta} shows that in this case the equality
$\sum_{j}|u_{j}|=r_{\beta}+r_{n-\beta}$ holds and at least one of $D_{\beta,j}$
and $D_{n-\beta,n-1-j}$ is empty for all $j\in\mathbb{Z}/n\mathbb{Z}$. The
cardinalities are thus determined by Equation \eqref{carddif}. We claim that
$(D_{\beta,j})_{\Xi}\neq\emptyset$ in this case. Indeed, if
$D_{n-\beta,k}=\emptyset$ for all $k$ ($\Upsilon$ satisfying the hypothesis of
case $(i)$ of Proposition \ref{difbeta}) then $\widehat{u}_{j}\geq0$ hence
$u_{j}\geq1$. On the other hand, the proof of Proposition \ref{difbeta} implies
that the only case in which $\Upsilon$ satisfies the hypothesis of case $(ii)$
of that proposition but $\Xi$ does not is when
$\sum_{k}|u_{k}|=r_{\beta}+r_{n-\beta}$ and
$\sum_{k}|\widehat{u}_{k}|<r_{\beta}+r_{n-\beta}$. Under the given relations
between the $u_{k}$s and the $\widehat{u}_{k}$s this can happen only if
$u_{j\pm1}<0<u_{j}$, so that in particular $(D_{\beta,j})_{\Xi}\neq\emptyset$.
Take $Q\in(D_{\beta,j})_{\Xi}$, and choose $\Sigma=\Xi$ in this case. Now, the
location of $Q$ and $S$ implies that the operator $\widehat{T}_{Q,S}$ in case
$\pm=+$ and $\widehat{T}_{S,Q}$ if $\pm=-$ is admissible on $\Sigma$. Moreover,
$\Upsilon$ and the image of $\Sigma$ under this operator satisfy the hypothesis
of Proposition \ref{difbeta}. An application of the aforementioned proposition
now completes the proof.
\end{proof}

Proposition \ref{difbeta} and Lemma \ref{muldivR} suggest another example for
the validity of Conjecture \ref{transact}. Consider a $Z_{n}$ curve $X$ for
which there is some $\beta$ such that one of the following is satisfied: $(i)$
$r_{\beta}$ is much larger than $r_{\alpha}$ for other $\alpha$, or $(ii)$
$r_{\beta}+r_{n-\beta}$ is a sum of positive integers which is relatively large.
Then, Conjecture \ref{transact} holds for $X$, at least heuristically. This is
so, since most divisors will satisfy the assumptions of Proposition
\ref{difbeta}, hence can be related to many ``neighboring'' divisors.

Combining Theorem \ref{NTinv} and Proposition \ref{Minv} yields the main result
of this paper, which is the Thomae formulae for the general fully ramified
$Z_{n}$ curve $X$:

\begin{thm}
Assume that $X$ satisfies Conjecture \ref{transact}. Then the value of the
quotient $\frac{\theta^{2en^{2}}[\Xi](0,\Pi)}{h_{\Xi}}$ from Theorem \ref{NTinv}
is independent of the choice of the divisor $\Xi$. \label{Thomaegen}
\end{thm}

Following \cite{[FZ]} we would like to relate the characteristics to those
arising from the non-special divisors $\Delta$ from Theorem \ref{nonsp} with the
choice of some branch point $Q$ as base point. For every such $Q$ and $\Delta$
we define the denominator $h_{\Delta}^{Q}$ to be $h_{\Xi}$ for
$\Xi=Q^{n-1}\Delta$. Theorem \ref{Thomaegen} then implies that if Conjecture
\ref{transact} holds for the $Z_{n}$ curve $X$ then the quotient
$\frac{\theta^{2en^{2}}[Q,\Delta](0,\Pi)}{h_{\Delta}^{Q}}$ is independent of
the choice of the divisor $\Delta$. Moreover, this quotient yields the same
constant for every base point $Q$. The formulation of Theorem \ref{Thomaegen} as
appears here corresponds to the ``symmetric'' Thomae formulae in \cite{[FZ]}.
The fact that \cite{[K]} finds a constant for every such $Z_{n}$ curve also
suggests that Conjecture \ref{transact} might be true.

\section{The Functions $f_{\beta,\alpha}$ \label{ExpForm}}

In this Section we investigate the functions $f_{\beta,\alpha}$ further, and
obtain explicit expressions to evaluate them in some cases. First observe that
$f_{\beta,\alpha}$ depends only on the number $\alpha k_{\beta}$ modulo $n$
(because $a_{\beta,\alpha}$ and $b_{\beta,\alpha}$ depend only on this number).
Fix $d\in\mathbb{N}_{n}$ which is prime to $n$, and consider the functions
$f_{\beta,\alpha}$ for which $\alpha \equiv d\beta(\mathrm{mod\ }n)$. Theorem
\ref{fba} implies that all these functions coincide to a unique function, which
we denote $f^{(n)}_{d}$. Lemma \ref{fsign} expresses $f^{(n)}_{n-d}$ in terms of
$f^{(n)}_{d}$, so that we can restrict attention to those $d\in\mathbb{N}_{n}$
satisfying $d<\frac{n}{2}$. Moreover, Lemma \ref{invind} relates $f^{(n)}_{d}$
to $f^{(n)}_{k_{d}}$ (recall that $k_{d}\in\mathbb{Z}$ satisfies
$dk_{d}\equiv1(\mathrm{mod\ }n)$, but here we do require
$k_{d}\in\mathbb{N}_{n}$), or in fact shows that in the expression for $h_{\Xi}$
given in Theorem \ref{NTinv} we can always use any one of these functions
instead of the other. Note that for $d=1$ and for $d=n-1$ we have $k_{d}=d$
(regardless of $n$). Lemma \ref{invind} reduces to Equation \eqref{ainv} for
$d=1$ (since $a_{\alpha,\alpha}(l)=n-l-n\delta_{l0}$) and is trivial for
$d=n-1$. Wishing to preserve this symmetry, we choose the constants
$c(\delta,\alpha)$ such that they depend only on the number $d\in\mathbb{N}_{n}$
satisfying $\alpha \equiv d\delta(\mathrm{mod\ }n)$. Denoting the appropriate
constant $c^{(n)}_{d}$, we must have $c^{(n)}_{d}=c^{(n)}_{k_{d}}$ by the
assumption made when we defined $h_{\Xi}$. The normalization
$c(\delta,\alpha)=\max_{l\in\mathbb{N}_{n}}f_{\delta,\alpha}(l)$ satisfies the
first condition (so that $c^{(n)}_{d}=\max_{l\in\mathbb{N}_{n}}f^{(n)}_{d}(l)$),
and Lemma \ref{invind} shows that it satisfies the second condition as well.
These considerations allow us, in some cases, to write the full expression for
$h_{\Xi}$ explicitly using just a few of the functions $f^{(n)}_{d}$ (as was the
case in our example).

\smallskip

Now fix $n$ and $d\in\mathbb{N}_{n}$ which is prime to $n$. We begin our
analysis of the functions $f^{(n)}_{d}$ with

\begin{prop}
Assume that $l\in\mathbb{N}_{n}$ equals $pd+q$ with some $q\in\mathbb{N}_{d}$.
Then $f^{(n)}_{d}(l)$ is related to $f^{(n)}_{d}(q)$ through the formula stating
that $f^{(n)}_{d}(l)$ equals
\[f^{(n)}_{d}(q)+p(n+d-1-q-l)=\frac{l(n+d-1-l)}{d}+f^{(n)}_{d}(q)-\frac{
q(n+d-1-q)}{d}.\] \label{flfq}
\end{prop}

\begin{proof}
Equation \eqref{binv} translates to the equality
$f^{(n)}_{d}(j+d)=f^{(n)}_{d}(j)+n-1-2j$ for all $j\in\mathbb{N}_{n}$ such that
$j+d<n$. By applying this for $j=di+q$ for all $i\in\mathbb{N}_{p}$ (with $j=q$
for $i=0$ and $j+d=l$ for $i=p-1$) we obtain
\[f^{(n)}_{d}(l)=f^{(n)}_{d}(q)+\sum_{i=0}^{p-1}[n-1-2(di+q)]=f^{(n)}_{d}
(q)+p(n-1-2q)-dp(p-1)=\]
\[=f^{(n)}_{d}(q)+p(n+d-1-2q-dp)=f^{(n)}_{d}(q)+\frac{l-q}{d}(n+d-1-q-l)\]
(recall that $l=pd+q$ hence $p=\frac{l-q}{d}$). This gives the first expression,
and the second follows from simple arithmetics. This proves the proposition.
\end{proof}

Proposition \ref{flfq} shows that knowing $f^{(n)}_{d}(q)$ only for
$q\in\mathbb{N}_{d}$ is sufficient for evaluating $f^{(n)}_{d}(l)$ for all
$l\in\mathbb{N}_{n}$. It turns out useful to write $n=sd+t$ for some
$t\in\mathbb{N}_{d}$ in what follows. We now derive a few properties of the
expression $f^{(n)}_{d}(q)-\frac{q(n+d-1-q)}{d}$ with $q\in\mathbb{N}_{d}$
appearing in Proposition \ref{flfq}. It turns out useful to multiply this
expression by $-d$, and call the result $g^{(d)}_{t}(q)$ (this is an abuse of
notation at this point, since we do not yet know that the value of
$g^{(d)}_{t}(q)$ depends only on $t$ and not on $n=sd+t$, but we use it
nontheless). Recall that Equation \eqref{ainv} for $f^{(n)}_{d}$ compares the
value which this function attains on $q\in\mathbb{N}_{d}$ with the value it
takes on $d-1-q$, as well as the image of some $d \leq l\in\mathbb{N}_{n}$ under
this function with the image of $n+d-1-l$. The following Lemma gives a similar
assertion for $g^{(d)}_{t}(q)$:

\begin{lem}
The expression $g^{(d)}_{t}(q)$ is invariant under replacing
$q\in\mathbb{N}_{t}$ by $t-1-q$ and $t \leq q\in\mathbb{N}_{d}$ by $d+t-1-q$.
\label{remainv}
\end{lem}

\begin{proof}
Express $f^{(n)}_{d}(l)$ for $l \geq d$ in terms of the formula from Proposition
\ref{flfq}, and compare it to $f^{(n)}_{d}(n+d-1-l)$. These values coincide by
Equation \eqref{ainv}, as remarked above. The terms involving $l$ (but not the
residue modulo $d$) coincide, so that the remaining term (which is
$-g^{(d)}_{t}(q)/d$ by definition) must also give the same value for $q$ and for
the residue of $n+d-1-l$ modulo $d$. As $n=sd+t$, this residue is $t-1-q$ if
$q\in\mathbb{N}_{t}$ and equals $d+t-1-q$ for $q \geq t$. This proves
the lemma.
\end{proof}

We remark that the proof of Lemma \ref{remainv} implicitly uses the assumption
$s\geq2$ (namely $d<\frac{n}{2}$), since it requires the existence of $d \leq
l\in\mathbb{N}_{n}$ which is congruent to $q$ modulo $d$ for every
$q\in\mathbb{N}_{d}$. However, its assertion is true in general---see the remark
after Theorem \ref{fndrec} below.

\begin{lem}
$(i)$ The equality $f^{(n)}_{d}(q+t)=f^{(n)}_{d}(q)-s(d-t-1-2q)$ holds for every
$q\in\mathbb{N}_{d-t}$ . $(ii)$ For $q\in\mathbb{N}_{d}$ satisfying $q \geq
d-t$ we have the equality $f^{(n)}_{d}(q+t-d)=f^{(n)}_{d}(q)-(s+1)(2d-t-1-2q)$.
\label{rembinv}
\end{lem}

\begin{proof}
Set $l=(s-1)d+q+t=n-d+q\in\mathbb{N}_{n}$ in Equation \eqref{binv}, which now
takes the form $f(l+d-n)=f(l)+n-1-2l$ since $l+d \geq n$. The left hand side is
just $f(q)$, while for the right hand side we use the first expression from
Proposition \ref{flfq}. Recall also that $l=n-d+q$ and $n=sd+t$. In case $(i)$
we have $p=s-1$ and the residue is $q+t$, hence the right hand side becomes
\[f^{(n)}_{d}(q+t)+(s-1)(2d-1-2q-t)-1-n+2d-2q=f^{(n)}_{d}(q+t)+s(d-t-2q-1).\]
This proves part $(i)$. On the other hand, in case $(ii)$ the parameter $p$ is
$s$ and the residue is $q+t-d$. The right hand side thus takes the form
\[f^{(n)}_{d}\!(q+t-d)+s(3d-1-2q-t)-1-n+2d-2q\!=\!f^{(n)}_{d}
(q+t-d)+(s+1)(2d-1-2q-t)\] and part $(ii)$ is also established. This proves the
lemma.
\end{proof}

We now prove that the functions $f^{(n)}_{d}$ can be evaluated using a
recursive process, based on Euclid's algorithm for finding greatest common
divisors using division with residue:

\begin{thm}
If $n=sd+t$ and $l=pd+q$ for $t$ and $q$ in $\mathbb{N}_{d}$ then we can write
$f^{(n)}_{d}(l)$ as $\frac{l(n+d-1-l)}{d}-\frac{n}{d}f^{(d)}_{t}(q)$. The value
of $f^{(n)}_{n-d}(l)$ can be written as
$\frac{n}{d}f^{(d)}_{t}(q)-\frac{l(n-d-1-l)}{d}$.
\label{fndrec}
\end{thm}

\begin{proof}
By Proposition \ref{flfq}, the first assertion boils down to the statement that
$g^{(d)}_{t}(q)=nf^{(d)}_{t}(q)$ for every $q\in\mathbb{N}_{d}$. Lemma
\ref{remainv} shows that $g^{(d)}_{t}$ satisfies Equation \eqref{ainv} for $d$
and $t$. We wish to prove that it satisfies also the appropriate Equation
\eqref{binv}, namely that $g^{(d)}_{t}(q)+n(d-1-2q)$ gives us
$g^{(d)}_{t}(q+t)$ if $q\in\mathbb{N}_{d-t}$ and $g^{(d)}_{t}(q+t-d)$ if $q
\geq d-t$. Recall that $s=\frac{n-t}{d}$. Take $q\in\mathbb{N}_{d-t}$ and apply
part $(i)$ of Lemma \ref{rembinv}. This evaluates
$g^{(d)}_{t}(q+t)=(q+t)(n+d-1-q-t)-df^{(n)}_{d} (q+t)$ as
\[q(n+d-1-q)+t(n+d-1-q-t)-tq-df^{(n)}_{d}(q)+(n-t)(d-t-1-2q),\] which
reduces to the asserted value $g^{(d)}_{t}(q)+n(d-1-2q)$. For $q \geq d-t$ we
use part $(ii)$ of Lemma \ref{rembinv} and write $s+1=\frac{n+d-t}{d}$ in order
to find that the difference between
\[g^{(d)}_{t}(q+t-d)=(q+t-d)(n+2d-1-q-t)-df^{(n)}_{d}(q+t-d)\] and
$g^{(d)}_{t}(q)=q(n+d-1-q)-df^{(n)}_{d}(q)$ is
\[(n+d-t)(2d-t-1-2q)-(d-t)(n+2d-1-q-t)+q(d-t)=n(d-1-2q),\] as desired. Since
$g^{(d)}_{t}(0)$ clearly vanishes, the desired equality
$g^{(d)}_{t}(q)=nf^{(d)}_{t}(q)$ for all $q\in\mathbb{N}_{d}$ follows from
Theorem \ref{fba}, which proves the formula for $f^{(n)}_{d}$. The result for
$f^{(n)}_{n-d}$ now follows from Lemma \ref{fsign}. This proves the theorem.
\end{proof}

The proof of Theorem \ref{fndrec} justifies the notation $g^{(d)}_{t}$ a
fortiori, since it indeed depends only on the residue $t$ of $n$ modulo $d$. We
remark that as the proof of Theorem \ref{fba} requires only vanishing at 0 and
Equation \eqref{binv}, Lemma \ref{remainv} is not necessary for the proof of
Theorem \ref{fndrec}. Hence Theorem \ref{fndrec} holds for every $n$ and $d$
(without the assumption $d<\frac{n}{2}$), and Lemma \ref{remainv} (for all $n$
and $d$) follows as a corollary of its proof. Lemma \ref{remainv} in fact
implies that the formula for $f^{(n)}_{d}(l)$ takes the same form for $l \equiv
q(\mathrm{mod\ }d)$ and for $l$ which is equivalent to $t-1-q$ or to $d+t-1-q$
modulo $d$.

\medskip

For any $y\in\mathbb{Z}$, the expression $l(n+d-1-l)-ny$ can be written as
$(l-y)(n+d-1-y-l)-y(y-d+1)$. Since for $y=f^{(d)}_{t}(q)$ the number $d-1-y$ is
$f^{(d)}_{d-t}(d-1-q)$, we can write
\[f^{(n)}_{d}(l)=\frac{(l-f^{(d)}_{t}(q))(n+f^{(d)}_{d-t}(d-1-q)-l)+f^{(d)}_{t}
(q)f^{(d)}_{d-t}(d-1-q)}{d}.\] The latter formula presents an interesting
symmetry between the formula for $f^{(n)}_{d}(l)$ with $n=sd+t$ and $l=pd+q$ and
the one for $f^{(m)}_{d}(j)$ with $m=rd-t$ and $j=kd-1-q$. In particular, if
$f^{(d)}_{t}(q)=d-1$ then the expression for $f^{(n)}_{d}(l)$ given in Theorem
\ref{fndrec} becomes just $\frac{(l-d+1)(n-l)}{d}$. Thus, Theorem \ref{fndrec}
(or Proposition \ref{flfq}) and Lemma \ref{remainv} combine with Lemma
\ref{fsign} to give

\begin{cor}
If $l$ is divisible by $d$, or if $l$ is congruent to $t-1$ modulo $d$, then
$f^{(n)}_{d}(l)=\frac{l(n+d-1-l)}{d}$ and
$f^{(n)}_{n-d}(l)=-\frac{l(n-d-1-l)}{d}$. If $d$ divides $l+1$, or if
$l \equiv t(\mathrm{mod\ }d)$, then $f^{(n)}_{d}(l)=\frac{(l-d+1)(n-l)}{d}$ and
$f^{(n)}_{n-d}(l)$ equals $\frac{(d-1)n-l(n-d-1-l)}{d}$, or equivalently
$2(d-1)-\frac{(l-d+1)(n-2d-l)}{d}$. \label{q0qd-1}
\end{cor}

Corollary \ref{q0qd-1} is already sufficient to evaluate $f^{(n)}_{d}$ for some
values of $d$. For $d=1$ all the assertions of that Corollary yield the formulae
$f^{(n)}_{1}(l)=l(n-l)$ and $f^{(n)}_{n-1}(l)=-l(n-2-l)$ for all
$l\in\mathbb{N}$. Equation \eqref{ainv} for these cases correspond to the fact
that the value of $f^{(n)}_{1}$ is invariant under sending $l$ to
$n-l-n\delta_{l0}$ and $f^{(n)}_{n-1}$ attains the same value on $l$ and on
$n-2-l+n\delta_{l,n-1}$. Observe that for even $n$ the function $f^{(n)}_{1}$
attains its maximal value $\frac{n^{2}}{4}$ at $l=\frac{n}{2}$, while for odd
$n$ the maximal value is $\frac{n^{2}-1}{4}$, attained on $l=\frac{n-1}{2}$ and
on $l=\frac{n+1}{2}$. Hence the expression for the function $f^{(n)}_{1}$ does
not depend on the parity of $n$, but its normalizing constant $c^{(n)}_{1}$ does
depend on the parity of $n$. This is the reason for the different formulae for
odd and even $n$ given in Chapter 4 of \cite{[FZ]}. $f^{(n)}_{n-1}$ attains its
maximal value $n-1$ at $l=n-1$ (regardless of the parity of $n$), and
$n-1-f^{(n)}_{n-1}(l)$ equals $(l+1)(n-1-l)$ for $l\in\mathbb{N}_{n}$. Note that
the expression $[C_{i},D_{i+k}]$ from Chapter 5 of \cite{[FZ]} becomes
$[D_{1,i+1},D_{n-1,n-2-i-k}]$ in our notation, and the second index of the
latter set is $i+1$ times $\alpha k_{\beta}\equiv-1(\mathrm{mod\ }n)$ plus
$n-1-k$. Since $k(n-k)$ agrees with our $n-1-f^{(n)}_{n-1}(l)$ for $l=n-1-k$, we
verify the results of this Chapter (and of Section A.7 there) as well.

\smallskip

Corollary \ref{q0qd-1} also yields the full formula for the case $d=2$ (for odd
$n$): $f^{(n)}_{2}$ takes even $l\in\mathbb{N}$ to $\frac{l(n+1-l)}{2}$ and odd
$l\in\mathbb{N}$ to $\frac{(l-1)(n-l)}{2}$ (we have indeed seen these
expressions in the powers appearing in our example), while $f^{(n)}_{n-2}(l)$
is $-\frac{l(n-3-l)}{2}$ if $l$ is even and equals $2-\frac{(l-1)(n-4-l)}{2}$ if
$l$ is odd. Equation \eqref{ainv} is satisfied since the involution for $d=2$
takes $l\geq2$ to $n+1-l$ (and interchanges 0 and 1) while the one with $d=n-2$
maps $l \leq n-3$ to $n-3-l$ (and interchanges $n-1$ and $n-2$). The maximal
value of $f^{(n)}_{2}(l)$ is obtained for even $l$, and depends on whether $n$
is equivalent to 1 or to 3 modulo 4: It equals $\frac{n^{2}+2n-3}{8}$ (for $l$
being $\frac{n-1}{2}$ or $\frac{n+3}{2}$) in the former case and
$\frac{n^{2}+2n+1}{8}$ (for $l=\frac{n+1}{2}$) in the latter case. This is the
reason for introducing $\varepsilon$ in the example. $f^{(n)}_{n-2}$ attains its
maximal value $n-1$ at $l=n-2$ and at $l=n-1$ for every odd $n$, and the
difference $n-1-f^{(n)}_{n-2}(l)$ equals $\frac{(l+2)(n-1-l)}{2}$ if $l$ is even
and $\frac{(l+1)(n-2-l)}{2}$ if $l$ is odd. One may compare these values with
the powers appearing in $h_{\Xi}$ in our example, after applying the index
inversion in the sets $D_{l}$ and $F_{l}$ there. Lemma \ref{invind} now shows
that $f^{(n)}_{\frac{n+1}{2}}$ takes $l\in\mathbb{N}_{\frac{n+1}{2}}$ to
$l(n-1-2l)$ and maps $\frac{n+1}{2} \leq l\in\mathbb{N}_{n}$ to $(n-l)(2l+1-n)$,
while the function $f^{(n)}_{\frac{n-1}{2}}$ sends
$l\in\mathbb{N}_{\frac{n+1}{2}}$ to $-l(n-3-2l)$ and takes $\frac{n+1}{2} \leq
l\in\mathbb{N}_{n}$ to $2-(n-2-l)(2l-1-n)$. These expressions also appeared
when we examined the expressions from our example. 

For $d=3$ we encounter the dependence of the form of $f^{(n)}_{d}$ on the
residue $t$ of $n$ modulo $d$. Corollary \ref{q0qd-1} again gives us the full
answer: If $t=1$ then $f^{(n)}_{3}(l)$ is $\frac{l(n+2-l)}{3}$ and
$f^{(n)}_{n-3}(l)$ equals $-\frac{l(n-4-l)}{3}$ if 3 divides $l$, while
otherwise $f^{(n)}_{3}$ and $f^{(n)}_{n-3}$ attain $\frac{(l-2)(n-l)}{3}$ and
$4-\frac{(l-2)(n-6-l)}{3}$ on $l$ respectively. On the other hand, for $t=2$ we
find that $f^{(n)}_{3}(l)=\frac{l(n+2-l)}{3}$ and
$f^{(n)}_{n-3}=-\frac{l(n-4-l)}{3}$ if $l$ is congruent to 0 or 1 modulo 3,
while the values $\frac{(l-2)(n-l)}{3}$ and $4-\frac{(l-2)(n-6-l)}{3}$ are
attained only on $l\equiv2(\mathrm{mod\ }3)$. A careful examination of these
functions show that their maximal values depend on $t$ as well as on the parity
of $n$. The functions for which we can now apply Lemma \ref{invind} depend
themselves on the value of $t$: $k_{3}$ is $s+1$ if $t=2$ and is $2s+1$ if
$t=1$, while $k_{n-3}$ equals $s$ for $t=1$ and is $2s+1$ for $t=2$. We shall
not write the formulae for these functions.

\medskip

To give the flavor of how the functions $f^{(n)}_{d}$ look like for larger
values of $d$, we use our knowledge of the function $f^{(d)}_{1}$ and
$f^{(d)}_{d-1}$ from above in order to extract from Theorem \ref{fndrec} and
Lemma \ref{fsign} the following

\begin{cor}
In the case $t=1$ (which means that $d$ divides $n-1$) $f^{(n)}_{d}(l)$ equals
$\frac{l(n+d-1-l)-nq(d-q)}{d}$, or alternatively
\[\frac{[l-q(d-q)][n-(q-1)(d-1-q)-l]-q(q-1)(d-q)(d-1-q)}{d},\] for every
$l\in\mathbb{N}_{n}$ which is congruent to $q\in\mathbb{N}_{n}$ modulo $d$. The
function $f^{(n)}_{n-d}$ attains on such $l$ the value
$\frac{nq(d-q)-l(n-d-1-l)}{d}$, which can also be written as
\[\frac{q(q+1)(d-q)(d+1-q)-[l-q(d-q)][n-(q+1)(d+1-q)-l]}{d}.\] For $t=d-1$
(namely, if $d$ divides $n+1$) and $l\in\mathbb{N}_{n}$ of the form $pd+q$ we
find that $f^{(n)}_{d}(l)$ is $\frac{l(n+d-1-l)+nq(d-2-q)}{d}$, namely
\[\frac{[l+q(d-2-q)][n+(q+1)(d-1-q)-l]-q(q+1)(d-1-q)(d-2-q)}{d}.\] An element
$l\in\mathbb{N}_{n}$ of this form is taken by the function $f^{(n)}_{n-d}$ to
the number $\frac{-nq(d-2-q)-l(n-d-1-l)}{d}$, which also equals
\[\frac{q(q\!+\!1\!)(\!d\!-\!2\!-\!q)(\!d\!-\!1\!-\!q)\!-\![
l\!+\!q(\!d\!-\!2\!-\!q)][n\!+\!(\!q\!-\!1\!)(\!d\!-\!3\!-\!q)\!-\!4\!-\!l]}{d}
-2q(\!d-2-q)\!.\]
\label{t1td-1}
\end{cor}

Corollary \ref{t1td-1} reproduces the formulae for $f^{(n)}_{2}$,
$f^{(n)}_{n-2}$, $f^{(n)}_{3}$, and $f^{(n)}_{n-3}$. It also gives the formula
for $f^{(n)}_{4}$ and $f^{(n)}_{n-4}$ for all odd $n$ and for $f^{(n)}_{6}$ and
$f^{(n)}_{n-6}$ for every odd $n$ which is not divisible by 3, since for both
$d=4$ and $d=6$ (as well as for $d=3$ considered above) the only two elements of
$\mathbb{N}_{d}$ which are prime to $d$ are 1 and $d-1$. For example, if
$n=4s+1$ then $f^{(n)}_{4}$ takes $l$ which is divisible by 4 to
$\frac{l(n+3-l)}{4}$, other even $l$ to $\frac{(l-4)(n-1-l)}{4}-1$, and odd $l$
to $\frac{(l-3)(n-l)}{4}$. In this case if 4 divides $l$ then
$f^{(n)}_{n-4}(l)=\frac{-l(n-5-l)}{4}$, the $f^{(n)}_{n-4}$-image of other even
$l$ is $9-\frac{(l-4)(n-9-l)}{4}$, and $f^{(n)}_{n-4}$ sends any odd $l$ to
$6-\frac{(l-3)(n-8-l)}{4}$. On the other hand, for $n=4s+3$ the function
$f^{(n)}_{4}$ takes any even $l$ to $\frac{l(n+3-l)}{4}$, $l$ which is congruent
to 3 modulo 4 to $\frac{(l-3)(n-l)}{4}$, and $l$ satisfying
$l\equiv1(\mathrm{mod\ }4)$ to $\frac{(l+1)(n+4-l)}{4}-1$. The function
$f^{(n)}_{n-4}$ here sends even $l$ to $\frac{-l(n-5-l)}{4}$, while $l$ which
satisfies $l\equiv3(\mathrm{mod\ }4)$ is taken to $6-\frac{(l-3)(n-8-l)}{4}$ and
$l$ which is congruent to 1 modulo 4 is sent to $-\frac{(l+1)(n-4-l)}{4}-1$.
For $d=5$ the formulae are based on the various functions $f^{(5)}_{t}$ for $0
\neq t\in\mathbb{N}_{5}$. It is now a simple exercise to determine their values,
and verify the validity of Corollaries \ref{t1td-1} and \ref{q0qd-1}. 

\section{Open Problems \label{Open}}

We close this paper by a brief discussion of a few questions which arise from
our considerations.

First, proving Conjecture \ref{transact}. As mentioned above, \cite{[K]}
establishes Thomae formulae for our general setting, which points to the
direction of the validity of Conjecture \ref{transact}. Therefore, even if
Conjecture \ref{transact} does not hold, the Thomae quotient should give the
same value on every orbit of this action. Therefore, in case some
counter-example to Conjecture \ref{transact} arises, one might look for
additional operations which extend the action of $G$ and the operators
$\widehat{T}_{Q,R}$ and leaving the Thomae quotient invariant.

Second, in the case $n=2$ of hyper-elliptic curves, every point of order $n=2$
in $J(X)$ is the image of some divisor which is supported on the branch points.
Hence our non-special divisors yield all the non-vanishing theta constants on
$X$ associated to such torsion points in $J(X)$. As already remarked in the
introduction to \cite{[FZ]}, this statement does not extend to any case with
$n>2$. Indeed, the set of $n$-torsion points in $J(X)$ is a free module of rank
$2g=(n-1)\big(\sum_{\alpha}r_{\alpha}-2\big)$ over $\mathbb{Z}/n\mathbb{Z}$, and
the points $P_{i,\alpha}$ generate (with the fixed base point $Q$), a subgroup
of rank at most $\sum_{\alpha}r_{\alpha}-2$ (one relation comes from the
vanishing of the base point or from the degree 0 condition, the other one
arising from the vanishing of $div(w)$). In fact, the observation that the only
relations between the divisors $\Xi$ in Section \ref{Trans} are given by powers
of the operator $M$ (where a relation means that two divisors represent the same
characteristic) suggests that these are the only relations holding between the
points $\varphi_{Q}(P_{\alpha,i})$ in $J(X)$. Hence the rank is precisely
$\sum_{\alpha}r_{\alpha}-2$. Therefore it is interesting to ask what kind of
divisors represent the other $n$-torsion points in $J(X)$, and whether they are
special or not. For the latter characteristics, we ask whether denominators like
our $h_{\Xi}$ (or $h_{\Delta}^{Q}$), which extends the Thomae formulae to these
characteristics as well, might exist.

Next, one may wish to consider the dependence of the Thomae formulae on the
actual choice of the function $z$ (rather than just a projection from $X$ to the
quotient under the Galois action). Since the denominators $h_{\Xi}$ are based
only on differences between $z$-values, replacing $z$ by $z-\zeta$ for some
$\zeta\in\mathbb{C}$ leaves the Thomae constant invariant. On the other hand,
multiplying $z$ by some number $u\in\mathbb{C}^{*}$ multiplies $h_{\Xi}$ by $u$
raised to the power $\deg h_{\Xi}$, which changes the constant. This shows, by
the way, that all the denominators $h_{\Xi}$ must have the same degree. One
might search for Thomae constants which are independent of the choice of $z$. 
For example, multiplying all the constants by some global product of
differences, which would render the denominators $h_{\Xi}$ rational functions,
should yield an expression which is independent of dilations of $z$ as well. If
this is indeed possible, one needs only to consider replacing $z$ by
$\frac{1}{z}$ and allowing branching at $\infty$. In case an expression which is
independent of this latter transformation exists as well (or more succinctly,
an expression which becomes invariant under the action of $PSL_{2}(\mathbb{C})$
on $P^{1}(\mathbb{C})$ via fractional linear transformations), our Thomae
constants would depend only on the cyclic group of automorphisms of the $Z_{n}$
curve, rather than the actual choice of the map $z$. In addition, assuming that
such an invariant constant exists, we observe that some Riemann surfaces may be
given several structures of $Z_{n}$ curves. For example, the $Z_{n}$ curves from
Section 3.3 of \cite{[FZ]} are all hyper-elliptic, hence carry an additional
structure of a $Z_{2}$ curve. In this case it is natural to ask whether
connections between the Thomae constants arising from the different $Z_{n}$
structures on $X$ can be found.

It may also be of interest to study the dependence on $n$ of the number of
non-special divisors of degree $g$ or of $M$-orbits in families of $Z_{n}$
curves with related equations. More precisely, consider the $Z_{n}$ curve
associated with an equation of the sort
\begin{equation}
w^{n}=\prod_{i=1}^{p}(z-\lambda_{i})^{c_{i}}\prod_{i=1}^{q}(z-\mu_{i})^{n-d_{i
}}, \label{family}
\end{equation}
where $\sum_{i=1}^{p}c_{i}=\sum_{i=1}^{q}d_{i}$ (hence the sum of powers is
divisible by $n$), $q \leq p$ (otherwise replace $w$ by $w_{n-1}$), and $n$ is
large enough (i.e., $n\to\infty$). For sufficiently large $n$ this yields
$t_{k}=q$ for small $k$ and $t_{k}=p$ for large $k$ (here $k$ is considered to
be in $\mathbb{N}_{n}$ and not in $\mathbb{Z}/n\mathbb{Z}$). Note that the
non-singular $Z_{n}$ curves do not describe such a family, since the
corresponding number of points depends on $n$. On the other hand, the singular
curves from Chapter 5 of \cite{[FZ]} do lie in such families for every choice of
the parameter $m$. We have seen in the example considered in Section 6.2 of
\cite{[FZ]} that these numbers are constant (18 divisors for $n\geq7$ and 6
$M$-orbits for $n\geq4$). In relation with the families of $Z_{n}$ curves
considered in \cite{[GDT]}, we consider, for an odd number $n$ which is not
divisible by 3, the $Z_{n}$ curve defined by the equation
\[w^{n}=(z-\lambda)(z-\sigma)^{2}(z-\tau)^{n-3}.\] On can show, by a direct
inspection, that there are no $M$-orbits for $n\geq11$ and no no
non-special divisors of degree $g=\frac{n-1}{2}$ for $n\geq17$. The two
examples with $n=11$ and $n=13$ present cases where non-special divisors of
degree $g$ exist but there are no $M$-orbits, meaning that the non-special
divisors must be supported on all the branch points. This also shows that for
$r=1$ the bound $p>12r$ of \cite{[GDT]} is not sufficient (see also another
example discussed briefly in Section \ref{Open}). Interestingly, for $n=7$ there
exists precisely one $M$-orbit (or characteristic), and an additional divisor
which is supported on all the branch points. Hence the Thomae formulae are
trivial also in this case. For $n=5$ we have 4 non-special divisors, yielding 2
$M$-orbits which are related by $N$. Considerations similar to that example show
that $w^{n}=(z-\lambda)(z-\sigma)^{3}(z-\tau)^{n-4}$ displays a similar
behavior. The curves of the form
$w^{n}=(z-\lambda_{1})(z-\lambda_{2})(z-\tau)^{n-2}$ have 4 non-special divisors
(for $n\geq5$) and 2 $M$-orbits which are related via $N$.
Note that in all these cases we have $q=1$, namely $t_{k}=1$ for small $k$. On
the other hand, the singular curves from Section 3.3 of \cite{[FZ]} (with
$p=q=2$ and the indices $d_{i}$ and $e_{i}$ being 1) admit
$2n-1$ divisors not containing $P_{0}$ in their support (hence $2n-1$ orbits of
$M$). One easily sees that the total number of divisors in this case is $4n-4$.
In the family of $Z_{n}$ curves considered in Section 6.1 of \cite{[FZ]} we
have, for $n\geq5$, $n+2$ characteristics and $2n+5$ divisors (see Theorem 6.3
of \cite{[FZ]}). Similar considerations show that the $Z_{n}$ curves of the form
$w^{n}=(z-\lambda)(z-\sigma)^{3}(z-\tau)^{n-3}(z-\mu)^{n-1}$ carry $n+2$
characteristics for $n=3s+t\geq5$ and $2n+4$ divisors in total for $n\geq8$. In
all these cases we have $q=2$, and the numbers of divisors and $M$-orbits are
both linear in $n$. As for an example with $q=3$, we state that for a singular
curve of the form
\[w^{n}=(z-\lambda_{1})(z-\lambda_{2})(z-\lambda_{3})(z-\mu_{1})^{n-1}(z-\mu_{2}
)^{n-1}(z-\mu_{3})^{n-1}\] (the $Z_{n}$ curves appearing in Chapter 5 of
\cite{[FZ]} with $m=3$) the total number of such divisors of degree $g$ is
$18n^{2}-45n+33$. Of these divisors, $3n^{2}+6n-14$ do not contain a pre-fixed
branch point and correspond to orbits of $M$ (the latter number indeed becomes
$10=\frac{5!}{3!2!}$ for the non-singular $Z_{n}$ curve arising from $n=2$, and
is equal to the number 31 from Section 3.2 of \cite{[FZ]} for $n=3$). These
observations lead us to formulate the following

\begin{conj}
The number of non-special divisors of degree $g$ on the $Z_{n}$ curve associated
to Equation \eqref{family} is described, for large enough $n$, by a polynomial
of degree $q-1$ in $n$. The same assertion holds for the number of $M$-orbits on
such a $Z_{n}$ curve. \label{polgrth}
\end{conj}

In fact, taking a closer look into the results of these examples leads to a
finer form of Conjecture \ref{polgrth}:

\begin{conj}
Let $X$ be a $Z_{n}$ curve described by Equation \eqref{family}, and let $c$
and $d$ be positive integers. Define $x_{c}=|\{i|c_{i}=c\}|$ and
$y_{d}=|\{i|d_{i}=d\}|$. Then $p=\sum_{c}x_{c}$ and $q=\sum_{d}y_{d}$ form
partitions of $p$ and $q$ respectively. The leading coefficients of the two
polynomials appearing in Conjecture \ref{polgrth} depend only on these
partitions of $p$ and $q$. \label{leadcoeff}
\end{conj}

Both Conjectures \ref{polgrth} and \ref{leadcoeff} can be formulated in terms of
assertions about the numbers of solutions of combinatorial equations (Theorem
\ref{nonsp} again), where each solution is assigned a multiplicity according to
the number of divisors it represents (another combinatorial expression).

\noindent\textsc{Fachbereich Mathematik, AG 5, Technische Universit\"{a}t
Darmstadt, Schlossgartenstrasse 7, D-64289, Darmstadt, Germany}

\noindent E-mail address: zemel@mathematik.tu-darmstadt.de

\end{document}